\documentclass[letter,final,oneside,notitlepage,10pt]{article}

\usepackage{amsthm}
\usepackage{graphicx}
\usepackage[margin=2cm,font=normalsize,labelfont=bf]{caption}
\usepackage{verbatim}
\usepackage{enumerate}
\usepackage{fancyhdr}
\usepackage[]{geometry}
\usepackage[normalem]{ulem}
\usepackage{xstring}
\usepackage{needspace}
\usepackage{color}
\usepackage{amstext}

\makeatletter

\newcounter{denseversion}
\newcounter{authorcounter}
\newcounter{adresscounter}

\def\title#1{\gdef\@title{#1}}
\def\@title{}
\def\subtitle#1{\gdef\@subtitle{#1}}
\def\@subtitle{}

\def\authortagsused{0}
\def\adresstag#1{\if!#1!\else$^{\;#1\;}$\fi}

\renewcommand{\author}[2][]{
  \stepcounter{authorcounter}
  \if!#1!\else\gdef\authortagsused{1}\fi
  \ifnum\value{authorcounter}=1
    \def\@authorstringa{#2\adresstag{#1}}
    \def\@authorstringb{#2}
    \def\@authorstringc{#2\adresstag{#1}}
  \else
    \g@addto@macro\@authorstringa{\ and #2\adresstag{#1}}
    \g@addto@macro\@authorstringb{\ and #2}
    \g@addto@macro\@authorstringc{\\#2\adresstag{#1}}
  \fi}
\def\@author{\ifnum\value{denseversion}=0\@authorstringa\else\@authorstringb\fi}

\def\@adressstringa{}
\def\@adressstringb{}
\newcommand{\adress}[2][]{
  \stepcounter{adresscounter}
  \ifnum\value{adresscounter}=1
    \g@addto@macro\@adressstringa{\ifnum\authortagsused=0\def\br{\\}\else\def\br{, }\fi\adresstag{#1}#2}
    \g@addto@macro\@adressstringb{\def\br{\\}\adresstag{#1}\parbox[t]{14cm}{#2}}
  \else
    \g@addto@macro\@adressstringa{\\[\bigskipamount]\adresstag{#1}#2}
    \g@addto@macro\@adressstringb{\\[\medskipamount]\adresstag{#1}\parbox[t]{14cm}{#2}}
  \fi}

\def\preprint#1{\gdef\@preprint{#1}}
\def\@preprint{}
\def\keywords#1{\gdef\@keywords{#1}}
\def\@keywords{}
\def\msc#1{\gdef\@msc{#1}}
\def\@msc{}
\def\email#1{
   \gdef\@email{#1}
   \g@addto@macro\@authorstringc{ {\it (#1)}}}
\def\@email{}
\def\dedication#1{\gdef\@dedication{#1}}
\def\@dedication{}

\def\mybaselinestretch#1{\gdef\@mybaselinestretch{#1}}
\def\@mybaselinestretch{}

\def\refname{References}

\mybaselinestretch{1.2}
\renewcommand{\baselinestretch}{\@mybaselinestretch}

\def\denseversion{
  \setcounter{denseversion}{1}
  \newgeometry{left=3cm,right=3cm,top=3cm}
  \mybaselinestretch{1.1}
  \renewcommand{\baselinestretch}{\@mybaselinestretch}
  \normalfont
  \def\possiblelinebreak{}
  \fancyfoot[C]{\itshape{--$\,\,$\thepage$\,\,$--}}}

\newlength{\myparskip}
\newlength{\myproofparskip}

\def\possiblelinebreak{\\}

\renewcommand{\emph}[1]{\def\reserved@a{it}\ifx\f@shape\reserved@a\uline{#1}\else\textit{#1}\fi}

\newcommand{\mytableofcontents}{
   \ifnum\value{denseversion}=0
     \renewcommand{\baselinestretch}{1}
     \normalfont
     \tableofcontents
     \renewcommand{\baselinestretch}{\@mybaselinestretch}
     \normalfont
   \else
     \renewcommand{\baselinestretch}{0.5}
     \normalfont
     \tableofcontents
     \renewcommand{\baselinestretch}{\@mybaselinestretch}
     \normalfont
   \fi}

\newlength{\zeilenlaenge}
\def\putindent#1{
  \settowidth{\zeilenlaenge}{#1}
  \ifnum\zeilenlaenge>\textwidth
    #1
  \else
    \noindent #1
  \fi
}

\def\href#1#2{#2}

\def\kohyp{
  \usepackage{hyperref}
  \hypersetup{
    linktocpage = true,
    pdftitle = {\@title},
    pdfauthor = {\@author},
    pdfkeywords = {\@keywords},    
    bookmarksopen = true,
    bookmarksopenlevel = 1
  }}  
\def\showkeywords{\begin{flushleft}\footnotesize\textbf{Keywords}: \@keywords\end{flushleft}}
\def\showmsc{\begin{flushleft}\footnotesize\textbf{MSC 2010}: \@msc\end{flushleft}}

\newcounter{mythm}[subsection]
\newcounter{mainthm}

\def\setsecnumdepth#1{
  \setcounter{secnumdepth}{#1}
  \setcounter{mythm}{0}
  \ifnum \c@secnumdepth >0
    \ifnum \c@secnumdepth >1
      \def\themythm{\thesubsection.\arabic{mythm}}
      \numberwithin{equation}{subsection}
      \renewcommand\theequation{\thesubsection.\arabic{equation}}
    \else
      \def\themythm{\thesection.\arabic{mythm}}
      \numberwithin{equation}{section}
      \renewcommand\theequation{\thesection.\arabic{equation}}
    \fi
  \else
    \def\themythm{\arabic{mythm}}
  \fi}

\newenvironment{mythmenv}{\strut\ \setlength{\parskip}{\myproofparskip}}{\setlength{\parskip}{\myparskip}}

\newlength{\mythmskip}
\newlength{\mythmtopskip}
\newtheoremstyle{mythmstylea}{\mythmtopskip}{\mythmskip}{\it}{}{\bf}{.}{0em}{}
\newtheoremstyle{mythmstyleb}{\mythmtopskip}{\mythmskip}{}{}{\bf}{.}{0em}{}

\theoremstyle{mythmstylea}
\newtheorem{mytheorem}[mythm]{\nameTheorem}
\newtheorem{mydefinition}[mythm]{\nameDefinition}
\newtheorem{mycorollary}[mythm]{\nameCorollary}
\newtheorem{myproposition}[mythm]{\nameProposition}
\newtheorem{mylemma}[mythm]{\nameLemma}

\newtheorem{mymaintheorem}[mainthm]{\nameTheorem}
\newtheorem{mymaincorollary}[mainthm]{\nameCorollary}
\newtheorem{mymainproposition}[mainthm]{\nameProposition}
\newtheorem{mymaindefinition}[mainthm]{\nameDefinition}

\theoremstyle{mythmstyleb}

\newtheorem{myremark}[mythm]{\nameRemark}
\newtheorem{myproblem}[mythm]{\nameProblem}
\newtheorem{myexample}[mythm]{\nameExample}
\newtheorem{myexercise}[mythm]{\nameExercise}

\newenvironment{theorem}[1][]{\begin{mytheorem}[#1]\begin{mythmenv}}{\end{mythmenv}\end{mytheorem}}
\newenvironment{definition}[1][]{\begin{mydefinition}[#1]\begin{mythmenv}}{\end{mythmenv}\end{mydefinition}}
\newenvironment{corollary}[1][]{\begin{mycorollary}[#1]\begin{mythmenv}}{\end{mythmenv}\end{mycorollary}}
\newenvironment{proposition}[1][]{\begin{myproposition}[#1]\begin{mythmenv}}{\end{mythmenv}\end{myproposition}}
\newenvironment{lemma}[1][]{\begin{mylemma}[#1]\begin{mythmenv}}{\end{mythmenv}\end{mylemma}}
\newenvironment{remark}[1][]{\begin{myremark}[#1]\begin{mythmenv}}{\end{mythmenv}\end{myremark}}

\newenvironment{example}[1][]{\begin{myexample}[#1]\begin{mythmenv}}{\end{mythmenv}\end{myexample}}

\newenvironment{maintheorem}[1]{\def\themainthm{#1}\begin{mymaintheorem}\begin{mythmenv}}{\end{mythmenv}\end{mymaintheorem}}

\renewenvironment{proof}[1][\nameProof]{\noindent #1. \begin{mythmenv}}{\hphantom{$\square$}\hfill$\square$\end{mythmenv}\medskip}

\def\tocsection#1{\section*{#1}\addcontentsline{toc}{section}{#1}}

\def\mytitle{}
\def\zmptitle{
  \begin{tabular}{cc}
    \begin{minipage}[c]{0.4\textwidth}
      \begin{flushleft}
        \includegraphics[width=110pt]{../../tex/zmp}
      \end{flushleft}  
    \end{minipage}&
    \begin{minipage}[c]{0.55\textwidth}
      \begin{flushright}
      {\small\sf\@preprint}
      \end{flushright}
    \end{minipage}
  \end{tabular}
  \vskip 2cm}

\def\maketitle{
  \setlength{\parskip}{\myparskip}  
  \newpage
  \noindent
  \mytitle
  \begin{center}
    \LARGE\@title\\
    \if!\@subtitle!\else\smallskip\LARGE\@subtitle\\\fi
    \bigskip
    \if!\@author!\else\bigskip\large\@author\\\fi
    \ifnum\value{denseversion}=0
      \if!\@adressstringa!\else\bigskip\normalsize\@adressstringa\\\fi
      \if!\@email!\else\ifnum\value{authorcounter}=1\bigskip\normalsize\textit{\@email}\\\else\fi\fi
    \else
    \fi
    \if!\@dedication!\else\bigskip\normalsize{\@dedication}\\\fi
  \end{center}
  \ifnum\value{denseversion}=0\vskip 1.5cm\else\vskip0.5cm\fi
  \if!\@draft!\else\thispagestyle{empty}\fi}

\def\kobiblink#1{
  \StrSubstitute{#1}{\~{}}{\string~}[\myurl]
  \StrSubstitute{#1}{_}{\underline{\;\;}}[\mylink]
  \StrSubstitute{\mylink}{&}{\&}[\mylink]
  \StrSubstitute{\mylink}{/}{/\allowbreak}[\mylink]
  \newline Available as: \mbox{\;}
  \href{\myurl}{\texttt{\mylink}}}

\def\nolinks{\def\kobiblink##1{}}

\def\kobib#1{
  \addcontentsline{toc}{section}{\refname}
  \begin{raggedright}
  \ifnum\value{denseversion}=0\else\small\fi
  \Oldbibliography{#1/kobib}
  \bibliographystyle{#1/kobib}
  \end{raggedright}
  \ifnum\value{denseversion}=0\else
      \noindent
      \if!\@authorstringc!\else
        \ifnum\authortagsused=0\ifnum\value{authorcounter}>1\normalsize\@authorstringc\\[\medskipamount]\else\fi\else\normalsize\@authorstringc\\[\medskipamount]\fi
      \fi
      \if!\@adressstringb!\else\normalsize\@adressstringb\\{}\fi
      \ifnum\authortagsused=0
        \ifnum\value{authorcounter}=1
          \if!\@email!\else\linebreak\normalsize\textit{\@email}\\{}\fi
        \else
        \fi
      \else
      \fi
  \fi
  }

\let\Oldbibliography\bibliography
\def\bibliography#1{
  \begin{raggedright}
  \ifnum\value{denseversion}=0\else\small\fi
  \Oldbibliography{#1}
  \end{raggedright}
  \ifnum\value{denseversion}=0\else
      \noindent
      \if!\@authorstringc!\else
        \ifnum\authortagsused=0\ifnum\value{authorcounter}>1\normalsize\@authorstringc\\[\medskipamount]\else\fi\else\normalsize\@authorstringc\\[\medskipamount]\fi
      \fi
      \if!\@adressstringb!\else\normalsize\@adressstringb\\{}\fi
      \ifnum\authortagsused=0
        \ifnum\value{authorcounter}=1
          \if!\@email!\else\linebreak\normalsize\textit{\@email}\\{}\fi
        \else
        \fi
      \else
      \fi
  \fi
}

\newenvironment{commentfigure}{\begin{comment}}{\end{comment}}
\newenvironment{sidewayscommentfigure}{\begin{minipage}}{\end{minipage}}

\def\draft#1#2#3#4{
  \ifnum#4=0
    \def\showcomments{ - Comments are not displayed}
  \else
    \renewenvironment{comment}{\begin{list}{}{\rightmargin=1cm\leftmargin=1cm}\item\sf\begin{small}}{\end{small}\end{list}}

    \def\showcomments{ - Comments are displayed}
  \fi
  \gdef\@draft{DRAFT - Version #1 - Last edited on #2 - Last edited by #3\showcomments}
  \fancyhead[C]{\footnotesize\tt\textcolor{red}{\@draft}}}
\def\@draft{}

\makeatother

\usepackage{amssymb}
\usepackage{amsmath}
\usepackage{amstext}
\usepackage[]{geometry}
\usepackage{mathrsfs}
\usepackage{euscript}
\usepackage[all]{xy}
\usepackage{xstring}
\usepackage{slashed}
\usepackage{mathtools}

\def\N {\mathbb{N}}
\def\Z {\mathbb{Z}}

\def\R {\mathbb{R}}

\def\im{\mathrm{i}}
\def\id{\mathrm{id}}

\def\hc#1{\mathrm{h}_{#1}}
\def\h {\mathrm{H}}

\def\subset{\subseteq}

\def\sep{\;|\;}
\def\maps{\colon}
\def\df{:=}
\def\eq{=}
\renewcommand{\varepsilon}{\epsilon}
\renewcommand{\to}{\!\xymatrix@R=0cm@C=1.4em{\ar[r] &}}
\renewcommand{\mapsto}{\!\xymatrix@R=0cm@C=1.4em{\ar@{|->}[r] &}\!}
\renewcommand{\Rightarrow}{\!\xymatrix@R=0cm@C=1.4em{\ar@{=>}[r] &}\!}
\renewcommand{\Leftarrow}{\!\xymatrix@R=0cm@C=1.4em{\ar@{<=}[r] &}\!}
\newcommand{\incl}{\!\xymatrix@R=0cm@C=1.4em{\ar@{^(->}[r] &}\!}

\renewcommand\Leftrightarrow{\!\xymatrix@R=0cm@C=1.4em{\ar@{<=>}[r] &}\!}

\makeatletter
\renewenvironment{proof}[1][\nameProof]
  {\par\pushQED{\qed}%
   \normalfont \topsep6\p@\@plus6\p@\relax
   \trivlist
   \item[\hskip\labelsep
         \itshape
         #1\@addpunct{.}]
   \begin{mythmenv}{}{}{}}
  {\end{mythmenv}\popQED\endtrivlist\@endpefalse}
\makeatother

\def\notebox#1#2{\begin{minipage}[b]{#1}\sloppy\renewcommand{\baselinestretch}{0.8}\footnotesize \begin{center}#2\end{center}\end{minipage}}

\def\erf#1{(\ref{#1})}

\def\stackref#1#2{\stackrel{\text{\ref{#1}}}{#2}}
\def\eqref#1{\stackref{#1}{=}}

\newlength{\myeqt} 
\newlength{\myeqs} 
\newlength{\myeqm} 
\newlength{\myeqn} 
\newcommand\symtext[3][\myeqn]{
  \settowidth{\myeqt}{#2}
  \settowidth{\myeqs}{$#3$}
  \addtolength{\myeqs}{\the\myeqm}
  \ifdim\myeqt>\myeqs
    \stackrel{\hspace{-#1}\notebox{#1}{\medskip #2 \\ $\downarrow$\smallskip}\hspace{-#1}}{#3}
  \else
    \stackrel{\text{#2}}{#3}
  \fi}

\def\brackets#1{\IfStrEq{#1}{-}{}{(#1)}}
\def\subindex#1{\IfStrEq{#1}{-}{}{_{#1}}}

\newcommand{\alxydim}[2]{\begin{aligned}\xymatrix#1{#2}\end{aligned}}

\makeatletter
\def\bigset#1#2{\left\lbrace\;\begin{minipage}[c]{#1}\begin{center}#2\end{center}\end{minipage}\;\right\rbrace}

\makeatother

\newlength{\myl}

\def\ddt#1#2#3{\left.\frac{\mathrm{d}^{\IfStrEq{#1}{1}{}{#1}}}{\mathrm{d}#2}\IfStrEq{#2}{#3}{\right.}{\right|_{#3}}}

\newcommand{\ueins}{{\mathrm{U}}(1)}

\def\diff{\mathcal{D}\!i\!f\!\!f}

\def\ev{\mathrm{ev}}

\def\pr{{\mathrm{pr}}}

\newlength{\widthtmp}
\def\length#1{\settowidth{\widthtmp}{#1}\the\widthtmp}

\def\lli#1{\prescript{}{#1}}

\def\buntech#1#2{\mathcal{B}\hspace{-0.01em}un_{\hspace{0.05em}#1}^{#2}}
\def\trivlin{\mathbf{I}}

\def\bun#1#2{\buntech{#1}{}\brackets{#2}}

\def\ubun#1{\bun\relax{#1}}

\def\ufusbun#1{\mathcal{F}\!us\buntech{}{}(#1)}

\def\ufusbunconsf#1{\mathcal{F}\!us\bunconsf{}{#1}}

\def\bunconsf#1#2{\buntech{#1}{\!\nabla}{}^{_{\!\!s\!f}}\hspace{-0.15em}\brackets{#2}}

\def\grbtech#1{\mathcal{G}\hspace{-0.06em}r\hspace{-0.06em}b_{\hspace{-0.07em}{#1}}}
\def\grb#1#2{\grbtech{#1}\brackets{#2}}
\def\grbcon#1#2{\grbtech{#1}^{\nabla\!}\brackets{#2}}

\def\ugrb#1{\grb{\,}{#1}}
\def\ugrbcon#1{\grbcon\relax{#1}}

\def\mult{\mathcal{M}\hspace{-0.14em}u\hspace{-0.04em}l\hspace{-0.1em}t}
\def\multgrb#1{\mult\hspace{-0.07em}\ugrb#1}
\def\multgrbadcon#1{\mult\hspace{-0.07em}\grbtech{\,}^{\infty}\brackets{#1}}
\def\multgrbcon#1{\mult\hspace{-0.07em}\ugrbcon#1}

\usepackage[latin1]{inputenc}
\usepackage[english]{babel}

\def\refname{References}

\def\quot#1{``#1''}

\def\quand{\quad\text{ and }\quad}
\def\quomma{\quad\text{, }\quad}

\def\quith{\quad\text{ with }\quad}

\def\nameTheorem{Theorem}
\def\nameDefinition{Definition}
\def\nameCorollary{Corollary}
\def\nameProposition{Proposition}
\def\nameLemma{Lemma}
\def\nameRemark{Remark}
\def\nameProblem{Problem}
\def\nameExample{Example}
\def\nameExercise{Exercise}
\def\nameProof{Proof}

\def\can{can}

\def\thinpairs#1{#1^2_{\text{\tiny \it thin}}}

\def\px#1#2{P_{\!#2}#1}
\def\p{P}

\def\ev{\mathrm{ev}}

\def\hc#1{\mathrm{h}_{#1}}
\def\pcomp{\star}

\def\prev#1{\overline{#1}}

\def\un{\mathscr{R}}
\def\uncon{\mathscr{R}^{\nabla}}
\def\tr{\mathscr{T}}
\def\trcon{\mathscr{T}^{\nabla}}

\def\fusextcon#1#2#3{\mathcal{F}\!us\!\,\extconsf{#1}{#2}{#3}}

\def\ufusbun#1{\mathcal{F}\!us\buntech{}{}(#1)}

\def\ufusbunconsf#1{\mathcal{F}\!us\bunconsf{}{#1}}

\def\bunconsf#1#2{\buntech{#1}{\!\nabla}{}^{_{\!\!s\!f}}\hspace{-0.15em}\brackets{#2}}

\title{Transgressive loop group extensions}
\author{Konrad Waldorf}
\email{konrad.waldorf@uni-greifswald.de}

\adress{Ernst-Moritz-Arndt-Universität Greifswald\\
Institut für Mathematik und Informatik\\
Walther-Rathenau-Str. 47\\
D-17487 Greifswald}

\msc{Primary 58B25, Secondary 53C08, 22E67}

\kohyp
\usepackage{amstext}
\usepackage{amsmath}

\def\rebox#1{\reflectbox{#1}}

\def\ext#1{\mathcal{E}\hspace{-0.2em}x\hspace{-0.05em}t(#1)}
\def\fusextth#1{\mathcal{F}\hspace{-0.22em}u\hspace{-0.15em}s\mathcal{E}\hspace{-0.2em}x\hspace{-0.05em}t^{th}(#1)}
\def\fusextcon#1{\mathcal{F}\hspace{-0.22em}u\hspace{-0.15em}s\mathcal{E}\hspace{-0.2em}x\hspace{-0.05em}t^{\nabla}(#1)}
\def\gbas{\mathcal{G}_{bas}}

\def\inf#1{\EuScript{#1}}
\def\can{1}

\def\wzwmodel{\mathcal{L}_{G}}

\def\multtr{\mathscr{M}\hspace{-0.9ex}\tr}
\def\multtrcon{\mathscr{M}\hspace{-0.9ex}\trcon}
\def\multun{\mathscr{M}\hspace{-0.7ex}\un}
\def\multuncon{\mathscr{M}\hspace{-0.7ex}\uncon}

\def\xyst{4em}

\usepackage{amssymb}
\usepackage{graphicx}

\begin{document}


\denseversion

\maketitle

\begin{abstract}
A central extension of the loop group of a Lie group is called transgressive, if it corresponds under transgression to  a degree four class in the cohomology of the classifying space of the Lie group. Transgressive loop group extensions are those that can be explored by finite-dimensional, higher-categorical geometry over the Lie group.  We show how transgressive central extensions can be characterized in a  loop-group theoretical way, in terms of loop fusion  and thin homotopy equivariance. 
\end{abstract}


\nolinks

\def\lop{\cup}

\mytableofcontents

\vspace{-0.1em}

\section{Introduction}

The present article is about a  Lie group $G$, its loop group $LG := C^{\infty}(S^1,G)$ and  central extensions \begin{equation*}
1 \to \ueins \to \mathcal{L} \to LG \to 1
\end{equation*}
in the category of Fréchet Lie groups. Some central extensions $\mathcal{L}$ can be obtained from structure over $G$ called \emph{multiplicative  bundle gerbe with connection} via a procedure called \emph{transgression}. These central extensions are called  \emph{transgressive}. In the case of  $G$  compact, transgression induces a map
\begin{equation*}
\h^4(BG,\Z) \to \bigset{11em}{Isomorphism classes of 
central extensions of $LG$}\text{,}
\end{equation*}
and a central extension is transgressive if and only if it is in the image of that map. An example of a transgressive central extension is the universal central extension of the loop group of a compact  simply-connected Lie group $G$: it corresponds to a generator of $\h^4(BG,\Z)\cong \Z$.

The goal of this article is to characterize transgressive central extensions for arbitrary connected Lie groups in purely loop group-theoretic terms. For this purpose we consider two  relations on  $LG$:
\begin{enumerate}
\item 
thin homotopy: a homotopy between two loops is called thin, if its differential has nowhere  full rank; these are homotopies that sweep out a surface of zero area. 

\item
loop fusion: it relates two loops that share a common line segment to a new loop with that segment deleted. 
\end{enumerate}
We introduce the notion of a \emph{thin fusion extension}: a central extension
\begin{equation*}
1 \to \ueins \to \mathcal{L} \to LG \to 1
\end{equation*}
in which   both relations are lifted  in a consistent way to $\mathcal{L}$. Thin fusion extensions form a subclass of central extensions of $LG$ with several interesting properties. For example, we show that they are disjoint-commutative: if $p_1,p_2\in \mathcal{L}$ project to loops supported on disjoint subintervals of $S^1$, then they commute, $p_1\cdot p_2=p_2 \cdot p_1$.

The main results of this paper are summarized in  the following main theorem.

\begin{maintheorem}{A}
\label{th:main}
Let $G$ be a connected Lie group. 
A central extension $\mathcal{L}$ of $LG$ is transgressive if and only if it can be equipped with the structure of  a thin fusion extension. Moreover, transgression is a group isomorphism
\begin{equation*}
\bigset{13em}{Isomorphism classes of multiplicative bundle gerbes over $G$ that admit connections} \cong \bigset{9em}{Isomorphism classes of thin fusion extensions of $LG$}\text{.}
\end{equation*}
If $G$ is compact,  both groups are isomorphic to $\h^4(BG,\Z)$.
\end{maintheorem}

As a consequence of Theorem \ref{th:main}, we obtain that transgressive central extensions are disjoint-com\-mu\-tative; this provides an accessible necessary condition for the transgressivity of a central extension.

The present work is a contribution to the programme of  exploring the geometry of loop groups via finite-dimensional, higher geometry over Lie groups. Our main theorem determines the class of central extensions of loop groups that are accessible by such methods: thin fusion extensions.

Transgression of gerbes has first been defined by Gaw\c edzki in relation with two-dimensional conformal field theories \cite{gawedzki3}, and then by Brylinski and McLaughlin in the setting of  sheaves of groupoids \cite{brylinski1,brylinski4}.    The multiplicative  bundle gerbes we use here have been introduced by Carey et al. in \cite{carey4}, and  transgression of those has been developed in \cite{waldorf5}.

The question which central extensions are transgressive has been studied before by Brylinski and McLaughlin \cite{brylinski4,Brylinski1996}. For connected semisimple complex Lie groups, they obtained a characterization in terms of the so-called Segal-Witten reciprocity law. In \cite{brylinski4} it is incorrectly stated \cite{brylinski3} that this reciprocity law also holds for non-complex Lie  groups. Indeed, we provide a counterexample to that statement and prove that only a \emph{weaker} version of the reciprocity law holds for    transgressive central extensions of general Lie groups. We also provide an example of a central extension that   is not transgressive and yet satisfies this weaker version of the reciprocity law. We come to the conclusion that no version of the reciprocity law  appropriately characterizes transgressive central extensions of general Lie groups. It was the main motivation for writing this article to attack the open characterization problem  from a different angle, namely via fusion and thin homotopy equivariance.

The results of this article are based on previous work on transgression    \cite{waldorf9,waldorf10,waldorf11}. A summary of these three papers on only three pages can be found in \cite[Section 1]{waldorf11}. The main result is that  transgression for general smooth manifolds $X$ establishes an equivalence between various categories of gerbes over $X$ and corresponding categories of $S^1$-bundles over the free loop space $LX$, equipped with structure rendering them compatible with fusion and thin homotopy.
The present paper  can be seen as an extension of these results to a multiplicative setting.

The organization of the present paper is as follows. In Section \ref{sec:basics} we introduce  the basic definitions of fusion and thin homotopy equivariance in loop group extensions. In Section \ref{sec:features} we provide a list of features that follow from the presence of these structures, among them    disjoint-commutativity  (Theorem \ref{th:localcomm}). In Section \ref{sec:fusextth} we formulate an integrability condition for thin homotopy equivariant structures on which our notion of thin fusion extensions is based (Definition \ref{def:thinfusionext}).
In Section \ref{sec:transgression} we discuss multiplicative bundle gerbes and their transgression, and introduce our definition of transgressive central extensions (Definition \ref{def:transgressive}). Then we give a proof of Theorem \ref{th:main} (split into Proposition \ref{prop:onlyif} and Corollaries \ref{co:if}, \ref{co:maintext}). In Section \ref{sec:segalwitten} we prove our weaker version of the  Segal-Witten reciprocity law (Theorem \ref{th:recprop}) and provide the two examples that indicate the above-mentioned problems  (Examples \ref{ex:nonmultsplitting} and \ref{ex:nonchar}).

Throughout the paper, we continuously look at two classes of examples: an explicit model of the universal central extension of the loop group of a compact simply-connected Lie group, and various central extensions of  $L\ueins$, of which some turn out to be transgressive and others not.
For the convenience of the reader we include on Page \pageref{sec:tableterm} a table summarizing some terminology we use in this paper.

\paragraph{Acknowledgements.} This work is supported by the DFG network \quot{String Geometry} (project code 594335). 

\setsecnumdepth{1}

\section{Fusion and thin homotopy equivariance over loop groups}

\label{sec:basics}

In this section $G$ is a Lie group and 
\begin{equation*}
1 \to \ueins \to \mathcal{L} \to LG \to 1
\end{equation*}
is a central extension of Fréchet Lie groups \cite{pressley1}. We introduce structures on $\mathcal{L}$ that lift loop fusion and thin homotopy, and discuss the interplay between them.

\begin{comment}
\subsection{Fusion products}
\end{comment}

By $PG$ we denote the set of paths $\gamma:[0,1] \to G$ with sitting instants, i.e. they are constant near the endpoints.  $PG$ is not a Fréchet manifold, but can be treated as a \emph{diffeological space}. Instead of charts, a diffeological space $X$ has plots: maps $c:U \to X$ defined on open subsets $U\subset\R^{n}$, for all $n\in \N$, satisfying three natural axioms, see e.g. \cite{iglesias1}. In case of $PG$, a map $c:U \to PG$ is a plot if and only if $U \times [0,1] \to G\maps (x,t)\mapsto c(x)(t)$ is  smooth.\footnote{A referee suggested  to use  paths all of whose
higher derivatives vanish at the end-points, as opposed to sitting
instants,
as this would simplify some of the arguments in Sections \ref{sec:loopconcat} and \ref{sec:discomm}. In order to stay consistent with my other papers \cite{waldorf9,waldorf10,waldorf11}, on which Sections \ref{sec:fusextth} and \ref{sec:transgression} rely on, I have decided to stick to sitting instants.}

We remark that Fréchet manifolds embed fully faithfully into diffeological spaces \cite{losik1}. Thus, every Fréchet manifold can be regarded as a diffeological space (the plots are are just all smooth maps), and a map between two Fréchet manifolds is smooth in the Fréchet sense if and only if it is smooth in the diffeological sense.

We denote by $PG^{[k]}$ the $k$-fold fibre product of $PG$ over the evaluation map 
\begin{equation*}
\ev\maps PG \to G \times G\maps \gamma \mapsto (\gamma(0),\gamma(1))\text{,}
\end{equation*}
i.e. $PG^{[k]}$ consists of $k$-tuples of paths all sharing a common initial point and a common end point. 
Due to the sitting instants, we have a well-defined smooth map
\begin{equation*}
\lop\maps  \p G^{[2]} \to L G\maps  (\gamma_1,\gamma_2) \mapsto \prev{\gamma_2} \pcomp \gamma_1\text{,}
\end{equation*}
where $\pcomp$ denotes the path concatenation, and $\prev{\gamma}$ denotes the reversed path. The set $PG^{[3]}$ is the modelling space for loop fusion: if $(\gamma_1,\gamma_2,\gamma_3)\in PG^{[3]}$, then we have the two loops $\tau_{12} := \gamma_1 \lop \gamma_2$ and $\tau_{23} := \gamma_2 \lop \gamma_3$ which have the common segment $\gamma_2$. Its deletion gives the new loop $\tau_{13} :=\gamma_1 \lop \gamma_3$.
Loop fusion is multiplicative and strictly associative. 

\begin{definition}
\label{def:fusion} 
\begin{enumerate}[(a)]
\item
A \emph{fusion product} on $\mathcal{L}$ is a smooth bundle morphism
\begin{equation*}
\lambda: \pr_{12}^{*}\lop^{*}\mathcal{L} \otimes \pr_{23}^{*}\lop^{*}\mathcal{L} \to \pr_{13}^{*}\lop^{*}\mathcal{L}
\end{equation*}
over $PG^{[3]}$
that is associative in the sense that 
\begin{equation*}
\lambda(\lambda(p_{12} \otimes p_{23} ) \otimes p_{34} ) = \lambda(p_{12}\otimes \lambda(p_{23} \otimes p_{34}))
\end{equation*} 
for all $p_{ij} \in \mathcal{L}_{\gamma _i \lop \gamma_j}$ and all $(\gamma_1,\gamma_2,\gamma_3,\gamma_4) \in PG^{[4]}$.

\item
A fusion product $\lambda$ is called \emph{multiplicative} if
\begin{equation*}
\lambda_{}(p_{12}^{} \otimes p_{23}^{}) \cdot \lambda(p_{12}' \otimes p_{23}') = \lambda(p^{}_{12}p_{12}' \otimes p^{}_{23}p_{23}')
\end{equation*}
for all elements $p_{ij}\in \mathcal{L}_{\gamma_i \lop \gamma_j}$, $p_{ij}'\in \mathcal{L}_{\gamma_i \lop \gamma_j'}$ and all $(\gamma_1,\gamma_2,\gamma_3),(\gamma_1',\gamma_2',\gamma_3') \in PG^{[3]}$.
\end{enumerate}
\end{definition}

Here $\pr_{ij}: PG^{[3]} \to PG^{[2]}$ is the projection to the indexed factors, and $\mathcal{L}_{\tau}$ denotes the fibre of $\mathcal{L}$ over a loop $\tau \in LG$.
\label{par:fuspres}
If $\mathcal{L}'$ is another central extension equipped with a fusion product $\lambda'$, then an isomorphism $\varphi:\mathcal{L} \to \mathcal{L}'$ is called \emph{fusion-preserving}, if $\varphi(\lambda(p_{12} \otimes p_{23}))=\lambda'(\varphi(p_{12}) \otimes \varphi(p_{23}))$
for all elements $p_{ij}\in \mathcal{L}_{\gamma_i \lop \gamma_j}$ and all $(\gamma_1,\gamma_2,\gamma_3) \in PG^{[3]}$.

\begin{comment}
\subsection{Thin homotopy equivariance}
\end{comment}

A homotopy between loops is the same thing as a path in the loop space: if  $\gamma:[0,1] \to LG$ is a path, then
\begin{equation*}
h_{\gamma}:[0,1] \times S^1 \to G: (t,z) \mapsto \gamma(t)(z) \end{equation*}
is the corresponding homotopy.
The rank of $h$  can at most be two. If it is \emph{less} than two  we call  the path $\gamma$ and the homotopy $h_{\gamma}$  \emph{thin}. A path $\gamma$ is thin if and only if $h_{\gamma}^{*}\omega=0$ for all 2-forms $\omega\in\Omega^2(G)$, this leads to the saying that thin homotopies  \quot{sweep out a surface of zero area}. 

\begin{comment}
Let $\diff^{+}(S^1)$ denote the Fréchet Lie group of orientation-preserving diffeomorphisms of $S^1$. It is connected, i.e. for every $\varphi\in \diff^{+}(S^1)$ there exists a smooth family $\varphi_t$ with $\varphi_0=\id_{S^1}$ and $\varphi_1=\varphi$. If $\tau \in LG$ is a loop, then $\gamma(t) := \tau\circ \varphi_t$ is a path between $\tau$ and $\tau \circ \varphi$, and it is thin because the corresponding homotopy it factors through $\varphi_t(z) \in S^1$, which is one-dimensional. In other words, the $\diff^{+}(S^1)$-action on $LG$ preserves the thin homotopy class. 
\end{comment}

\begin{comment}
The group structure of $LG$ is quite incompatible with  thin homotopies. If $(\tau_0, \tau_1)$ and $(\gamma_0, \gamma_1)$ are pairs of thin homotopic loops in $LG$, then $\tau_0\gamma_0$ and $\tau_1\gamma_1$ are in general not thin homotopic. They are only thin homotopic, if we find  paths  $\tau$ and $\gamma$ in $LG$ connecting them such that $(\tau,\gamma)$ is a thin path in $G \times G$. That $\tau$, $\gamma$ and $\tau\gamma$ are thin paths in $G$ is necessary, but not sufficient. 
\end{comment}

We denote by $\thinpairs {LG} \subset LG \times LG$ the diffeological space  consisting of pairs $(\tau_1,\tau_2)$ of thin homotopic loops, i.e. there exists a  homotopy $h: [0,1] \times S^1 \to G$  of rank one. The plots are smooth maps $c:U \to \thinpairs{LG}$ such that locally the thin homotopies can be chosen in smooth families, see \cite[Section 3.1]{waldorf11}.

\begin{definition}
\label{def:thes}
\begin{enumerate}[(a)]
\item 
\label{def:thes:a}
A \emph{thin homotopy equivariant structure} on $\mathcal{L}$ is a smooth bundle isomorphism
\begin{equation*}
d: \mathrm{pr}_1^{*}\mathcal{L} \to \mathrm{pr}_2^{*}\mathcal{L}
\end{equation*}
over $\thinpairs {LG}$ that satisfies the  cocycle condition
$d_{\tau_2,\tau_3} \circ d_{\tau_1,\tau_2} = d_{\tau_1,\tau_3}$
for any triple $(\tau_1,\tau_2,\tau_3)$ of thin homotopic loops.

\item
\label{def:thes:b}
A thin homotopy equivariant structure $d$ is called \emph{multiplicative} if
\begin{equation*}
d_{\tau_0\gamma_0,\tau_1\gamma_1}(p \cdot q)=d_{\tau_0,\tau_1}(p)\cdot d_{\gamma_0,\gamma_1}(q)
\end{equation*}
for all $((\tau_0,\gamma_0),(\tau_1,\gamma_1))\in \thinpairs{L(G\times G)}$ and all $p\in \mathcal{L}_{\tau_0}$, $q\in \mathcal{L}_{\gamma_0}$.

\end{enumerate}
\end{definition}

Note that $((\tau_0,\gamma_0),(\tau_1,\gamma_1))\in \thinpairs{L(G\times G)}$ means that there exists a thin path $(\tau,\gamma)$ in $L(G \times G)$ connecting $(\tau_0,\gamma_0)$ with $(\tau_1,\gamma_1)$. It is necessary, but not sufficient, that the paths $\tau$, $\gamma$, and $\tau\gamma$ in $LG$ are separately thin.

\label{par:thinbm}
If $\mathcal{L}'$ is another central extension equipped with a thin homotopy equivariant structure $d'$, then an isomorphism $\varphi:\mathcal{L} \to \mathcal{L}'$ is called \emph{thin} if $\varphi(d_{\tau_0,\tau_1}(p))=d'_{\tau_0,\tau_1}(\varphi(p))$.

A thin homotopy equivariant structure $d$ induces an equivariant structure on $\mathcal{L}$ for the action of the group $\diff^{+}(S^1)$ of orientation-preserving diffeomorphisms of the circle on $LG$ by pre-composition. Indeed, suppose $\varphi$ is an orientation-preserving diffeomorphism, $\tau \in LG$ and $p \in \mathcal{L}_{\tau}$. Since $\diff^{+}(S^1)$ is connected, there exists a path $\varphi_t\in \diff^{+}(S^1)$ with $\varphi_0=\id_{S^1}$ and $\varphi_1=\varphi$. The map $\gamma:[0,1] \to LG\maps t \mapsto \tau \circ \varphi_t$ is a thin path; in particular, $\tau$ and $\tau \circ \varphi$ are thin homotopic. We define
\begin{equation}
\label{eq:diffaction}
p \cdot \varphi := d_{\tau,\tau\circ\varphi}(p) \in \mathcal{L}_{\tau \circ \varphi}\text{.}
\end{equation}

\begin{lemma}
\label{lem:equiv}
Formula \erf{eq:diffaction} defines a smooth action of $\diff^{+}(S^1)$ on $\mathcal{L}$. If $d$ is multiplicative, it is an action by group homomorphisms and acts trivially on the  central $\ueins$-subgroup of $\mathcal{L}$.
\end{lemma}

\begin{proof}
That it is an action follows from the cocycle condition for $d$. It is smooth because $d$ is a smooth bundle isomorphism over $\thinpairs{LG}$. If $d$ is multiplicative, we compute
for $p_1\in \mathcal{L}_{\tau_1}$ and $p_2\in \mathcal{L}_{\tau_2}$ 
\begin{equation*}
(p_1p_2) \cdot \varphi = d_{\tau_1\tau_2,(\tau_1 \tau_2) \circ \varphi}(p_1p_2) = d_{\tau_1\tau_2,(\tau_1  \circ \varphi)(\tau_2 \circ \varphi)}(p_1p_2)= d_{\tau_1,\tau_1  \circ \varphi}(p_1) d_{\tau_2,\tau_2  \circ \varphi}(p_2)\text{.}
\end{equation*}
The restriction of the $\diff^{+}(S^1)$-action on $LG$ to constant loops is trivial. So if $p\in \mathcal{L}$ projects to a constant loop, we have $p\cdot\varphi= p$.       In particular, $\diff^{+}(S^1)$ acts trivially on $\ueins$. 
\end{proof}

\begin{comment}
\subsection{Compatibility and symmetrization}
\end{comment}

\begin{definition}
\label{def:comp}
Suppose $\mathcal{L}$ is equipped with a fusion product $\lambda$ and  a thin homotopy equivariant structure $d$.
\begin{enumerate}[(a)]
\item 
\label{def:comp:a}
We say that $d$ is \emph{compatible} with  $\lambda$, if  for all  paths $(\gamma_1,\gamma_2,\gamma_3) \in P(PG^{[3]})$ such that the three paths $t \mapsto \gamma_i(t) \lop \gamma_j(t)\in LG$ are  thin, the diagram
\begin{equation*}
\alxydim{@=\xyst@R=\xyst}{\mathcal{L}_{\gamma_1(0) \lop\gamma_2(0)} \ar[d]_{d \otimes d} \otimes \mathcal{L}_{\gamma_2(0) \lop \gamma_3(0)} \ar[r]^-{\lambda} & \mathcal{L}_{\gamma_1(0) \lop \gamma_3(0)} \ar[d]^{d} \\\mathcal{L}_{\gamma_1(1) \lop \gamma_2(1)} \otimes \mathcal{L}_{\gamma_2(1) \lop \gamma_3(1)} \ar[r]_-{\lambda} & \mathcal{L}_{\gamma_1(1) \lop \gamma_3(1)}}
\end{equation*}
is commutative. 

\item
\label{item:symmetrizes}
We say that $d$ \emph{symmetrizes} $\lambda$
if for all $(\gamma_1,\gamma_2,\gamma_3) \in PG^{[3]}$ and all $p\in \mathcal{L}_{\gamma_1 \lop \gamma_2}$ and $p' \in \mathcal{L}_{\gamma_2 \lop \gamma_3}$
\begin{equation*}
d_{\gamma_1 \lop \gamma_3,\prev{\gamma_3} \lop \prev{\gamma_1}}(\lambda(p \otimes p')) = \lambda(d_{\gamma_2 \lop \gamma_3,\prev{\gamma_3} \lop \prev{\gamma_2}}(p') \otimes d_{\gamma_1 \lop \gamma_2,\prev{\gamma_2} \lop \prev{\gamma_1}}(p))\text{.}
\end{equation*}

\item
\label{def:comp:fusive}
We say that $d$ is \emph{fusive} with respect to $\lambda$, if it is compatible and symmetrizing,

\end{enumerate}
\end{definition}

\noindent
For \erf{item:symmetrizes} we remark that if $r_{\pi}\in \diff^{+}(S^1)$ denotes the rotation by an angle of $\pi$, then $(\gamma_i \lop \gamma_j)\circ r_{\pi} = (\prev{\gamma_j} \lop \prev{\gamma_i})$; in particular, $\gamma_i \lop \gamma_j$ and $\prev{\gamma_j}\lop \prev{\gamma_i}$ are thin homotopic.

\begin{example}
\label{ex:trivial}
\label{ex:toptrivial}
Suppose $P$ is a principal $\ueins$-bundle over $G \times G$ with connection, such that there exists a connection-preserving isomorphism
\begin{equation}
\label{eq:bundleiso}
P_{g_1,g_2}\otimes P_{g_1g_2,g_3} \cong P_{g_2,g_3}\otimes P_{g_1,g_2g_3}
\end{equation}
over $G \times G \times G$. The subscript notation is that, for instance, $P_{g_1g_2,g_3}$ is the pullback of $P$ along the map $(g_1,g_2,g_3) \mapsto (g_1g_2,g_3)$. 
The holonomy of $P$ is a smooth map
$\eta:LG \times LG \to \ueins$
such that 
\begin{equation*}
\eta(\tau_1\tau_2,\tau_3) \cdot \eta(\tau_1,\tau_2)=\eta(\tau_1,\tau_2\tau_3)\cdot\eta(\tau_2,\tau_3)
\end{equation*}
for all $\tau_1,\tau_2,\tau_3\in LG$. Thus, $\eta$ is a 2-cocycle in the smooth group cohomology of $LG$. It defines a group structure on $\mathcal{L}_P:= \ueins \times LG$ via $(z_1,\tau_1) \cdot (z_2,\tau_2) := (z_1z_2\eta(\tau_1,\tau_2),\tau_1\tau_2)$, making $\mathcal{L}_P$ a central extension of $LG$.
\begin{comment}
The cocycle condition implies (via $\tau_1=\tau$, $\tau_2=\tau_3=1$) $\eta(\tau,1)=\eta(1,1)$ and (via $\tau_1=\tau_2=1$, $\tau_3=\tau$) $\eta(1,1) =\eta(1,\tau)$. Further (via $\tau_1=\tau_3=\tau$, $\tau_2=\tau^{-1}$) $\eta(\tau,\tau^{-1})=\eta(\tau^{-1},\tau)$. The identity element is $(\eta(1,1)^{-1},1)$, as we have $(z,\tau)\cdot(\eta(1,1)^{-1},1)=(z\eta(1,1)^{-1}\eta(\tau,1),1)=(z,1)$. The inverse of an element $(z,\tau)$ is $(z^{-1}\eta(\tau,\tau^{-1})^{-1}\eta(1,1)^{-1},\tau^{-1})$.

Note that one can always find a cohomologous cocycle $\eta'$ with $\eta'(1,1)=1$, by using the 1-cocycle $\varepsilon(\tau):= \eta(1,1)^{-1}$. Indeed, $\eta'(1,1)=\eta(1,1)\Delta\varepsilon(1,1)=\eta(1,1)\varepsilon(1)\varepsilon(1)\varepsilon(1)^{-1}$. 

If $G$ is abelian then $\Delta\varepsilon(\tau_1,\tau_2)=\varepsilon(\tau_1)\varepsilon(\tau_2)\varepsilon(\tau_1\tau_2)^{-1}=\varepsilon(\tau_2)\varepsilon(\tau_1)\varepsilon(\tau_2\tau_1)^{-1}=\Delta\varepsilon(\tau_2,\tau_1)$. In particular, if $\eta=\Delta\varepsilon$, then it follows $\eta(\tau_1,\tau_2)=\eta(\tau_2,\tau_1)$. Also, if $\eta$ and $\eta'$ are cohomologous, the we have $\eta(\tau_1,\tau_2)\eta'(\tau_2,\tau_1)=\eta(\tau_2,\tau_1)\eta'(\tau_1,\tau_2)$. 
\end{comment}
As the holonomy of a bundle, $\eta$ is a fusion map in the sense of \cite{waldorf9}, i.e. it satisfies
\begin{equation*}
\eta(\gamma_1\lop\gamma_3,\gamma_1'\lop\gamma_3') 
\\= \eta(\gamma_1\lop\gamma_2,\gamma_1'\lop\gamma_2')\cdot \eta(\gamma_2\lop\gamma_3,\gamma_2'\lop\gamma_3')\text{.}
\end{equation*}
for all $(\gamma_1,\gamma_2,\gamma_3),(\gamma_1',\gamma_2',\gamma_3') \in PG^{[3]}$. This is equivalent to the statement that the trivial fusion product $\lambda((z_{12},\gamma_1 \lop \gamma_2)\otimes (z_{23},\gamma_2 \lop\gamma_3)):=(z_{12}z_{23},\gamma_1\lop\gamma_3)$ is multiplicative. 
\begin{comment}
A fusion product on $\mathcal{L}$ is the same as a smooth map $\lambda\maps PG^{[3]} \to S^1$ such that 
\begin{equation*}
\lambda(\gamma_1,\gamma_3,\gamma_4)\cdot \lambda(\gamma_1,\gamma_2,\gamma_3)=\lambda(\gamma_1,\gamma_2,\gamma_4)\cdot \lambda(\gamma_2,\gamma_3,\gamma_4)
\end{equation*}
for all $(\gamma_1,\gamma_2,\gamma_3,\gamma_4)\in PG^{[4]}$. It is multiplicative if and only if
\begin{multline*}
\omega(\gamma_1\lop\gamma_3,\gamma_1'\lop\gamma_3')\cdot \lambda_{}(\gamma_1,\gamma_2,\gamma_3) \cdot \lambda(\gamma_1',\gamma_2',\gamma_3') 
\\= \lambda(\gamma_1\gamma_1',\gamma_2\gamma_2',\gamma_3\gamma_3')\cdot \omega(\gamma_1\lop\gamma_2,\gamma_1'\lop\gamma_2')\cdot \omega(\gamma_2\lop\gamma_3,\gamma_2'\lop\gamma_3')
\end{multline*}
for all $(\gamma_1,\gamma_2,\gamma_3),(\gamma_1',\gamma_2',\gamma_3') \in PG^{[3]}$.
If $\omega$ is trivial, this means that $\lambda$ is a group homomorphism. If $\omega$ is non-trivial, it is not clear that multiplicative fusion products exist. For example, the trivial fusion product $\lambda \equiv 1$ is  multiplicative if and only if $\omega$  satisfies the fusion relation
\begin{equation*}
\omega(\gamma_1\lop\gamma_3,\gamma_1'\lop\gamma_3') 
\\= \omega(\gamma_1\lop\gamma_2,\gamma_1'\lop\gamma_2')\cdot \omega(\gamma_2\lop\gamma_3,\gamma_2'\lop\gamma_3')\text{.}
\end{equation*}
\end{comment}
Likewise, we have the trivial thin homotopy equivariant structure $d(z,\tau_0):=(z,\tau_1)$ for each  $(\tau_0,\tau_1)\in \thinpairs{LG}$.
It is  fusive with respect to the trivial fusion product, and it is multiplicative with respect to the group structure defined by $\eta$ because the holonomy of a connection only depends on the thin homotopy class of a loop. 

As a concrete example of this construction, one can take  $G=\ueins$ and $P$ the Poincaré bundle over $T:=\ueins \times \ueins$, equipped with its canonical connection. In differential cohomology, $P\in \hat H^2(T)$ is the cup product of the two projections  $\pr_1,\pr_2: T \to \ueins$ regarded as elements in $\hat H^1(T)$. This implies that the Poincaré bundle has an isomorphism \erf{eq:bundleiso}. Its holonomy can be described in the following way. If $\tau \in L\ueins$, let $n\in \Z$ be the winding number of $\tau$. One can find a smooth map $f:\R \to \R$ such that $f(t+1)=f(t)+n$ and $\tau=\mathrm{e}^{2\pi \im f}$. For $\tau=(\tau_1,\tau_2)\in LT$, we get
\begin{equation*}
\eta(\tau_1,\tau_2)=\mathrm{Hol}_{P}(\tau) =\exp 2\pi\im \left ( n_1f_2(0)- \int_0^1 f_1(s)f_2'(s) \mathrm{d}s \right )\text{.} 
\end{equation*}
\begin{comment}
The Poincaré bundle has the total space 
\begin{equation*}
P := (\R \times \R \times S^1) \;/\; \sim
\end{equation*}
with $(a,b,t)\sim (a',b',t'):=(a+n,b+m,nb+t)$ for all $n,m\in \Z$, the projection $(a,b,t) \mapsto (a,b)$, and the $S^1$-action $(a,b,t)\cdot s := (a,b,t+s)$.
With $A_{a,b} := a\mathrm{d}b$ a 1-form on $\R \times \R$ we can check that $\omega_{a,b,t} := A_{a,b}-\mathrm{d}t \in \Omega^1(\R \times \R \times S^1)$ descends and gives a connection on $P$.
If $\tau=(\tau_1,\tau_2)\in LT$, then 
\begin{equation*}
f:\R \to P : t \mapsto (f_1(t),f_2(t), \exp 2\pi \im   \int_0^t f_1(s)f_2'(s) \mathrm{d}s )
\end{equation*}
is a lift. We compute
\begin{equation*}
\omega(f'(t)) = A(f'_1(t),f'_2(t)) -f_1(t)f_2'(t) = 0\text{,}
\end{equation*}
i.e. the lift $f$ is horizontal. Further, we have 
\begin{multline*}
f(1)=(f_1(1),f_2(1), \int_0^1 f_1(s)f_2'(s) \mathrm{d}s)=(f_1(0)+n_1,f_2(0)+n_2,n_1f_2(0)-n_1f_2(0)+ \int_0^1 f_1(s)f_2'(s) \mathrm{d}s) \\\sim(f_1(0),f_2(0),-n_1f_2(0)+ \int_0^1 f_1(s)f_2'(s) \mathrm{d}s)=f(0)-n_1f_2(0)+ \int_0^1 f_1(s)f_2'(s) \mathrm{d}s\text{.}
\end{multline*}
This shows the holonomy formula.
\end{comment}
Using this formula, we obtain a central extension of $L\ueins$, equipped with a multiplicative fusion product and a multiplicative and fusive thin homotopy equivariant structure. 
\end{example}

\begin{example}
\label{ex:mickelsson}
\label{ex:wzwmodel}
Let $G$ be a compact, simple, connected, simply-connected  Lie group, so that $LG$ has a universal central extension \cite{pressley1}. It can be realized by the following model of Mickelsson \cite{mickelsson1}. 
We consider pairs $(\phi,z)$ where $\phi\maps D^2 \to G$ is a smooth map that is radially constant near the boundary, and $z\in \ueins$. \begin{comment}
For technical reasons, we require that $\phi$ is radially constant near the boundary, i.e.  there exists  $\varepsilon>0$ such that $\phi(r \mathrm{e}^{2\pi\im\varphi})=\phi(\mathrm{e}^{2\pi\im\varphi})$ for all  $1-\varepsilon<r\leq 1$. 
If $\phi$ is any smooth map, one can always make it radially constant. For this purpose, consider a smooth map $\psi\maps  [0,1] \to [0,1]$ with $\psi(r)=0$ for $0\leq r < \varepsilon$ and $\psi(r)=1$ for $r> 1-\varepsilon$. It induces a smooth map $\psi_{D}\maps  D^2 \to D^2\maps  r\mathrm{e}^{2\pi\im \varphi} \mapsto \psi(r) \cdot \mathrm{e}^{2\pi\im \varphi}$. Then, the composite $\phi \circ \psi_D$ is radially constant: if $r>1-\varepsilon$, we have
\begin{equation*}
(\phi \circ \psi_D)(r\mathrm{e}^{2\pi\im\varphi}) = \phi(\psi(r)  \mathrm{e}^{2\pi\im\varphi})= \phi(  \mathrm{e}^{2\pi\im\varphi})\text{.}
\end{equation*} 
\end{comment}
We impose the following equivalence relation:
\begin{equation*}
(\phi,z) \sim (\phi',z') 
\quad\Leftrightarrow\quad
\partial\phi = \partial\phi'
\quad\text{ and }\quad
z = z' \cdot \mathrm{e}^{2\pi\im S_{\mathrm{WZ}}(\Phi)}\text{.}
\end{equation*}
Here, $\partial\phi\in LG$ denotes the restriction of $\phi$ to the boundary, and  $\Phi\maps  S^2 \to G$ is the  map defined on the northern hemisphere by $\phi$ (with the orientation-preserving identification) and on the southern hemisphere by $\phi'$ (with the orientation-reversing identification). 
\begin{comment}
Due to the above technical assumptions, this gives a smooth map.
We use the open sets:
\begin{eqnarray*}
S^2_+ &\df& S^2 \cap \R^3_{z>0}
\\
S^2_- &\df& S^2 \cap \R^3_{z<0}
\\
Ä_{\varepsilon} &\df& \left \lbrace (x,y,z) \in S^2 \;|\; x^2+y^2 > (1-\varepsilon)^2 \right \rbrace
\end{eqnarray*}
Here, $\varepsilon$ is the minimum of the two $\varepsilon$ of the maps $\phi$ and $\phi'$. 
On these we have smooth maps:
\begin{eqnarray*}
p_{\pm}&:& S^2_{\pm} \to D^2\maps  (x,y,z) \mapsto (x,y)
\\
p_{Ä}&:& Ä_{\varepsilon} \to S^1\maps  (x,y,z) \mapsto \frac{(x,y)}{\sqrt{x^2+y^2}}
\end{eqnarray*}
We oriented the 2-sphere with \emph{outward} pointing normal vector, so that $p_{+}$ is orientation-preserving and $p_{-}$ is orientation-reversing.
We denote by $\tau\maps S^1 \to G$ the loop with $\tau=\partial\phi=\partial\phi'$ by assumption, and consider the smooth maps
\begin{equation*}
\phi \circ p_+
\quomma
\phi' \circ p_-
\quand
\tau \circ p_{Ä}\text{.}
\end{equation*}
We claim that they coincide on the overlaps $S^2_\pm \cap Ä_{\varepsilon}$ and so give a smooth map $\Phi$. Indeed, we have for $(x,y,z)$ with $z> 0$ and $x^2+y^2> (1-\varepsilon)^2$:
\begin{equation*}
(\tau \circ p_{Ä})(x,y,z) = \tau(\frac{(x,y)}{\sqrt{x^2+y^2}})= \phi(\frac{(x,y)}{\sqrt{x^2+y^2}})= \phi((x,y)) = (\phi \circ p_+)(x,y,z)
\end{equation*} 
and the same for $\phi'$ over $z<0$.
\end{comment}
The symbol  $S_{\mathrm{WZ}}$ stands for the \emph{Wess-Zumino term} defined as follows. Because $G$ is 2-connected, the map $\Phi$ can be extended to a smooth map $\tilde\Phi\maps  D^3 \to G$ defined on the solid ball. Then,
\begin{equation}
\label{eq:H}
S_{\mathrm{WZ}}(\Phi) \df \int_{D^3} \tilde\Phi^{*}H
\quith 
H\df\frac{1}{6} \left \langle \theta \wedge [\theta \wedge \theta]  \right \rangle \in \Omega^3(G)\text{.}
\end{equation}
Here, $\theta \in \Omega^1(G,\mathfrak{g})$ is the left-invariant Maurer-Cartan from on $G$. The bilinear form $\left \langle -,-  \right \rangle$ is  normalized such that the closed 3-form $H$ represents a generator $\h^3(G,\Z)\cong \Z$. Now, the total space of the principal $\ueins$-bundle $\wzwmodel$ is the set of equivalence classes of  pairs $(\phi,z)$. 
The bundle projection sends $(\phi,z)$  to $\partial\phi\in LG$, and the $\ueins$-action is given by multiplication in the $\ueins$-component. 
The group structure on $\wzwmodel$ turning it into a central extension is given by the  \emph{Mickelsson product}   \cite{mickelsson1}:
\begin{equation*}
\wzwmodel \times \wzwmodel \to \wzwmodel\maps  ((\phi_1, z_1), (\phi_2, z_2)) \mapsto (\phi_1\phi_2, z_1z_2 \cdot \exp 2\pi\im \left (-\int_{D^2} (\phi_1,\phi_2)^{*}\rho \right )  )  \text{,}
\end{equation*}
where $\rho$ is defined by
\begin{equation}
\label{eq:rho}
\rho \df  \frac{1}{2}\left \langle  \pr_1^{*}\theta \wedge \pr_2^{*}\bar\theta  \right \rangle \in \Omega^2(G \times G)\text{.}
\end{equation}
The two differential forms $H$ and $\rho$ satisfy the identities
\begin{eqnarray}
\label{eq:deltaH}
\Delta H &:=&  H_{g_1}-H_{g_1g_2}  + H_{g_2}=\mathrm{d}\rho
\\
\label{eq:deltarho}
\Delta\rho &:=& \rho_{g_1,g_2}  + \rho_{g_1g_2,g_3}  -  \rho_{g_2,g_3} -  \rho_{g_1,g_2g_3}=0\text{.}
\end{eqnarray}
for all $g_1,g_2,g_3\in G$, where the subscripts are meant so that $\rho_{g_1g_2,g_3}$ is the pullback of $\rho$ along the map $(g_1,g_2,g_3) \mapsto (g_1g_2,g_3)$. Eq.
\erf{eq:deltaH} assures that the Mickelsson product is well-defined on equivalence classes, and \erf{eq:deltarho} implies its associativity.

A fusion product on $\wzwmodel$ is defined as follows. For $(\gamma_1,\gamma_2,\gamma_3) \in PG^{[3]}$, we define
\begin{equation}
\label{eq:explicitfusionproduct}
\begin{array}{rcccl}
\lambda_{\gamma_1,\gamma_2,\gamma_3} &\maps&  \wzwmodel|_{\gamma_1\lop\gamma_2} \otimes \wzwmodel|_{\gamma_2\lop\gamma_3} &\to& \wzwmodel|_{\gamma_1\lop\gamma_3}
\\ &&(\phi_{12},z_{12}) \otimes (\phi_{23},z_{23}) &\mapsto& (\phi_{13},z_{12}  z_{23}\cdot \mathrm{e}^{-2\pi\im S_{\mathrm{WZ}}(\Psi)})\text{,} 
\end{array}
\end{equation}
where $\phi_{13}\maps  D^2 \to G$ is an arbitrarily chosen smooth map with  $\partial\phi_{13}=\gamma_1 \lop \gamma_3$, and $\Psi\maps  S^2 \to G$ is obtained by trisecting $S^2$ along the longitudes $0$, $\frac{2\pi}{3}$ and $\frac{4\pi}{3}$, and  prescribing $\Psi$ on each sector with the maps $\phi_{12}$, $\phi_{23}$ (with orientation-preserving identification) and $\phi_{13}$ (with orientation-reversing identification), respectively, see Figure \ref{fig:fusion}. 
\begin{figure}[h]
\begin{center}
\includegraphics{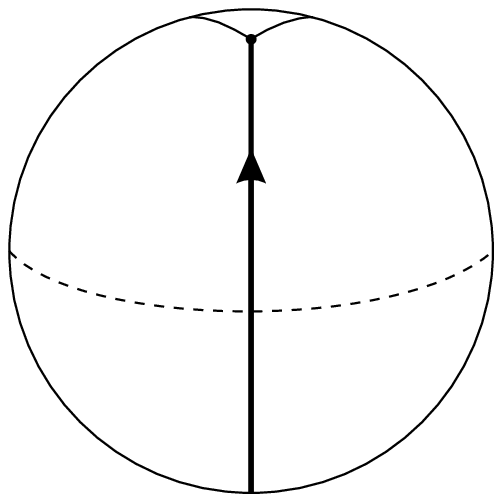}\setlength{\unitlength}{1pt}\begin{picture}(0,0)(202,648)\put(138.40186,740.21499){$\gamma_2$}\put(164.29550,718.90951){$\phi_{23}$}\put(89.17512,719.19893){$\phi_ {12}$}\put(135.75572,772.48205){$h$}\put(136.41194,692.16775){$0$}\end{picture}
\hspace{4em}
\includegraphics{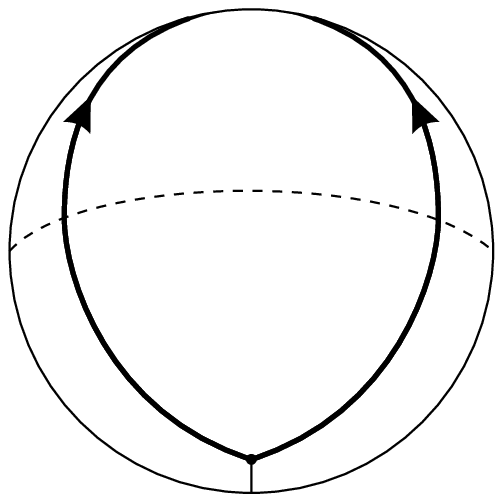}\setlength{\unitlength}{1pt}\begin{picture}(0,0)(408,648)\put(270.36737,753.73767){$\gamma_3$}\put(309.71424,706.42687){\rebox{$\phi_{13}$}}\put(317.72268,652.44614){$g$}\put(345.25701,754.02198){$\gamma_1$}\put(263.90523,717.71839){$\frac{2\pi}{3}$}\put(351.38069,715.94787){$\frac{4\pi}{3}$}\end{picture}
\end{center}
\caption{Front and back view to the map $\Psi$ in the definition of the fusion product over a triple $(\gamma_1,\gamma_2,\gamma_3) \in PG^{[3]}$ with common initial point $g$ and common end point $h$.}
\label{fig:fusion}
\end{figure}
\begin{comment}
This map $\Psi$ is  smooth due to the sitting instants of the paths and the requirement that the maps $\phi_{ij}$ are radially constant.
The open complement in the 2-sphere $S^2$ of the longitudes $0$, $\frac{3\pi}{2}$ and $\frac{4\pi}{2}$ has three connected components $T_1$, $T_2$, $T_3$. Each can be identified with the open disc by means of the following diffeomorphisms $\sigma_k\maps T_k \to D^2_{open}$.   
First we have to map the three angle intervals 
\begin{equation*}
\textstyle
I_1 = (0,\frac{2\pi}{3})
\quomma
I_2 = (\frac{2\pi}{3},\frac{4\pi}{3})
\quand
I_3 = (\frac{4\pi}{3},2\pi)
\end{equation*}
diffeomorphically to $(0,\pi)$, which is done by the maps
\begin{equation*}
s_1(t) \df\pi- \textstyle\frac{3}{2}t
\quomma
s_2(t) \df 2\pi- \frac{3}{2}t}
\quand
s_3(t) \df \frac{3}{2}t- 2\pi
\end{equation*}
of which $s_3$ is orientation-preserving.
Then we define
\begin{equation*}
\sigma_k \maps  T_k \to D^2_{open} \maps  (r\sin\varphi,r\cos\varphi,z) \mapsto (r \cos(s_k(\varphi)),z)\text{.}
\end{equation*}

The smooth maps $\phi_{12} \circ \sigma_1$, $\phi_{23} \circ \sigma_2$   and $\phi_{13} \circ \sigma_3$ define a smooth map $\Psi\maps  S^2 \to G$. 
\end{comment}
That  \erf{eq:explicitfusionproduct} is independent of the choice of $\phi_{13}$ follows from the identity $ S_{\mathrm{WZ}}(\Psi) =  S_{\mathrm{WZ}}(\Psi')S_{\mathrm{WZ}}(\Phi_{13})$ for Wess-Zumino terms, where $\Psi'$ is obtained as described above but using a different map $\phi_{13}'$ instead of $\phi_{13}$, and $\Phi_{13}$ is obtained in the way described earlier from $\phi_{13}$ and $\phi_{13}'$. 
\begin{comment}
Suppose $\phi_{13}'$ is a different smooth map with $\partial\phi_{13}'=\gamma_1\lop\gamma_3$. We denote by $\Psi'$ the map obtained using $\phi_{13}'$ instead of $\phi_{13}$.
We have to show that
\begin{equation*}
(\phi_{13},z_{12}  z_{23}\cdot \mathrm{e}^{2\pi\im S_{\mathrm{WZ}}(\Psi)}) \sim (\phi'_{13},z_{12}  z_{23}\cdot \mathrm{e}^{2\pi\im S_{\mathrm{WZ}}(\Psi')})\text{,}
\end{equation*} 
which is equivalent to showing that
\begin{equation*}
\mathrm{e}^{2\pi\im S_{\mathrm{WZ}}(\Psi')} = \mathrm{e}^{2\pi\im (S_{\mathrm{WZ}}(\Psi) + S_{\mathrm{WZ}}(\Phi_{13}))} \text{.}
\end{equation*} 
This follows from the claimed identity
\begin{equation*}
S_{\mathrm{WZ}}(\Psi) =  S_{\mathrm{WZ}}(\Psi')S_{\mathrm{WZ}}(\Phi_{13})
\end{equation*} 
which follows from the following consideration. Let $\tilde\Psi'\maps  D^3 \to G$ extend $\Psi'$ to the solid ball, and let $\tilde \Phi_{13}$ extend $\Phi_{13}$. Both maps can be glued along the common piece $T_{3}$ of the boundary $S^2$ of $D^3$, where $\tilde\Psi'$ and $\tilde\Phi_{13}$ coincide (namely, they are both $\phi'_{13}$, and the orientations are opposite). The glued map $\tilde\Psi'$ extends $\Psi'$ to the solid ball. 
\end{comment}
Definition \erf{eq:explicitfusionproduct} is also well-defined under the equivalence relation $\sim$ due to a similar identity for Wess-Zumino terms.  
\begin{comment}
Suppose we have equivalences $(\phi_{12},z_{12}) \sim (\phi_{12}',z_{12}')$ and $(\phi_{23},z_{23}) \sim (\phi_{23}',z_{23}')$. We may choose the same map $\phi_{13}$  for both computations, and denote now by $\Psi'$ the map obtained from $\phi_{12}'$, $\phi_{23}'$ and $\phi_{13}$. We
have to show that
\begin{equation*}
(\phi_{13},z_{12}  z_{23}\cdot \mathrm{e}^{2\pi\im S_{\mathrm{WZ}}(\Psi)}) \sim (\phi_{13},z_{12}  'z_{23}'\cdot \mathrm{e}^{2\pi\im S_{\mathrm{WZ}}(\Psi')})\text{,}
\end{equation*}
i.e. that
\begin{equation*}
z_{12}  z_{23}\cdot \mathrm{e}^{2\pi\im S_{\mathrm{WZ}}(\Psi)} = z'_{12}  z'_{23}\cdot \mathrm{e}^{2\pi\im (S_{\mathrm{WZ}}(\Psi') + S_{\mathrm{WZ}}(\Phi_{13}))\text{,}}
\end{equation*}
where $\Phi_{13}$ is obtained from $\phi_{13}$ and itself -- in fact one can show that $\exp(2\pi\im S_{\mathrm{WZ}}(\Phi_{13}))=1$. So it boils down to the identity
\begin{equation*}
S_{\mathrm{WZ}}(\Phi_{12})S_{\mathrm{WZ}}(\Phi_{23})S_{\mathrm{WZ}}(\Psi) = S_{\mathrm{WZ}}(\Psi')\text{.}
\end{equation*}
In order to prove this identity, we proceed similarly as above. Let $\tilde\Psi\maps  D^3 \to G$ be an extension of $\Psi$ to the solid ball, and let $\tilde\Phi_{12}$, $\tilde\Phi_{23}$ be extensions of $\Phi_{12}$, $\Phi_{23}$, respectively. The orientations fit to glue those to a map $\tilde\Psi'$ that extends $\Psi'$.
\end{comment}
Associativity follows from reparameterization invariance of the integral, and multiplicativity follows from the Polyakov-Wiegmann formula 
\begin{equation}
\label{eq:polwieg}
\mathrm{e}^{2\pi \im S_{\mathrm{WZ}}(\Phi_1)}\cdot \mathrm{e}^{2\pi \im S_{\mathrm{WZ}}(\Phi_2)} = \mathrm{e}^{2\pi\im S_{\mathrm{WZ}}(\Phi_1\Phi_2)}\cdot \exp 2\pi\im \left ( \int_{(\Phi_1,\Phi_2)} \rho \right )\text{,}
\end{equation}
which in turn follows from \erf{eq:deltarho}.

A thin homotopy equivariant structure $d$ is  defined as follows. Suppose $(\gamma_0,\gamma_1)\in \thinpairs{LG}$. Then,
\begin{equation*}
d_{\gamma_0,\gamma_1}:\wzwmodel|_{\gamma_0} \to \wzwmodel|_{\gamma_1}:(\phi_0,z_0) \mapsto (\phi_1,z_0\cdot \mathrm{e}^{2\pi\im S_{\mathrm{WZ}}(\Phi_{\gamma})})\text{,}
\end{equation*}
where $\phi_1: D^2 \to G$ is an arbitrarily chosen smooth map with $\partial\phi_1=\gamma_1$, and $\Phi_{\gamma}:S^2 \to G$ is the following map. On the polar caps $D^2\subset S^2$ we prescribe $\Phi$ with $\phi_0$ (around the north pole with orientation-reversing identification) and  $\phi_1$ (around the south pole with orientation-preserving identification), and on the remaining cylinder $Z \cong [0,1]\times S^1$ by the homotopy $h_\gamma$ of a thin path $\gamma:[0,1] \to LG$ with $\gamma_0=\gamma(0)$ and $\gamma_1=\gamma(1)$.
A different choice of $\phi_1$ gives an equivalent result. If another thin path $\gamma'$ is chosen, then the two paths constitute a rank one loop in $LG$, i.e. a map $\Phi:S^1 \times S^1 \to G$ of rank one, and we have to prove that $S_{\mathrm{WZ}}(\Phi)=0$. This follows as a special case of \cite[Prop. 3.3.1]{waldorf10}. In the next paragraphs we give an independent argument for the vanishing of Wess-Zumino terms for maps $\Phi:\Sigma \to G$ of rank one, mapping a compact oriented surface $\Sigma$ to a connected, simple, simply-connected compact Lie group $G$.

Let $T$ be a maximal torus of $G$, $\mathfrak{t}$ its Lie algebra, and  $\mathfrak{A} \subset \mathfrak{t}$ a closed Weyl alcove, a simplex with vertices $0=\mu_0,...,\mu_r$, where $r$ is the rank of $G$. We let $F_0 \subset \mathfrak{A}$ denote the closed face spanned by $\{\mu_1,...,\mu_r\}$,  $\mathfrak{A}_0 := \mathfrak{A} \setminus F_0$ its complement, and 
\begin{equation*}
U := \{ h\exp\xi h^{-1} \sep  h\in G\text{ and }\xi\in \mathfrak{A}_0 \}
\end{equation*}
the corresponding open subset of $G$, which deformation retracts to $1\in G$, see \cite[Lemma 5.1]{meinrenken1}.  We show below that there exists $g\in G$ such that the left-translated map $\Phi_g: \Sigma \to G$, $\Phi_g(x):=g\Phi(x)$, has its image contained in $U$. Due to the left-invariance of $H$, we have $S_{\mathrm{WZ}}(\Phi)=S_{\mathrm{WZ}}(\Phi_g)$. By composition with the deformation retract one obtains a rank-two map $\tilde \Phi_g: D \to U$ defined on a 3-dimensional manifold $D$ with $\partial D=\Sigma$, such that $\tilde\Phi_g|_{\Sigma}=\Phi_g$. Thus, \begin{equation*}
S_{\mathrm{WZ}}(\Phi)= S_{\mathrm{WZ}}(\Phi_g) = \int_{D} \tilde\Phi_g^{*}H = 0\text{.}
\end{equation*}
Next we prove the existence of an appropriate left translating element $g$.

Recall that $g\in G$ is called \emph{regular} if the conjugacy class of $g$ has maximal dimension. For the class of Lie groups under consideration, an element is regular if and only if it is conjugate to $\exp\xi$ for $\xi$ in the interior of $\mathfrak{A}$. In particular, the open subset $G^{\mathrm{reg}}$ of regular elements is contained in $U$. We define $Q := G \setminus U$, and obtain $Q \subset G\setminus G^{\mathrm{reg}}$. The set $G\setminus G^{\mathrm{reg}}$ is the disjoint union of finitely many submanifolds of $G$ of codimension $\geq 2$ \cite[Page 137 Item (g)]{Duistermaat2000}. Let $Q_{\alpha}$ be one of these submanifolds. We consider the family $g \mapsto \Phi_g$ of left translates, which, as a map $G \times \Sigma \to G$, is a submersion and hence transverse to $Q_{\alpha}$. By the parametric transversality theorem \cite[Chp. 3, Thm. 2.7]{hirsch1} the set $G_{\alpha}:=\{g\in G \sep \Phi_g \text{ is transverse to }Q_{\alpha}\}$ is residual.  But for  dimensional reasons, a rank one map can only be transverse to a submanifold of codimension $\geq 2$ if its image does not intersect $Q_{\alpha}$, i.e. $G_{\alpha}=\{g\in G \sep \Phi_g(\Sigma)\cap Q_{\alpha}=\emptyset \}$.
The intersection of finitely many residual sets is residual, in particular non-empty. Thus, there exists $g\in G$ such that $\Phi_g(\Sigma)$ does not intersect any of the submanifolds $Q_{\alpha}$, i.e. $\Phi_g(\Sigma) \subset U$.
This finishes the proof that  $S_{\mathrm{WZ}}(\Phi)=0$.

By now we have shown that the thin homotopy equivariant structure $d$ on $\mathcal{L}_G$ is well-defined as a bundle isomorphism over $\thinpairs{LG}$. The cocycle condition follows from the identity $S_{\mathrm{WZ}}(\Phi_{\gamma_2})+S_{\mathrm{WZ}}(\Phi_{\gamma_1})=S_{\mathrm{WZ}}(\Phi_{\gamma_2\pcomp\gamma_1})$, if at $\gamma_1(1)=\gamma_2(0)$ the same extension $\phi:D^2 \to G$ is chosen. Multiplicativity follows because in the Polyakov-Wiegmann formula \erf{eq:polwieg} the error term vanishes, as the integral of the 2-form $\rho$ along a thin path $(\gamma,\tau)$ through $L(G \times G)$ gives zero. Compatibility with the fusion product $\lambda$ can be seen by inspection of the occurring integrals. Finally, let us check in some more detail that $d$ symmetrizes $\lambda$. For $(\gamma_1,\gamma_2,\gamma_3)\in PG^{[3]}$, let $\phi_{12},\phi_{23},\phi_{13}:D^2 \to G$ such that $\lambda((\phi_{12},1)\otimes (\phi_{23},1))=(\phi_{13},1)$ holds. This means that $\mathrm{e}^{S_{\mathrm{WZ}}(\Psi)}=1$, with $\Psi:S^2 \to G$ defined as shown in Figure \ref{fig:fusion}. We have to check that
\begin{equation}
\label{eq:checksymm}
\lambda((\phi_{23},1)\cdot r_{\pi} \otimes (\phi_{12},1)\cdot r_{\pi}) = (\phi_{13},1)\cdot r_{\pi}\text{,} 
\end{equation}
where $r_{\pi}\in \diff^{+}(S^1)$ is the rotation by an angle of $\pi$. The definition of the thin homotopy equivariant structure implies $(\phi,z) \cdot r_{\pi}=(\phi \circ r_{\pi},z)$, where on the right hand side $r_{\pi}$ is extended to a rotation of $D^2$. In order to check \erf{eq:checksymm} we have to form the map $\Psi':S^2 \to G$ using  $\phi_{23} \circ r_{\pi}$, $\phi_{12} \circ r_{\pi}$ and $\phi_{13} \circ r_{\pi}$. By inspection, $\Psi'=\Psi \circ r_{\pi}$, where now $r_{\pi}$ is extended to a rotation of $S^2$ about the front-back axis, compare Figures \ref{fig:fusion} and \ref{fig:fusionrot}.
\begin{figure}[h]
\begin{center}
\includegraphics{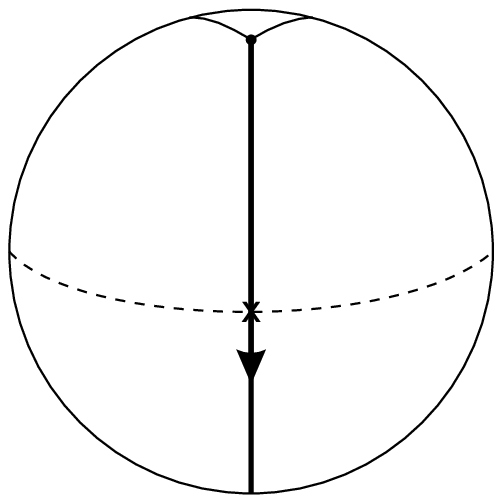}\setlength{\unitlength}{1pt}\begin{picture}(0,0)(202,423)\put(137.90183,455.64378){$\gamma_2$}\put(82.78044,503.95975){$r_\pi(\phi_{23})$}\put(154.17496,504.37814){$r_\pi(\phi_ {12})$}\put(135.75572,546.76800){$g$}\put(136.41194,478.84706){$0$}\end{picture}
\hspace{4em}
\includegraphics{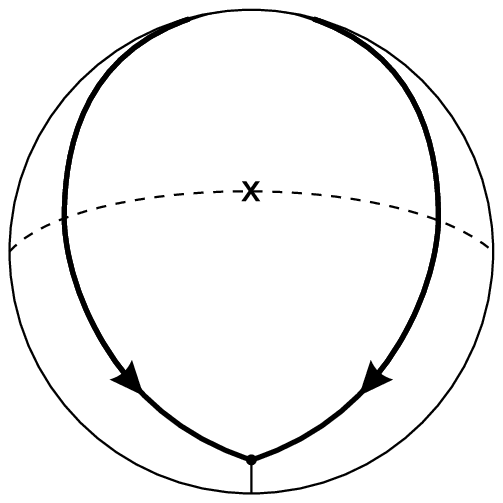}\setlength{\unitlength}{1pt}\begin{picture}(0,0)(408,423)\put(284.11739,458.16661){$\gamma_1$}\put(300.96425,527.60608){\rebox{$r_\pi(\phi_{13})$}}\put(317.72268,426.88460){$h$}\put(330.00718,457.95090){$\gamma_3$}\put(263.90523,492.64753){$\frac{2\pi}{3}$}\put(351.38069,490.87701){$\frac{4\pi}{3}$}\end{picture}
\end{center}
\caption{Front and back view to the map $\Psi'$ defined as in Figure \ref{fig:fusion} but using the maps $\phi_{23} \circ r_{\pi}$, $\phi_{12} \circ r_{\pi}$ and $\phi_{13} \circ r_{\pi}$ instead of $\phi_{12}$, $\phi_{23}$ and $\phi_{13}$. The pictures are precisely the ones of Figure \ref{fig:fusion} rotated by an angle of $\pi$ around the axis that intersects $S^2$ at the two marked points.}
\label{fig:fusionrot}
\end{figure}
Thus, $\mathrm{e}^{S_{\mathrm{WZ}}(\Psi)}=1$ and \erf{eq:checksymm} holds. 
\end{example}

\setsecnumdepth{2}

\section{Features of fusion and thin homotopy equivariance}

In this section we derive some consequences of the presence of a fusion product and a thin homotopy equivariant structure on a central extension. In particular, all results of this section hold for thin fusion extensions, and thus, by our main theorem, for transgressive central extensions.

\label{sec:features}

\subsection{Flat loops and  retraction}

\label{sec:flatloops}

Loops in the image of the map $\flat:PG \to LG:\gamma \mapsto \gamma \lop \gamma$ are called \emph{flat loops}.  Note that $\flat$ is a group homomorphism and  that every constant loop is flat. Suppose $\phi$ is a smoothing map, i.e. it is a smooth map $\phi:[0,1] \to [0,1]$ with $\phi(0)=0$ and $\phi(1)=1$,  locally constant in a neighborhood of $\{0,1\}$, and smoothly homotopic to $\id_{[0,1]}$. \emph{Path retraction} is the map
\begin{equation*}
[0,1] \times PG \to PG: (t,\gamma)\mapsto\phi_\gamma(t)
\end{equation*}
defined by $\phi_{\gamma}(t)(s) := \gamma(t\phi(s))$;  $\phi$ is necessary to guarantee that $\phi_{\gamma}(t)$ has sitting instants. Clearly, $t \mapsto \flat (\phi_{\gamma}(t))$ is a thin path in $LG$; in particular, for every $t\in [0,1]$, $\flat (\phi_{\gamma}(t))$ is thin homotopic to $\flat(\gamma \circ \phi)$, which is in turn thin homotopic to $\flat(\gamma)$.

\begin{proposition}
\label{prop:can}
Suppose  $\mathcal{L}$ is a central extension of $LG$ equipped with a multiplicative fusion product $\lambda$ and a fusive thin homotopy equivariant structure $d$. Then, there exists a unique section $PG \to \mathcal{L}:\gamma \mapsto \can_{\gamma}$  along $\flat$ such that $\lambda(\can_{\gamma},\can_{\gamma})=\can_{\gamma}$.
It has the following properties:
\begin{enumerate}[(i)]

\item
\label{prop:can:neutral}
It is neutral with respect to fusion, i.e. $\lambda(p\otimes \can_{\gamma_2})=p=\lambda(\can_{\gamma_1} \otimes p)$ for all $p\in \mathcal{L}_{\gamma_1 \lop \gamma_2}$.

\item 
\label{prop:can:hom}
It is a group homomorphism, i.e. $\can_{\gamma_1}\cdot \can_{\gamma_2} = \can_{\gamma_1\gamma_2}$.

\item
\label{prop:can:inv}
It is retraction-invariant, i.e. $\can_{\phi_\gamma(t)}=d_{\flat(\gamma),\flat(\phi_\gamma(t))}(\can_{\gamma})$ for all $t\in[0,1]$.

\end{enumerate}
\end{proposition}

\begin{proof}
Two sections $s,s':PG \to \mathcal{L}$ differ by a smooth map $\alpha:PG \to \ueins$, i.e., $s'=\alpha s$. If the sections satisfy  $\lambda(s,s)=s$ and $\lambda(s',s')=s'$, we get $\alpha^2=\alpha$ and so $s=s'$. For the existence, we notice that pulling back $\lambda$ along the diagonal map $PG \to PG^{[3]}$ shows that $\flat^{*}\mathcal{L}$ is trivializable. Let $s: PG \to \mathcal{L}$ be any section. Then, there exists a unique smooth map $\alpha:PG \to \ueins$ such that $\lambda(s \otimes s)=s \cdot \alpha$.  Hence, $\can_{\gamma} := s(\gamma) \cdot \alpha(\gamma)$ has the desired property.

For (i) we have $\lambda(\can_{\gamma_1} \otimes p) = \lambda(\lambda(\can_{\gamma_1} \otimes \can_{\gamma_1}) \otimes p) = \lambda(\can_{\gamma_1} \otimes \lambda (\can_{\gamma_1} \otimes p))$ using the associativity of the fusion product, and since $\lambda(\can_{\gamma_1}\otimes - )$ is an isomorphism, we get $p=\lambda (\can_{\gamma_1} \otimes p)$. Analogously we show neutrality from the right. For (ii)
we have 
\begin{equation*}
\lambda((\can_{\gamma_1}\cdot \can_{\gamma_2}) \otimes (\can_{\gamma_1} \cdot \can_{\gamma_2})) = \lambda(\can_{\gamma_1} \otimes \can_{\gamma_1})\cdot \lambda(\can_{\gamma_2} \otimes \can_{\gamma_2})
= \can_{\gamma_1} \cdot \can_{\gamma_2}
\end{equation*}
 using the multiplicativity of the fusion product; the uniqueness of the section then shows that $\can_{\gamma_1}\cdot \can_{\gamma_2}=\can_{\gamma_1\gamma_2}$. For (iii) we compute with the compatibility of Definition \ref{def:comp} \erf{def:comp:a}
\begin{equation*}
\lambda(d_{\flat(\gamma),\flat(\phi_\gamma(t))}(\can_{\gamma}) \otimes d_{\flat(\gamma),\flat(\phi_\gamma(t))}(\can_{\gamma})) = d_{\flat(\gamma),\flat(\phi_{\gamma}(t))}(\lambda(\can_{\gamma} \otimes \can_{\gamma}))= d_{\flat(\gamma),\flat(\phi_{\gamma}(t))}(\can_{\gamma} )\text{,} 
\end{equation*}
from which the claim follows again from the uniqueness of the section. 
\end{proof}

In particular, the restriction of $\mathcal{L}$ to flat loops is canonically trivializable as a central extension of $PG$ by $\ueins$.

\subsection{Loop concatenation}

\label{sec:loopconcat}

Let $\mathcal{L}$ be a central extension of $LG$ equipped with a  fusion product.
We start with the prototypical situation for loop concatenation: let $\gamma_1,\gamma_2\in PG$ be closed, i.e. loops with sitting instants. We have $(\gamma_1,\id,\gamma_2)\in PG^{[3]}$. The concatenated loop is $\gamma_1\lop\gamma_2$. If $p_1\in \mathcal{L}_{\gamma_1\lop\id}$ and $p_2\in \mathcal{L}_{\id \lop \gamma_2}$, then
\begin{equation*}
\lambda(p_1 \otimes p_2) \in \mathcal{L}_{\gamma_1\lop\gamma_2}\text{;}
\end{equation*}
this lifts loop concatenation from $LG$ to $\mathcal{L}$.
In the following we use a  fusive thin homotopy equivariant structure in order to generalize to arbitrary loops that admit concatenation (not necessarily with sitting instants).

We denote by $LG \times_G^{\infty} LG$ the subset of $LG \times LG$ consisting of pairs $(\tau_1,\tau_2)$ such that
 $\tau_1(1)=\tau_2(1)$ and the concatenation $\prev{\tau_2} \pcomp \tau_1$   is again a smooth loop. Thus, we have a well-defined map
\begin{equation*}
con:LG \times_G^{\infty} LG \to LG: (\tau_1,\tau_2)\mapsto \prev{\tau_2} \pcomp \tau_1 \text{.}
\end{equation*}
If we equip $LG \times_G^{\infty} LG \subset LG \times LG$ with the subspace diffeology (i.e. a map $c:U \to LG \times_G^{\infty} LG$ is a plot if and only if its extension to $LG \times LG$ is smooth), then $con$ is smooth.
\begin{comment}
Indeed, if $c:U \to LG \times_G^{\infty} LG$ is a plot, we let $c_i := \pr_i \circ c$ and have that $U \times S^1\to G :(u,z) \mapsto c_i(u)(z)$ are smooth maps. Then, $con \circ c:U \to LG$ is a plot.
\end{comment}
Further, $con$ is a group homomorphism.

We fix a smoothing map $\phi$ and construct new smooth maps  $\phi_1,\phi_2:[0,1] \to [0,1]$ by setting: 
\begin{equation*}
\phi_1(t) := \begin{cases}
0 & 0\leq t \leq \frac{1}{2} \\
\phi(2t-1) & \frac{1}{2}\leq t \leq 1 \\
\end{cases}
\quand
\phi_2(t) := \begin{cases}
\phi(2t) & 0\leq t \leq \frac{1}{2} \\
1 & \frac{1}{2}\leq t \leq 1\text{.} \\
\end{cases}
\end{equation*}
These cover well-defined smooth maps $\phi_1,\phi_2:S^1 \to S^1$ that are smoothly homotopic to $\id_{S^1}$. 
Thus, if $(\tau_1,\tau_2)\in LG \times_G^{\infty} LG$ with $g:=\tau_1(1)=\tau_2(1)$, then $\tilde\tau_k := \tau_k \circ \phi_k$ is thin homotopic to $\tau_k$ for $k=1,2$; furthermore $\tilde\tau_1 = (\tau_1 \circ \phi) \lop \id_g$ and $\tilde\tau_2 = \id_g \lop (\tau_2 \circ \phi)$.

Suppose we have elements $p_k\in \mathcal{L}_{\tau_k}$ for $k=1,2$.
Using the thin homotopy equivariant structure, we define $\tilde p_k := d_{\tau_k,\tilde \tau_k}(p_k)$, and form the fusion product
\begin{equation*}
\tilde p:= \lambda(\tilde p_1 \otimes \tilde p_2) \in \mathcal{L}_{(\tau_1 \circ \phi) \lop (\tau_2 \circ \phi)}\text{.}
\end{equation*}
The loop $(\tau_1 \circ \phi) \lop (\tau_2 \circ \phi)$ is thin homotopic to $con(\tau_1,\tau_2)$. We obtain an element
\begin{equation*}
p := d_{(\tau_1 \circ \phi) \lop (\tau_2 \circ \phi),con(\tau_1,\tau_2)}(\tilde p)\in \mathcal{L}_{con(\tau_1,\tau_2)}\text{.}
\end{equation*}

\begin{proposition}
Suppose  $\mathcal{L}$ is a central extension of $LG$ equipped with a  fusion product $\lambda$ and a fusive thin homotopy equivariant structure $d$. Then, the assignment $(p_1,p_2) \mapsto p$ defined above is a smooth map
\begin{equation*}
\widetilde{con}:\pr_1^{*}\mathcal{L} \otimes \pr_2^{*}\mathcal{L} \to con^{*}\mathcal{L}
\end{equation*}
over $LG \times_G^{\infty} LG$ that is independent of the choice of the smoothing function. If $\lambda$ and $d$  are multiplicative, then  
\begin{equation*}
\widetilde{con}(p_1 \otimes p_2) \cdot \widetilde{con}(p_1' \otimes p_2') = \widetilde{con}(p_1p_1'\otimes p_2p_2')
\end{equation*}
for all $p_i \in \mathcal{L}_{\tau_i}$, $p_i'\in \mathcal{L}_{\tau_i'}$ with $i=1,2$ and $(\tau_1,\tau_2),(\tau_1',\tau_2')\in LG \times_G^{\infty} LG$. 
\end{proposition}

\begin{proof}
We show the independence of the smoothing function $\phi$. If $\phi'$ is another smoothing function, then $\phi$ and $\phi'$ are smoothly homotopic. We obtain  loops $\tilde\tau'_1$ and $\tilde\tau'_2$ that are thin homotopic to $\tilde\tau_1$ and $\tilde\tau_2$, respectively. The diagram
\begin{equation*}
\alxydim{@C=4em}{&\mathcal{L}_{\tilde\tau_1} \otimes \mathcal{L}_{\tilde\tau_2} \ar[dd]_{d} \ar[r]^-{\lambda} & \mathcal{L}_{(\tau_1 \circ \phi) \lop (\tau_2 \circ \phi)} \ar[dr]^{d} \ar[dd]^{d} \\ \mathcal{L}_{\tau_1} \otimes \mathcal{L}_{\tau_2} \ar[ur]^{d} \ar[dr]_{d} &&& \mathcal{L}_{con(\tau_1,\tau_2)} \\ &\mathcal{L}_{\tilde\tau_1'} \otimes \mathcal{L}_{\tilde\tau_2'} \ar[r]_-{\lambda} & \mathcal{L}_{(\tau_1 \circ \phi') \lop (\tau_2 \circ \phi')} \ar[ur]_{d}}
\end{equation*}
is commutative: the triangular diagrams commute due to the cocycle condition for $d$ (Definition \ref{def:thes} \erf{def:thes:a}), and the rectangular diagram commutes due to the compatibility condition for $\lambda$ and $d$ (Definition \ref{def:comp} \erf{def:comp:a}).
This shows the independence of the choice of $\phi$.

Smoothness can be checked with fixed smoothing function and then follows from the smoothness of the thin structure and the fusion product. In order to see the multiplicativity, we notice that the thin paths from $\tau_i$ to $\tilde\tau_i$ give a thin path in $L(G \times G)$. Thus, the multiplicativity of $d$ and the one of the    fusion product imply the claimed property.
\end{proof}

\subsection{Disjoint-commutativity}

\label{sec:discomm}

A loop $\tau\in LG$ is \emph{supported on an interval} $I \subset S^1$ if $\tau(z)=1$ for all $z\in S^1\setminus I$. 

\begin{theorem}
\label{th:localcomm}
Suppose $\mathcal{L}$ is a central extension of $LG$ equipped with a multiplicative fusion product and a multiplicative, fusive thin homotopy equivariant structure. Then,
$\mathcal{L}$ is disjoint commutative, i.e. if $\tau_1,\tau_2 \in LG$ are supported on disjoint intervals, and  $p_1 \in \mathcal{L}_{\tau_1}$ and $p_2\in \mathcal{L}_{\tau_2}$,  then $p_1\cdot p_2=p_2\cdot p_1$.
\end{theorem}

Theorem \ref{th:localcomm} applies to the  universal central extension $\wzwmodel$ of a compact, simple, connected, simply-connected  Lie group $G$, as seen in Example \ref{ex:wzwmodel}. In that case, disjoint commutativity was known to hold, verified by a direct calculation using the Mickelsson model \cite[Lemma 3.1]{Gabbiani1993}.

For the proof Theorem \ref{th:localcomm}, we start with the following prototypical situation. Suppose $\tau_1,\tau_2\in PG$ are loops based at $1\in G$. Then, $\tau_1 \lop \id_1$ and $\id_1 \lop \tau_2$ are loops supported on disjoint intervals. In particular, \begin{equation*}
(\tau_1 \lop \id_1)(\id_1 \lop \tau_2) = (\id_1 \lop \tau_2)(\tau_1 \lop \id_1)\text{.}
\end{equation*}

\begin{proposition}
\label{prop:localcomm}
We have $p_1 \cdot p_2= p_2 \cdot p_1$ for all  $p_1\in \mathcal{L}_{\tau_1\lop \id_1}$ and $p_2\in\mathcal{L}_{\id_1 \lop \tau_2}$. 
\end{proposition}

In order to prove this,  we note that the multiplication in this particular situation coincides with the fusion product.

\begin{lemma}
\label{lemma2}
$p_1\cdot p_2 = \lambda(p_1\otimes p_2)$.
\end{lemma}

\begin{proof}
Proposition \ref{prop:can} \erf{prop:can:hom} implies  $\can_{\id_1}=1$. Thus,
over the triple $(\tau_1,\id_1,\id_1) \in  PG^{[3]}$ we can write $p_1=\lambda(p_1 \otimes 1)$. Likewise, over $(\id_1,\id_1,\tau_2)$ we have $p_2=\lambda(1\otimes p_2)$. Multiplicativity of the fusion product then shows the claim:
\begin{equation*}
p_1\cdot p_2 =\lambda(p_1 \otimes 1) \cdot \lambda(1\otimes p_2) = \lambda(p_1 \cdot 1 \otimes 1 \cdot p_2)=\lambda(p_1 \otimes p_2)\text{.}
\end{equation*}
\end{proof}

\begin{proof}[Proof of Proposition \ref{prop:localcomm}]
That the thin homotopy equivariant structure symmetrizes the fusion product (Definition \ref{def:comp} \erf{item:symmetrizes}) means, over the triple $(\tau_1,\id_1,\tau_3)$,
\begin{equation*}
d_{\tau_1 \lop \tau_3,\prev{\tau_3} \lop \prev{\tau_1}}(\lambda(p_1 \otimes p_2)) = \lambda(d_{\id_1 \lop \tau_3,\prev{\tau_3} \lop \id_1}(p_2) \otimes d_{\tau_1 \lop \id_1,\id_1 \lop \prev{\tau_1}}(p_1))\text{.}
\end{equation*}
We can apply  Lemma \ref{lemma2} on the left (to the given data) and on the right  (to the pair $(\prev{\tau_3},\prev{\tau_1})$ and the elements $d_{\id_1 \lop \tau_3,\prev{\tau_3} \lop \id_1}(p_2)$ and $d_{\tau_1 \lop \id_1,\id_1 \lop \prev{\tau_1}}(p_1)$) and obtain
\begin{equation*}
d_{\tau_1 \lop \tau_3,\prev{\tau_3} \lop \prev{\tau_1}}(p_1 \cdot p_2) = d_{\id_1 \lop \tau_3,\prev{\tau_3} \lop \id_1}(p_2) \cdot d_{\tau_1 \lop \id_1,\id_1 \lop \prev{\tau_1}}(p_1)\text{.} \end{equation*}
Now we want to use the multiplicativity of the thin homotopy equivariant structure. We claim that $((\id_1 \lop \tau_3,\tau_1 \lop \id_1),(\prev{\tau_3} \lop \id_1,\id_1 \lop \prev{\tau_1})) \in \thinpairs{L(G \times G)}$. Indeed, let $r_{-\pi}:S^1 \to S^1$ denote the rotation  by an angle of $-\pi$. Then, $(\id_1 \lop \tau_3)\circ r_{-\pi}= \prev{\tau_3} \lop \id_1$ and $(\tau_1 \lop \id_1) \circ r_{-\pi}=\id_1 \lop \prev{\tau_1}$. Fixing a path $\varphi_t$ from $\varphi_0=\id_{S^1}$ to $\varphi_1=r_{-\pi}$ we obtain a path
\begin{equation*}
((\id_1 \lop \tau_3)\circ \varphi_t,(\tau_1 \lop \id_1) \circ \varphi_t)
\end{equation*} 
in $L(G \times G)$, which is thin as it factors through $\varphi_t(z)\in S^1$. Thus, the multiplicativity of the thin structure now implies $p_1 \cdot p_2= p_2 \cdot p_1$.
\end{proof}

\begin{proof}[Proof of Theorem \ref{th:localcomm}]
We consider two general loops $\tau_1,\tau_2 \in LG$ with disjoint supports $I,J \subset S^1$. There is an orientation-preserving diffeomorphism $\varphi:S^1 \to S^1$ such that $\varphi^{-1}(I) \subset (\frac{1}{2},1)$ and $\varphi^{-1}(J) \subset (0,\frac{1}{2})$. Then, $\tau_1 \circ \varphi = \tau_1' \lop \id_1$ and $\tau_2 \circ \varphi = \id_1 \lop \tau_2'$, where $\tau_1'(\mathrm{e}^{2\pi\im t}) = (\tau_1\circ \varphi)(\mathrm{e}^{\pi\im(1-t)})$ and $\tau_2'(\mathrm{e}^{2\pi\im t}) = (\tau_2 \circ \varphi)(\mathrm{e}^{\pi\im t})$. The map $\varphi$ is homotopic to $\id_{S^1}$ via a family $\varphi_t$ with $\varphi_0=\id_{S^1}$ and $\varphi_1=\varphi$. The path $(\tau_1 \circ \varphi_t, \tau_2 \circ \varphi_t)=(\tau_1,\tau_2) \circ \varphi_t$ in $L(G \times G)$ is thin. 
If $p_1\in \mathcal{L}_{\tau_1}$ and $p_2 \in \mathcal{L}_{\tau_2}$, then $d_{\tau_1,\tau_1'\lop \id_1}(p_1)\in \mathcal{L}_{\tau_1'\lop\id_1}$ and $d_{\tau_2,\id_1\lop\tau_2'}(p_2)\in\mathcal{L}_{\id_1\lop\tau_2'}$. Proposition \ref{prop:localcomm} applies and
\begin{equation*}
d_{\tau_1,\tau_1'\lop \id_1}(p_1)\cdot d_{\tau_2,\id_1\lop\tau_2'}(p_2)= d_{\tau_2,\id_1\lop\tau_2'}(p_2) \cdot d_{\tau_1,\tau_1'\lop \id_1}(p_1)\text{.}
\end{equation*} 
We use the multiplicativity of the thin structure: on the left for the thin path $(\tau_1,\tau_2) \circ \varphi_t$ and on the right for the thin path $(\tau_2,\tau_1) \circ \varphi_t$, and obtain
\begin{equation*}
d_{\tau_1\tau_2,(\tau_1'\lop \id_1)(\id_1\lop\tau_2')}(p_1\cdot p_2)= d_{\tau_2\tau_1,(\id_1\lop\tau_2')\cdot(\tau_1'\lop \id_1)}(p_2\cdot p_1)\text{.}
\end{equation*}
We  have
\begin{equation*}
(\tau_1'\lop \id_1)(\id_1\lop\tau_2')=(\id_1\lop\tau_2')\cdot(\tau_1'\lop \id_1)=\tau_1'\lop\tau_2'
\quand
\tau_1 \tau_2=\tau_2\tau_1\text{,}
\end{equation*}
and as $d_{\tau_1\tau_2,\tau_1'\lop\tau_2'}$ is an isomorphism, we get $p_1 \cdot p_2= p_2 \cdot p_1$.
\end{proof}

\setsecnumdepth{1}

\section{Integrable thin homotopy equivariant structures}

\label{sec:fusextth}

A thin homotopy equivariant structure on a central extension $\mathcal{L}$ can be induced from certain connections on (the underlying principal $\ueins$-bundle of) $\mathcal{L}$. 

\begin{definition}
\label{def:thin}
\begin{enumerate}[(a)]
\item 
A connection $\nu$ on  $\mathcal{L}$ is called \emph{thin}, if two thin  paths $\gamma_1,\gamma_2$ with common initial point $\gamma_1(0)=\gamma_2(0)$ and common end point $\gamma_1(1)=\gamma_2(1)$ induce the same parallel transport, $pt^{\nu}_{\gamma_1}=pt^{\nu}_{\gamma_2}$. 

\item
A thin connection $\nu$ on  $\mathcal{L}$ is called \emph{superficial}, if two loops $\tau_1,\tau_2\in LLG$ have the same holonomy whenever they are homotopic via a map $h\maps [0,1] \times S^1 \times S^1 \to G$ of rank at most two.

\end{enumerate}
\end{definition}

\noindent
Every thin connection $\nu$ induces a thin homotopy equivariant structure $d^{\nu}$ by
$d^{\nu}_{\tau_0,\tau_1} \df pt^\nu_\tau$,
where $\tau$ is an arbitrary thin path from $\tau_0$ to $\tau_1$, see \cite[Lemma 3.1.5]{waldorf11}. Superficiality is a property of connections in the image of transgression, and will be relevant in Section \ref{sec:transgression}.

\begin{definition}
\label{def:compatible}
Suppose $\mathcal{L}$ is equipped with a fusion product $\lambda$ and a thin connection $\nu$. 
\begin{enumerate}[(a)]

\item 
$\nu$ is  \emph{compatible} with $\lambda$, if the bundle morphism $\lambda$ over $PG^{[3]}$ is  connection-preserving. 

\item
$\nu$ \emph{symmetrizes} $\lambda$, if $d^{\nu}$ is symmetrizing.

\item
$\nu$ is  \emph{fusive} if it is compatible and symmetrizing.

\end{enumerate}
\end{definition}

\noindent
If $\nu$ is a thin and fusive connection, then $d^{\nu}$ is fusive with respect to $\lambda$. The question whether or not a given thin homotopy equivariant structure can be induced from a thin connection gives rise to the following definition.

\begin{definition}
\label{def:thinstructure}
A   thin homotopy equivariant structure $d$ is called \emph{thin structure}, if there exists a  superficial connection $\nu$ such that $d^{\nu}=d$. In the presence of a fusion product, it is called \emph{fusive thin structure} if $\nu$ can be chosen fusive. 
\end{definition}

It remains to discuss multiplicativity in the setting of thin structures. This requires some attention, because it is not clear which multiplicativity condition one should impose on a connection $\nu$. First of all we note that every connection $\nu$ on $\mathcal{L}$ determines a 1-form $\varepsilon_{\nu}\in\Omega^1(LG \times LG)$ by 
\begin{equation}
\label{eq:erroroneform}
 \nu_{p_1}(X_1) + \nu_{p_2}(X_2)=\nu_{p_1p_2}(p_1X_2 + X_1p_1) + \varepsilon_{\nu}|_{\tau_1,\tau_2}(\mathcal{X}_1,\mathcal{X}_2)
\end{equation}
for all $\tau_1,\tau_2\in LG$, $\mathcal{X}_i\in T_{\tau_i}LG$, as well as $p_i \in \mathcal{L}_{\tau_i}$ and $X_i \in T_{p_i}\mathcal{L}$ such that $p_{*}(X_i)=\mathcal{X}_i$. We call $\varepsilon_{\nu}$ the \emph{error 1-form} of $\nu$, it can be seen as a measure for the non-multiplicativity of $\nu$. We want to impose a multiplicativity condition for the connection $\nu$ by requiring that $\varepsilon_{\nu}$ admits a path splitting in the following sense.

\begin{definition}
\label{def:pathsplitting}
Let $X$ be a smooth manifold, $k\in \N$, and $\varepsilon\in \Omega^{k}(LX)$. A \emph{path splitting} of $\varepsilon$ is a $k$-form $\kappa\in\Omega^k(PX)$ such that 
$\lop^{*}\varepsilon = \pr_2^{*}\kappa - \pr_1^{*}\kappa$ on $PX^{[2]}$.
\end{definition}

\label{par:pathsplittingmult}
We formulate two additional conditions for path splittings that will be required, too. We recall that for every (Fréchet or diffeological) Lie group $K$ and every $k \in \N$ we have a complex
\begin{equation}
\label{eq:multiplicative}
\alxydim{}{0\ar[r] & \Omega^{k}(K) \ar[r] & \Omega^{k}(K \times K) \ar[r] & \hdots }
\end{equation}
whose differential $\Delta$ is the alternating sum over the pullbacks along the  face maps of the simplicial manifold $BK$. A form in the kernel of $\Delta$ is called \emph{multiplicative}. 

For $K=LG$, \erf{eq:erroroneform} implies that $\varepsilon_{\nu}$ is multiplicative: $\Delta\varepsilon_{\nu}=0$. 

For $K=PG$, it makes sense to require that path splittings of $\varepsilon_{\nu}$ are multiplicative.

\noindent
\label{par:pathsplittingcon}
For the second condition,
we recall from Section \ref{sec:flatloops} that for $\gamma\in PX$ and a smoothing function $\phi$ we have a retraction
$\phi_{\gamma}:[0,1] \to PX$ with $\phi_{\gamma}(0)=\id_{\gamma(0)}$ and $\phi_{\gamma}(1)=\gamma \circ \phi$. A 1-form $\kappa\in\Omega^1(PX)$ is called \emph{contractible}, if\begin{equation*}
\int_{\phi_{\gamma}}\kappa=0
\end{equation*}
for all $\gamma\in PX$ and some (and thus all) smoothing functions $\phi$. 

Before continuing, let us try to elucidate path splittings with the following example.

\begin{example}
\label{ex:pathsplitting}
For $S$ a $l$-dimensional compact oriented smooth  manifold, possibly with boundary, we denote by $\ev:C^{\infty}(S,X)\times S \to X$ the evaluation map, and let
\begin{equation*}
\tau_{S}: \Omega^{k}(X) \to \Omega^{k-l}(C^{\infty}(S,X)): \rho \mapsto \int_{S}\ev^{*}\rho
\end{equation*}
be the usual transgression of differential forms to the mapping space. 
\begin{comment}
For $S=S^1$, this means
\begin{equation*}
\tau_S(\rho)|_{\tau}(X_{2},...,X_k) = \int_{0}^1   \rho_{\tau(z)}(\partial_z\tau(z),X_2,...,X_k) \mathrm{d}z\text{.}
\end{equation*}
\end{comment}
For $\rho\in \Omega^{k+1}(X)$ we set  $\varepsilon:= \tau_{S^1}(\rho)\in\Omega^k(LX)$ and  $\kappa := \tau_{[0,1]}(\rho)|_{PX} \in \Omega^{k}(PX)$. We claim:
\begin{itemize}

\item 
$\kappa$ is a path splitting for $\varepsilon$. To see this, consider a path $\gamma:[0,1] \to PG^{[2]}$, with $\gamma=(\gamma_1,\gamma_2)$ and the associated homotopies $h_{\gamma_i}: [0,1]^2 \to G$; $h_{\gamma_i}(t,s)=\gamma_i(t)(s)$. Note that
\begin{equation*}
\int_{\gamma_i}\kappa = \int_{0}^1 \kappa_{\gamma_i(t)}(\partial_t\gamma_i(t)) \mathrm{d}t=\int_0^1 \int_0^1 \rho_{\gamma_i(t)(s)}(\partial_s\gamma_i(t)(s),\partial_t\gamma_i(t)(s))\mathrm{d}s\mathrm{d}t =-\int_{h_{\gamma_i}}\rho\text{.}
\end{equation*}
We have
\begin{equation*}
\int_{\gamma}\lop^{*}\varepsilon=\int_0^1 \int_0^1 \rho_{(\lop\circ \gamma)(t)(s)}(\partial_s(\lop\circ\gamma)(t)(z),\partial_t(\lop\circ\gamma)(t)(s))  \mathrm{d}s\mathrm{d}t\text{.}
\end{equation*}
Splitting the integral over $s$ in two parts (from $0$ to $\frac{1}{2}$) and (from $\frac{1}{2}$ to $1$), expressing it in terms of $\gamma_1$ and $\gamma_2$, and reparameterizing, we get
\begin{equation*}
\int_{\gamma}\lop^{*}\varepsilon = \int_{h_{\gamma_1}}\rho- \int_{h_{\gamma_2}}\rho=\int_{\gamma} \pr_2^{*}\kappa - \pr_1^{*}\kappa\text{.}
\end{equation*}
But two 1-forms  coincide if their integrals along all paths coincide; this shows $\lop^{*}\varepsilon=\pr_2^{*}\kappa-\pr_1^{*}\kappa$.
\begin{comment}
The full calculation is:
\begin{eqnarray*}
\int_{\gamma}\lop^{*}\varepsilon &=& \int_{\lop\circ \gamma}\varepsilon\\&=& \int_0^1 \varepsilon_{(\lop\circ\gamma)(t)}(\partial_t(\lop\circ\gamma)(t))\mathrm{d}t \\&=&\int_0^1 \int_0^1 \rho_{(\lop\circ \gamma)(t)(s)}(\partial_s(\lop\circ\gamma)(t)(z),\partial_t(\lop\circ\gamma)(t)(s))  \mathrm{d}s\mathrm{d}t
\\&=&\int_0^1 \left\lbrace\int_0^{\frac{1}{2}} \rho_{\gamma_1(t)(2s)}(\partial_s\gamma_1(t)(2s),\partial_t\gamma_1(t)(2s))\mathrm{d}s+\int_{\frac{1}{2}}^{1} \rho_{\gamma_2(t)(2-2s)}(\partial_s\gamma_2(t)(2-2s),\partial_t\gamma_2(t)(2-2s))\mathrm{d}s \right \rbrace \mathrm{d}t
\\&=&\int_0^1 \left\lbrace\int_0^{1} \rho_{\gamma_1(t)(s)}(\partial_s\gamma_1(t)(s),\partial_t\gamma_1(t)(s))\mathrm{d}s-\int_{0}^{1} \rho_{\gamma_2(t)(s)}(\partial_s\gamma_2(t)(s),\partial_t\gamma_2(t)(s))\mathrm{d}s \right \rbrace \mathrm{d}t
\\&=&- \int_{\gamma_2^{\vee}}\rho \k + \int_{\gamma_1^{\vee}}\rho \\&=&\int_{\gamma_2}\kappa - \int_{\gamma_1}\kappa \\&=& \int_{\gamma} \pr_2^{*}\kappa - \pr_1^{*}\kappa
\end{eqnarray*}
\end{comment}

\item
If $X$ is a Lie group and $\rho$ is multiplicative, then $\kappa$ and $\varepsilon$ are multiplicative, too.
Indeed, $\tau_S$ is linear and natural with respect to smooth maps between smooth manifolds, and so commutes with the differential $\Delta$.

\item
If $k=1$, then $\kappa$ is contractible: if $\gamma \in PX$, then we have
\begin{equation*}
\int_{\phi_{\gamma}} \kappa = \int_{\phi_{\gamma}} \int_{[0,1]} \ev^{*}\rho=\int_{[0,1]^2} (\ev \circ (\phi_{\gamma} \times \id))^{*}\rho=0\text{,} \end{equation*}
because $(\ev \circ (\phi_{\gamma} \times \id))(t,s)=\gamma(t\phi(s))$  is a rank one map.

\end{itemize}
\end{example}

One can show that the error 1-form $\varepsilon_{\nu} \in \Omega^1(LG \times LG)$ of a fusive connection $\nu$ on $\mathcal{L}$ admits a path splitting, and  for compact groups $G$ even a multiplicative path splitting.
However, I do not know conditions that would guarantee the existence of a \emph{contractible} path splitting. This constitutes our multiplicativity condition.

\begin{definition}
\label{def:mfthinstructure}
A multiplicative and fusive  thin homotopy equivariant structure $d$ is called  \emph{mul\-ti\-pli\-ca\-ti\-ve and fusive thin structure}, if there exists a fusive superficial connection $\nu$ with $d^{\nu}=d$, whose error 1-form $\varepsilon_{\nu}$ admits a multiplicative and contractible path splitting. \end{definition}

\begin{definition}
\label{def:thinfusionext}
Let $G$ be a Lie group and $LG = C^{\infty}(S^1,G)$ be its loop group. 
\begin{enumerate}[(a)]
\item 
A \emph{thin fusion extension} of $LG$ is a central extension
\begin{equation*}
1 \to \ueins \to \mathcal{L} \to LG \to 1
\end{equation*}
together with a multiplicative fusion product and a multiplicative and fusive thin structure.

\item
An \emph{isomorphism} between thin  fusion extensions is a smooth isomorphism between central extensions that is fusion-preserving and thin.

\end{enumerate}
\end{definition}

\noindent
Thin fusion extensions form a category that we denote  by $\fusextth{LG}$.

Before coming to examples, we shall investigate the following interesting  feature of a thin fusion extension (Proposition \ref{prop:splitting} below). We recall that the Lie algebra of $LG$ is $L\mathfrak{g} = C^{\infty}(S^1,\mathfrak{g})$ with all operations defined pointwise. We denote the Lie algebra of a central extension  $\mathcal{L}$ by $\mathfrak{l}$. It is a central extension
\begin{equation*}
\alxydim{}{0 \ar[r] & \R \ar[r] & \mathfrak{l} \ar[r]^-{p_{*}} & L\mathfrak{g} \ar[r] & 0}
\end{equation*}
of Fréchet Lie algebras. 
We recall that a \emph{splitting} is a continuous linear map $\delta: L\mathfrak{g} \to \mathfrak{l}$ such that $p_{*} \circ \delta=\id_{L\mathfrak{g}}$.  Every connection $\nu$ on $\mathcal{L}$ induces -- via horizontal lift -- a splitting $\delta_{\nu}$.

An element  $\mathcal{X} \in L\mathfrak{g}$ is called \emph{linear loop}, if there exist a smooth map $f:S^1 \to \R$ and $X\in \mathfrak{g}$ such that $\mathcal{X}(z) = f(z)X$ for all $z\in S^1$.
The linear loops span $L\mathfrak{g}$. Every linear loop $\mathcal{X}$ can be represented -- as a tangent vector -- by a \emph{thin} curve, namely by  $\gamma_{\mathcal{X}}:\R \to LG: t \mapsto \exp(t\mathcal{X})$, with $\gamma_{\mathcal{X}}(t)(z)=\exp(tf(z)X)$. Note that $t \mapsto (1,\gamma_{\mathcal{X}}(t))$ is a smooth curve in $\thinpairs{LG}$. Thus, a thin homotopy equivariant structure $d$ on $\mathcal{L}$ produces a smooth curve $d_{\mathcal{X}}:\R \to \mathcal{L}:t \mapsto d_{1,\gamma_{\mathcal{X}}(t)}(1)$. \begin{comment}
In fact, $\R \to \thinpairs{LG}:t\mapsto (1,\gamma_{\mathcal{X}}(t))$ is a plot.  
\end{comment}

\begin{comment}
If $X_1,...,X_k$ is a basis of $\mathfrak{g}$, and $\mathcal{X}:S^1 \to \mathfrak{g}$ is smooth, then  $\mathcal{X}=f_1X_1 + ... + f_k X_k$.  
The subset  of linear loops is closed under scalar multiplication but not under sum. 
\end{comment}

\begin{lemma}
\label{lem:deltanu}
Suppose $\mathcal{L}$ is a central extension with a thin homotopy equivariant structure $d$. 
Let $\nu$ be a thin connection on $\mathcal{L}$ with $d^{\nu}=d$. Then, the splitting $\delta_{\nu}$ is $\diff^{+}(S^1)$-equivariant and satisfies
\begin{equation*}
\delta_{\nu}(\mathcal{X}) = \ddt 1t0 d_{\mathcal{X}}(t)
\end{equation*} 
for all linear loops $\mathcal{X}$. 
\end{lemma}

\begin{proof}
We calculate
\begin{equation}
\label{eq:deltanu}
\delta_{\nu}(\mathcal{X}) = \ddt 1t0 pt_{\gamma_{\mathcal{X}}}^{\nu}(1,t)=\ddt 1t0 d_{1,\gamma_{\mathcal{X}}(t)}(1) = \ddt 1t0 d_{\mathcal{X}}(t)\text{.}
\end{equation}
Next we consider $\varphi \in \diff^{+}(S^1)$  and a linear loop $\mathcal{X}$. Then, $\exp(t\mathcal{X})\circ \varphi=\exp(t(\mathcal{X}\circ \varphi))$ is thin homotopic to $\exp(t\mathcal{X})$, and we have
\begin{equation*}
d_{\mathcal{X}}(t)\cdot \varphi =  d_{1,\exp(t\mathcal{X})}(1)\cdot \varphi = d_{\exp(t\mathcal{X}),\exp(t(\mathcal{X}\circ \varphi))}(d_{1,\exp(t\mathcal{X})}(1)) = d_{1,\exp(t(\mathcal{X}\circ \varphi))}(1)=d_{\mathcal{X}\circ \varphi}(t)\text{.}
\end{equation*}
Taking derivatives and using \erf{eq:deltanu}, we obtain $\delta_{\nu}(\mathcal{X})\cdot \varphi = \delta_{\nu}(\mathcal{X} \circ \varphi)$.
Since the $\diff^{+}(S^1)$-actions on $L\mathfrak{g}$ and $\mathfrak{l}$ are linear, and $L\mathfrak{g}$ is spanned by the linear loops, we conclude that $\delta_{\nu}$ is equivariant.
\end{proof}

In case of a thin fusion extension, we obtain: 

\begin{proposition}
\label{prop:splitting}
Let $\mathcal{L}$ be a thin fusion extension with thin structure $d$. Then, there exists a unique splitting $\delta: L\mathfrak{g} \to \mathfrak{l}$ of the  Lie algebra extension, such that 
\begin{equation*}
\delta(\mathcal{X}) = \ddt 1t0 d_{\mathcal{X}}(t)
\end{equation*}
for all linear loops $\mathcal{X}$.  Moreover, this splitting is $\diff^{+}(S^1)$-equivariant. 
\end{proposition}

\begin{proof}
Uniqueness follows because the linear loops span $L\mathfrak{g}$. Existence uses the existence of a thin connection $\nu$ on $\mathcal{L}$ such that $d=d^{\nu}$. The corresponding splitting $\delta_{\nu}$ has the claimed properties by Lemma \ref{lem:deltanu}. 
\end{proof}

In the remainder of this section we discuss three examples. 

\begin{example}
\label{ex:wzwmodelcon}
The universal central extension $\wzwmodel$ of a compact, simple, connected, simply-connected Lie group $G$ is a thin fusion extension. Since $\wzwmodel$ is universal, it follows that \emph{every} central extension of $LG$ is a thin fusion extension.  In the model of Example \ref{ex:wzwmodel} we obtain a connection $\nu$ by declaring its parallel transport along a path $\gamma:[0,1] \to LG$ via
\begin{equation}
\label{eq:wzwmodelconnection}
pt_{\gamma}:\wzwmodel|_{\gamma(0)} \to \wzwmodel|_{\gamma(1)}:(\phi_0,z_0)\mapsto (\phi_1,z_0\cdot \mathrm{e}^{2\pi \im S_{\mathrm{WZ}}(\Phi_{\gamma})})\text{,}
\end{equation}
where $\phi_1:D^2 \to G$ is arbitrarily chosen such that $\partial\phi_1=\gamma(1)$, and $\Phi_{\gamma}$ is obtained from $\gamma$, $\phi_0$ and $\phi_1$ exactly as described in the definition of the thin homotopy equivariant structure $d$ on $\wzwmodel$. In order to show that \erf{eq:wzwmodelconnection} defines a connection, it suffices to check (see \cite[Theorem 5.4]{schreiber3}):
\begin{enumerate}[(a)]

\item 
It is compatible with the concatenation of paths: this is obvious.

\item
It depends only on   the thin homotopy class of the path, i.e. if $\gamma$ and $\gamma'$ are paths in $LG$ with common initial loop and common end loop, and $h: [0,1]^2 \to LG$ is a smooth map with $h(0,t)=\gamma(t)$ and $h(1,t)=\gamma'(t)$ (i.e. $h$ is a fixed-ends-homotopy) and with the property that $\int h^{*}\omega = 0$ for all 2-forms $\omega\in \Omega^2(LG)$ (i.e. $h$ is thin), then $pt_{\gamma_1}=pt_{\gamma_2}$. That this is the case can be seen by expressing the difference of the parallel transport maps  \erf{eq:wzwmodelconnection} as the integral of $\omega := \tau_{S^1}(H)$ along $h$, which thus vanishes. 

\item
It depends smoothly on the path. This can be checked on smooth one-parameter family of paths, for which smoothness follows from the one of the integral of differential forms. 

\end{enumerate}
Thus, we have a connection $\nu$ on $\wzwmodel$. It is straightforward to see that it is  compatible with the fusion product $\lambda$. 
We have already seen in Example \ref{ex:wzwmodel} that for a thin path $\gamma$ the parallel transport $pt_{\gamma}$   is independent of the choice of the thin path: this shows that $\nu$ is thin and induces $d$.
It is also superficial: if two loops $\tau_1,\tau_2\in LLG$ are homotopic via a homotopy $h$, then the difference between their holonomies is given by $\exp 2\pi \im \int_hH$, where $h$ is the homotopy. When $h$ has rank two, the difference vanishes. The curvature of $\nu$ is $\tau_{S^1}(H)$, and the error 1-form is $\varepsilon_{\nu} = \tau_{S^1}(\rho)$. Example \ref{ex:pathsplitting} shows that $\varepsilon_{\nu}$ has a multiplicative and contractible path splitting. \end{example}

\begin{example}
\label{ex:thinfusionextriv}
For any Lie group $G$, the central extension $\mathcal{L}_P=\ueins \times LG$ with the group structure defined from the holonomy $\eta$ of a principal $\ueins$-bundle $P$ over $G \times G$, the trivial fusion product and the trivial thin structure (see Example \ref{ex:trivial}), is a thin fusion extension. Indeed, an integrating connection $\nu$ can be obtained from any 2-form $\omega\in \Omega^2(G)$ by $\nu :=\tau_{S^1}(\omega)\in\Omega^1(LG)$, e.g. $\omega=0$ works. It is easy to see that this gives a fusive superficial connection, and  that it induces the trivial thin structure, see \cite[Proposition 3.1.8]{waldorf11}.  
\end{example}

\begin{example}
\label{ex:notthinfusion}
We construct a central extension that cannot be equipped with the structure of a thin fusion extension. 
\begin{comment}
It is clear that $G$ cannot be taken to be compact, simple, connected and simply-connected, since in that case there is a universal central extension which is a thin fusion extension; it follows then that all central extensions of $LG$ are thin fusion extensions. 
\end{comment}
We work with $G=\ueins$, and consider $\mathcal{L}=\ueins \times L\ueins$ equipped with the group structure induced by the following 2-cocycle $\eta: L\ueins \times L\ueins \to  \ueins$. If $\tau \in L\ueins$, we denote by $n\in \Z$ the winding number of $\tau$, and find a smooth map $f:\R \to \R$ such that $f(t+1)=f(t)+n$ and $\tau=\mathrm{e}^{2\pi \im f}$. Note that $f$ is determined by $\tau$ up to a shift by a constant $z\in \Z$. We define for $f:\R \to \R$ the average
\begin{equation*}
\widehat f := \int_0^1 f(s)\mathrm{d}s\text{.}
\end{equation*}
We define for $\alpha,\beta,\gamma\in \R$:
\begin{equation*}
\eta(\tau_1,\tau_2)=\exp 2\pi \im \left ( \alpha \int_0^1 f_1(s)f_2'(s)\mathrm{d}s + \beta (n_1 \widehat f_2 + n_2 \widehat f_1)+\gamma n_1f_2(0) \right )\text{.} 
\end{equation*}
It is straightforward to check that the cocycle condition is satisfied for arbitrary values of $\alpha,\beta,\gamma$. 
\begin{comment}
Note that $\tau_1\tau_2=\exp 2\pi\im (f_{12})$ for $f_{12}=f_1+f_2$ and the winding number of $\tau_1\tau_2$ is $n_{12}=n_1+n_2$. The exponent of $\eta(\tau_1,\tau_2\tau_3)\eta(\tau_2,\tau_3)$ is
\begin{eqnarray*}
&& \hspace{-5em}\alpha \int_0^1 f_1(s)(f_2'(s)+f_3'(s))+f_2(s)f_3'(s)\mathrm{d}s \\ &&+ \beta(n_1 \widehat{f_2+f_3}+(n_2+n_3)\widehat f_1+n_2\widehat f_3 + n_3\widehat f_2)+\gamma( n_1(f_2(0)+f_3(0)) n_2f_3(0))
\end{eqnarray*}
and the exponent of $\eta(\tau_1\tau_2,\tau_3)\eta(\tau_1,\tau_2)$ is
\begin{eqnarray*}
&&\hspace{-5em} \alpha \int_0^1 (f_1(s)+f_2(s))f_3'(s)+f_1(s)f_2'(s) \mathrm{d}s
\\&& + \beta((n_1+n_2)\widehat f_3 + n_3(\widehat {f_1+f_2})+n_1 \widehat f_2 + n_2 \widehat f_1)+\gamma((n_1+n_2)f_3(0)+n_1f_2(0))\text{.}
\end{eqnarray*}
Since $\widehat{f_1+f_2}= \widehat f_1+\widehat f_2$, these exponents cancel on the nose, before exponentiation. 
\end{comment}
We remark that the 2-cocycle $\eta$ is normalized, i.e. $\eta(1,1)=1$, for all parameters.
We have to assure that $\eta$ is well-defined under shifting $f_k$ by integers $z_k$. We get
\begin{eqnarray*}
 \int_0^1 (f_1(s)+z_1)(f_2+z_2)'(s)\mathrm{d}s &=&  \int_0^1 f_1(s)f_2'(s)\mathrm{d}s+z_1n_2 \\
n_1 \widehat {f_2+z_2} + n_2 \widehat {f_1 + z_1} &=& n_1\widehat f_2 + n_2 \widehat f_1+n_1 z_2+ n_2 z_1
\\
n_1(f_2+z_2)(0)&=& n_1f_2(0)+n_1z_2\text{.}
\end{eqnarray*}
Note that all differences that arise are integers. We see two options to obtained well-definedness:
\begin{enumerate}

\item 
We choose the constants $\alpha,\beta,\gamma$ such that all differences cancel: $\beta=-\gamma$ and $\alpha=\gamma$. Then we have an $\R$-family of well-defined 2-cocycles. The corresponding central extensions are denoted by $\mathcal{L}_{\R}(\gamma)$. 

\item
We let $\alpha,\beta,\gamma\in \Z$ be arbitrary integers. Then, all differences vanish separately under exponentiation. This gives a $\Z^3$-family of well-defined 2-cocycles. The corresponding central extensions are denoted by  $\mathcal{L}_{\Z}(\alpha,\beta,\gamma)$. 

\end{enumerate}
We have coincidence $\mathcal{L}_{\R}(k)=\mathcal{L}_{\Z}(k,-k,k)$ for all $k\in \Z$. We observe that $\mathcal{L}_{\Z}(-1,0,1)$ is the extension $\mathcal{L}_P$ of Examples \ref{ex:trivial} and \ref{ex:thinfusionextriv}, with $P$ the Poincaré bundle over $T=\ueins\times \ueins$. Further,  $\mathcal{L}_{\Z}(1,1,-1)$ is the \quot{basic central extension} of $L\ueins$ \cite[Prop. 4.7.5]{pressley1}.

One can show that for all $\tau_1,\tau_2$, and all $\alpha,\beta,\gamma\in\R$ the following symmetry law holds: 
\begin{equation*}
\eta(\tau_1,\tau_2)= \exp 2\pi \im \left ( 2\alpha \int_0^1 f_1(s)f_2'(s)\mathrm{d}s-\alpha n_2n_1+(\gamma-\alpha) n_1f_2(0)-(\gamma+\alpha ) n_2f_1(0)) \right ) \eta(\tau_2,\tau_1)\text{.}
\end{equation*}
\begin{comment}
We compute
\begin{eqnarray*}
\eta(\tau_1,\tau_2)\eta(\tau_2,\tau_1)^{-1}&=&\exp 2\pi \im \left ( \alpha \int_0^1 (f_1(s)f_2'(s)-f_2(s)f_1'(s))\mathrm{d}s +\gamma (n_1f_2(0)-n_2f_1(0)) \right )
\\ &=&\exp 2\pi \im \left ( 2\alpha \int_0^1 f_1(s)f_2'(s)\mathrm{d}s-\alpha f_2(1)f_1(1)+\alpha f_1(0)f_2(0) +\gamma (n_1f_2(0)-n_2f_1(0)) \right )
\\ &=&\exp 2\pi \im \left ( 2\alpha \int_0^1 f_1(s)f_2'(s)\mathrm{d}s-\alpha n_2n_1+(\gamma-\alpha) n_1f_2(0)-(\gamma+\alpha ) n_2f_1(0)) \right )
\end{eqnarray*}
\end{comment}
This shows in the first place that the extensions $\mathcal{L}_{\Z}(0,\beta,0)$ are commutative. 
Let $\phi: [0,1] \to [0,1]$ be a smoothing map. Define smooth maps $f_1,f_2:[0,1] \to \R$ by
\begin{equation*}
f_1(t):= \begin{cases}
\phi(2t) & 0\leq t \leq \frac{1}{2} \\
0 & \frac{1}{2}\leq t \leq 1 \\
\end{cases}
\quand
f_2(t):= \begin{cases}
1 & 0\leq t \leq \frac{1}{2} \\
\phi(2t-1) & \frac{1}{2}\leq t \leq 1 \\
\end{cases}
\end{equation*} 
and extend them periodically with shifts by $n_1=n_2=-1$. The corresponding loops $\tau_1$ and $\tau_2$ have disjoint support.
We have $f_1f_2'=0$, $f_1(0)=f_2(0)=1$, and hence
\begin{equation*}
\eta(\tau_1,\tau_2)=\exp (2\pi \im \alpha) \cdot \eta(\tau_2,\tau_1) \text{.}
\end{equation*}
This shows that the central extensions $\mathcal{L}_{\R}(\gamma)$ with $\gamma\notin \Z$ do not satisfy the disjoint commutativity law of Theorem \ref{th:localcomm}, and we conclude that they cannot be equipped with the structure of thin fusion extensions. 
\begin{comment}
As $LS^1$ is commutative, all 2-cocycles that are coboundaries are symmetric.
\begin{comment}
Indeed, $\Delta\varepsilon(\tau_1,\tau_2)=\varepsilon(\tau_1)\varepsilon(\tau_2)\varepsilon(\tau_1\tau_2)^{-1}=\Delta\varepsilon(\tau_2,\tau_1)$.
This can be used in order to exclude that two 2-cocycles are cohomologous. Indeed, two 2-cocycles $\eta_1,\eta_2$ are cohomologous only if $\eta_1(\tau_1,\tau_2)\eta_1(\tau_2,\tau_1)^{-1}=\eta_2(\tau_1,\tau_2)\eta_2(\tau_2,\tau_1)^{-1}$. For $\eta_1 =\eta_{\alpha_1,\beta_1,\gamma_1}$ and $\eta_2 =\eta_{\alpha_2,\beta_2,\gamma_2}$ this holds if and only if
\begin{equation*}
 2(\alpha-\alpha') \int_0^1 f_1(s)f_2'(s)\mathrm{d}s-(\alpha-\alpha') n_2n_1+(\gamma-\gamma'-\alpha+\alpha') n_1f_2(0)-(\gamma-\gamma'+\alpha-\alpha' ) n_2f_1(0)) \in \Z\text{.}
\end{equation*}
If $\alpha\neq \alpha'$ or $\gamma\neq \gamma'$ it is easy to produce loops $\tau_1,\tau_2$ such that this quantity is not integral. Thus, $\gamma\neq \gamma'$ implies $\mathcal{L}_{\R}(\gamma) \ncong \mathcal{L}_{\R}(\gamma')$, and $\alpha\neq\alpha$ or $\gamma\neq\gamma'$ implies $\mathcal{L}_{\Z}(\alpha,\beta,\gamma)\ncong \mathcal{L}_{\Z}(\alpha',\beta',\gamma')$. 
\end{comment}
\end{example}

\setsecnumdepth{2}

\section{Transgression-regression machine}

\label{sec:transgression}

\subsection{Multiplicative bundle gerbes}

We  use the theory of bundle gerbes and connections on those \cite{murray,stevenson1,carey2,waldorf1}. We  denote by $\ugrb X$ and $\ugrbcon X$ the  bicategories of bundle gerbes without and  with connection over a smooth manifold $X$, respectively. Forgetting the connection is an essentially surjective, and in general neither full nor faithful functor
$\ugrbcon X \to \ugrb X$. 
The 1-morphisms are called (connection-preserving) \emph{isomorphisms}, and the 2-morphisms are called (connection-preserving)  \emph{transformations}. 2-forms $\rho\in \Omega^2(X)$ can be considered as connections on the trivial bundle gerbe $\mathcal{I}$; as a bundle gerbe with connection it is denoted by  $\mathcal{I}_{\rho}$. 

\begin{definition}[{{\cite{carey4,waldorf5}}}]
A \emph{multiplicative bundle gerbe with connection} over a Lie group $G$ is a bundle gerbe $\mathcal{G}$ with connection over $G$, a multiplicative 2-form $\rho\in \Omega^2(G \times G)$,
a connection-preserving isomorphism
\begin{equation*}
\mathcal{M}:\mathcal{G}_1 \otimes \mathcal{G}_2 \to \mathcal{G}_{12} \otimes \mathcal{I}_{\rho}
\end{equation*}
over $G \times G$, and a connection-preserving transformation
\begin{equation}
\label{eq:alpha}
\alxydim{@R=\xyst@C=4.5em}{\mathcal{G}_1 \otimes \mathcal{G}_2 \otimes \mathcal{G}_3 \ar[r]^-{\mathcal{M}_{1,2} \otimes \id} \ar[d]_{\id \otimes \mathcal{M}_{2,3}} & \mathcal{G}_{12} \otimes \mathcal{G}_3 \otimes \mathcal{I}_{\rho_{1,2}} \ar[d]^{\mathcal{M}_{12,3} \otimes \id} \ar@{=>}[dl]|*+{\alpha} \\ \mathcal{G}_1 \otimes \mathcal{G}_{23} \otimes \mathcal{I}_{\rho_{2,3}} \ar[r]_-{\mathcal{M}_{1,23} \otimes \id} & \mathcal{G}_{123} \otimes \mathcal{I}_{\rho_{\Delta}}}
\end{equation}
between isomorphisms over $G\times G\times G$, such that $\alpha$ satisfies a pentagon axiom over $G^4$. 
\end{definition}

Here our  index convention is so that e.g. the index $(..)_{ij,k}$ stands for the pullback along the map $(g_i,g_j,g_k) \mapsto (g_ig_j,g_k)$. For instance, $\mathcal{G}_i = \pr_i^{*}\mathcal{G}$ and $\mathcal{G}_{12}$ is the pullback along the multiplication of $G$. Further, we have written $\rho_{\Delta} := \rho_{1,2} + \rho_{12,3} = \rho_{2,3} + \rho_{1,23}$, with the equality coming from the multiplicativity of $\rho$, see \erf{eq:multiplicative}. The pentagon axiom for $\alpha$ can be found  in \cite[Definition 1.3]{waldorf5}.
 For later purpose, we need the following lemma.

\begin{lemma}
\label{lem:canalpha}
Suppose $G$ is connected. Let $\mathcal{G}$ be a bundle gerbe with connection over $G$, $\rho\in \Omega^2(G \times G)$ be a multiplicative 2-form, and  $\mathcal{M}: \mathcal{G}_1 \otimes \mathcal{G}_2 \to \mathcal{G}_{12} \otimes \mathcal{I}_{\rho}$ be a connection-preserving isomorphism, such that the set $\inf X$ of  connection-preserving transformations \erf{eq:alpha}
is non-empty.
Then, $\inf X$ contains  a unique element $\alpha$ that satisfies the pentagon axiom. 
\end{lemma}

\begin{proof}
We pick some $\alpha \in \inf X$. The pentagon axiom is an equality between two connection-preserving transformations over $G^4$. Set of connection-preserving transformations between two fixed connection-preserving isomorphisms is a torsor over the group of locally constant $\ueins$-valued maps.  Thus, the pentagon axiom for $\alpha$ is satisfied up to a locally constant map $\varepsilon\maps G^4 \to \ueins$.  Since $G$ is connected, $\varepsilon$ is constant; $\varepsilon\in \ueins$. We  regard this constant as a locally constant map  $\varepsilon:G^3 \to \ueins$, and define  a new element $\alpha' := \alpha \cdot \varepsilon \in \inf X$. The pentagon axiom has five occurrences of $\alpha'$: three on one side and two on the other. Thus, four occurrences of $\varepsilon$ cancel, and the remaining one compensates the error caused by $\alpha$; hence,  $\alpha'$ satisfies the pentagon axiom. Assume  $\alpha$, $\alpha' \in \inf X$  satisfy the pentagon axiom. They differ by a locally constant map $\varepsilon: G^3 \to \ueins$, i.e. a constant. In the pentagon axioms for $\alpha$ and $\alpha'$, this leads to $\varepsilon^3=\varepsilon^2$, i.e. $\varepsilon=1$. Therefore, $\alpha=\alpha'$. 
\end{proof}

If $(\mathcal{G},\rho,\mathcal{M},\alpha)$ and $(\mathcal{G}',\rho,\mathcal{M}',\alpha')$ are multiplicative bundle gerbes over $G$ with connections (with the same 2-form $\rho$),
a \emph{1-morphism} is a connection-preserving isomorphism $\mathcal{A}:\mathcal{G} \to \mathcal{G}'$ together with a connection-preserving  transformation
\begin{equation}
\label{eq:beta}
\alxydim{@R=\xyst@C=2cm}{\mathcal{G}_1 \otimes \mathcal{G}_2 \ar[r]^-{\mathcal{M}} \ar[d]_{\mathcal{A}_{1} \otimes \mathcal{A}_2} & \mathcal{G}_{12} \otimes \mathcal{I}_{\rho} \ar@{=>}[dl]|*+{\beta} \ar[d]^{\mathcal{A}_{12}  \otimes \id} \\ \mathcal{G}_1' \otimes \mathcal{G}_2' \ar[r]_-{\mathcal{M}'} & \mathcal{G}_{12}' \otimes \mathcal{I}_{\rho'}}
\end{equation}
over $G \times\ G$  that satisfies a compatibility condition with respect to  $\alpha$ and $\alpha'$ over $G^3$, see \cite[Definition 1.7]{waldorf5}. 

\begin{lemma}
\label{lem:canbeta}
Suppose $(\mathcal{G},\rho,\mathcal{M},\alpha)$ and $(\mathcal{G}',\rho,\mathcal{M}',\alpha')$ are multiplicative bundle gerbes with connection over a connected Lie group $G$,  suppose $\mathcal{A}: \mathcal{G} \to \mathcal{G}'$ is a connection-preserving isomorphism, and suppose $\beta$ is a connection-preserving transformation \erf{eq:beta}. Then, $\beta$ is  compatible with $\alpha$ and $\alpha'$.
\end{lemma}

\begin{proof}
We argue as in the proof of Lemma \ref{lem:canalpha}. The compatibility condition is an equality between two transformations over $G^3$, with four occurrences of $\beta$. Thus, it is satisfied up to a locally constant map $\varepsilon:G^3 \to \ueins$, i.e. a constant. The two pentagon axioms for $\alpha$ and $\alpha'$ over $G^4$ are related by $20=4\cdot 5$ occurrences of $\beta$. As the pentagon axioms are satisfied, and the compatibility diagrams commute up to $\varepsilon$, we obtain that five occurrences of $\varepsilon$ have to cancel. This requires $\varepsilon=1$.\end{proof}

If $(\mathcal{A},\beta)$ and $(\mathcal{A}',\beta')$ are 1-morphisms between multiplicative bundle gerbes with connection, a \emph{2-morphism} is a connection-preserving transformation $\varphi:\mathcal{A} \Rightarrow \mathcal{A}'$ such that the diagram
\begin{equation*}
\alxydim{@R=\xyst@C=2cm}{(\mathcal{A}_1 \otimes \mathcal{A}_2) \circ \mathcal{M} \ar@{=>}[r]^-{\beta} \ar@{=>}[d]_{(\varphi_1 \otimes \varphi_2) \circ \id} & \mathcal{M}' \circ \mathcal{A}_{12} \ar@{=>}[d]^{\id\circ \varphi_{12}} \\ (\mathcal{A}'_1 \otimes \mathcal{A}'_2) \circ \mathcal{M} \ar@{=>}[r]_-{\beta'} & \mathcal{M}' \circ \mathcal{A}_{12}'}
\end{equation*}
is commutative.
With these definitions, multiplicative bundle gerbes with connection form a  bicategory that we denote by $\multgrbcon G$.

Multiplicative bundles gerbes \emph{without} connections are defined analogously, without the 2-form $\rho$ and without occurrences of trivial gerbes.
 We denote by $\multgrb G$ the bicategory of multiplicative bundle gerbes over $G$. We have the following classification result of \cite{carey4}:
\begin{equation}
\label{eq:multgrbclass}
\hc 0 \multgrb G \cong \h^4(BG,\Z)\text{,}
\end{equation}
where $\hc 0$ denotes taking the set of isomorphic objects.
We denote by $\multgrbadcon G$ the full sub-bicategory of $\multgrb G$ over those multiplicative bundle gerbes that \emph{admit} connections. For compact Lie groups  $G$, we have $\multgrbadcon G = \multgrb G$ \cite[Proposition 2.8]{waldorf5}. All bicategories of multiplicative bundle gerbes are symmetric monoidal, under the tensor product of bundle gerbes, and \erf{eq:multgrbclass} is an isomorphism between groups.

\begin{example}
\label{ex:trivialgerbe}
The trivial gerbe $\mathcal{I}_{\omega}$ for any $\omega\in\Omega^2(G)$ carries  multiplicative structures parameterized by principal  $\ueins$-bundles $P$ with connection over $G \times G$ such that
\begin{equation*}
P_{1,2}\otimes P_{12,3} \cong P_{2,3}\otimes P_{1,23}
\end{equation*}
via a coherent connection-preserving isomorphism, see \cite[Example 1.4]{waldorf5}. In this case, the 2-form is $\rho \eq \mathrm{curv}(P)-\omega$. For $G=\ueins$, we have $\h^4(B\ueins,\Z)=\h^4(K(\Z,2),\Z)=\Z$. In \cite[Prop. 2.4]{waldorf5} it is shown that there is an exact sequence
\begin{equation*}
0 \to \h^3(B\ueins,U(1)) \to \hc 0 \multgrbcon {{\ueins}} \to \Z\text{.}
\end{equation*} 
\begin{comment}
Instead of $\Z$ there is $M_{\Z}^3(S^1)$, which is a subset of $M_{\R}^3(S^1) = \Omega^2_{\Delta}(S^1 \times S^1)$. The element there is represented by $\rho$. The subset is composed of those forms that are integral under the simplicial de Rham isomorphism from $\h^4(BS^1,\Z)=\Z$. \end{comment}
As $\h^3(B\ueins,U(1))=\h^3(K(\Z,2),\ueins)=0$, we see that $\hc 0 \multgrbcon {{\ueins}}\cong\hc 0 \multgrb {{\ueins}}\cong \Z$. This $\Z$-family of multiplicative gerbes is obtained by taking $\omega=0$ and $P$ the Poincaré bundle over $T= \ueins \times \ueins$.
\end{example}

\begin{example}
\label{ex:basicgerbe}
Suppose $G$ is compact, simple  and simply-connected. There exists a (up to connection-preserving isomorphisms) unique bundle gerbe $\gbas$ with connection of curvature $H$, where $H\in \Omega^3(G)$ is the 3-form of Example \ref{ex:wzwmodel}; it is called the \emph{basic gerbe}. One can show that $\gbas$ has a unique multiplicative structure \cite[Example 1.5]{waldorf5}, where $\rho\in\Omega^2(G \times G)$ is the 2-form  of Example \ref{ex:wzwmodel}.  There exist  Lie-theoretical constructions of $\gbas$  \cite{gawedzki1,meinrenken1}. Constructions of the corresponding multiplicative structures are notoriously difficult; one option is described in \cite[Section 7]{Waldorf}.
For (non-simply connected) compact simple Lie groups, all multiplicative gerbes with connection are tabulated, and can be constructed via descent from their simply-connected covers \cite{gawedzki9}.  
\end{example}

\subsection{Transgressive central extensions}

For every smooth manifold $X$, there is a transgression functor
\begin{equation}
\label{eq:tr}
\tr: \hc 1 \ugrbcon X \to \ubun{LX}
\end{equation} 
with target the category of Fréchet principal $\ueins$-bundles over  $LX$. The symbol  $\hc 1$ stands for passing from a bicategory to a category by identifying 2-isomorphic isomorphisms.

Transgression for gerbes has  been defined by Gaw\c edzki  in terms  of cocycles for Deligne cohomology \cite{gawedzki3}, and by Gaw\c edzki-Reis  for bundle gerbes \cite{gawedzki1}.  Brylinski   has defined transgression in terms of sheaves of categories \cite{brylinski1}.  The functor  \erf{eq:tr} that we use here is an adaption of Brylinski's functor to bundle gerbes, and  defined in \cite{waldorf5}. It is monoidal, and natural with respect to smooth maps $f\maps X \to X'$ between smooth manifolds and the induced maps $Lf\maps LX \to LX'$ between their loop spaces. Furthermore, if $\rho\in \Omega^2(X)$, $\mathcal{I}_{\rho}$ is the trivial bundle gerbe with connection $\rho$, then its transgression $\tr_{\mathcal{I}_{\rho}}$ has a canonical trivialization $t_{\rho}:\tr_{\mathcal{I}_{\rho}} \to \trivlin$, where $\trivlin$ is the trivial $\ueins$-bundle over $LX$.

Suppose $(\mathcal{G},\rho,\mathcal{M},\alpha)$ is a multiplicative bundle gerbe with connection over $G$. 
Applying the transgression functor to  $\mathcal{G}$, we obtain a Fréchet principal $\ueins$-bundle $\mathcal{L} :=\tr_{\mathcal{G}}$ over  $LG$. Because transgression is functorial and monoidal, the transgression of the connection-preserving isomorphism $\mathcal{M}$ together with the trivialization $t_{\rho}$ give a bundle isomorphism 
\begin{equation*}
\alxydim{@C=1.6cm}{\mathcal{L}_1 \otimes \mathcal{L}_2 \ar[r]^-{\tr_{\mathcal{M}}} & \mathcal{L}_{12} \otimes \tr_{\mathcal{I}_{\rho}} \ar[r]^-{\id \otimes t_{\rho}} & \mathcal{L}_{12}}
\end{equation*}
over $LG \times LG$. It induces a binary operation on   $\mathcal{L}$ that covers the group structure of $LG$.  The  existence of the transformation $\alpha$ implies  under transgression the associativity of that binary operation. 
This equips  $\mathcal{L}$  with the structure of a Fréchet Lie group \cite[Theorem 3.1.7]{waldorf5},   making up a central extension
\begin{equation*}
1 \to \ueins \to \mathcal{L} \to LG \to 1\text{.}
\end{equation*}

\begin{definition}
\label{def:transgressive}
A central extension $\mathcal{L}$ of $LG$ is called \emph{transgressive}, if there exists a multiplicative bundle gerbe with connection over $G$ whose transgression is isomorphic to $\mathcal{L}$ as a central extension. 
\end{definition}

In \cite{waldorf10} a category $\ufusbunconsf {LX}$
is considered with objects the  Fréchet principal $\ueins$-bundles over $LX$ equipped with fusion products and fusive superficial connections, and morphisms the fusion-preserving, connection-preserving bundle morphisms. A construction in \cite[Section 4.2]{waldorf10} lifts the transgression functor \erf{eq:tr} to this category:
\begin{equation}
\label{eq:trfull}
\trcon: \hc 1 \ugrbcon X \to \ufusbunconsf {LX}\text{.}
\end{equation}
In case of a multiplicative bundle gerbe $\mathcal{G}$ with connection over $G$, this means in the first place that the underlying principal $\ueins$-bundle of the central extension $\mathcal{L}$ is equipped with a  fusion product $\lambda$   and with a fusive superficial connection $\nu$.

Under the lifted transgression functor, the transgression  $\trcon_{\mathcal{I}_{\rho}}$ of the trivial bundle gerbe with connection $\rho\in\Omega^2(X)$ is equipped with a fusion product and a connection, which under the trivialization $t_{\rho}\maps \trcon_{\mathcal{I}_{\rho}} \to \trivlin$ correspond to the trivial fusion product on $\trivlin$ and the connection 1-form $\varepsilon_{\nu} :=\tau_{S^1}(\rho) \in \Omega^1(LX)$ \cite[Lemma 3.6]{waldorf13}.
Thus, the group structure of $\mathcal{L}$ is induced by the fusion-preserving, connection-preserving bundle morphism
\begin{equation*}
\alxydim{@C=1.6cm}{\mathcal{L}_1 \otimes \mathcal{L}_2 \ar[r]^-{\trcon_{\mathcal{M}}} & \mathcal{L}_{12} \otimes \trcon_{\mathcal{I}_{\rho}} \ar[r]^-{\id \otimes t_{\rho}} & \mathcal{L}_{12} \otimes \trivlin_{\varepsilon_{\nu}}\text{.}}
\end{equation*}
This means (1) that $\varepsilon_{\nu} \in \Omega^1(LG^2)$ is the error 1-form of $\nu$. By Example \ref{ex:pathsplitting}, $\kappa := \tau_{[0,1]}(\rho)|_{P(G\times G)}$ is a multiplicative and contractible path splitting for $\varepsilon_{\nu}$. It means (2) that the fusion product is multiplicative, see \cite[Theorem 4.3.1]{Waldorfa}.

\noindent
Collecting all this data forces us to consider a category $\fusextcon{LG}$ with:
\begin{itemize}

\item 
Objects: Central extensions of $LG$ equipped with a multiplicative fusion product, a fusive superficial connection, and a multiplicative, contractible path splitting of its error 1-form.

\item
Morphisms: Fusion-preserving, connection-preserving isomorphisms of central extensions, with the same error 1-form and the same path splitting on both sides. 

\end{itemize}
The fact that transgression is a functor and monoidal implies that above procedure  defines a monoidal functor
\begin{equation}
\label{eq:multtrans}
\multtrcon: \hc 1\multgrbcon G \to \fusextcon {LG}\text{.}
\end{equation}

We can pass from superficial connections to thin structures in terms of an essentially surjective functor
\begin{equation*}
th: \fusextcon{LG} \to \fusextth{LG}
\end{equation*}
to the category of thin fusion extensions introduced in Section \ref{sec:fusextth}. 
Forgetting the fusion product and the thin structure  gives another functor from $\fusextth{LG}$ to the category $\ext {LG}$ of bare central extensions of $LG$. The composite
\begin{equation}
\label{eq:transgressivity}
\alxydim{}{\hc 1 \multgrbcon G \ar[r]^-{\multtrcon} & \fusextcon {LG} \ar[r]^-{th} &  \fusextth{LG}  \ar[r] & \ext {LG}}
\end{equation}
is the procedure from the beginning of the present subsection: the transgression of a multiplicative bundle gerbe with connection to a central extension. Thus, we obtain the following result, constituting the first part of Theorem \ref{th:main}.

\begin{proposition}
\label{prop:onlyif}
A central extension is transgressive only if it can be equipped with the structure of  a thin fusion extension.
\end{proposition}

The functor \erf{eq:multtrans} has a version when the connections on both sides are dropped, at the price that it only exists as map, not as a functor.

\begin{proposition}
\label{prop:multtrans}
There exists a unique map $\multtr: \hc 0 \multgrbadcon G \to \hc 0 \fusextth{LG}$
such that the diagram
\begin{equation*}
\alxydim{@C=4em@R=\xyst}{\hc 0 \multgrbcon G \ar[d] \ar[r]^{\hc 0 \multtrcon}  & \hc 0 \fusextcon{LG} \ar[d]^{\hc 0 th} \\ \hc 0 \multgrbadcon G \ar[r]_{\multtr} & \hc 0 \fusextth{LG}}
\end{equation*}
is commutative.
\end{proposition}

\begin{proof}
Uniqueness is clear as the vertical map on the left is surjective (by definition of $\multgrbadcon G$). For the existence, we prove that the thin fusion extensions one gets from different choices of connections on the same multiplicative bundle gerbe are isomorphic.

Let $\mathcal{G}$ be a multiplicative bundle gerbe over $G$ with two connections, say $\lambda_1$ and $\lambda_2$, with corresponding thin fusion extensions $\mathcal{L}_1$ and $\mathcal{L}_2$, respectively. 
In \cite[Proposition 5.1.3]{Nikolausa} we have have constructed an isomorphism $\varphi: \mathcal{L}_1 \to \mathcal{L}_2$ between central extensions. It was defined as the composition $\varphi := (\id \otimes t_{\beta}) \circ \trcon_{\id_{\varepsilon}}$, where  $\beta \in \Omega^2(G)$ and $\id_{\varepsilon}\maps  \mathcal{G}_{\lambda_1} \to \mathcal{G}_{\lambda_2} \otimes \mathcal{I}_\beta$ is a connection-preserving isomorphism. The transgressed isomorphism  $\trcon_{\id_\varepsilon}$  and the canonical trivialization $t_{\beta}:\trcon_{\mathcal{I}_\beta} \to \trivlin_{\tau_{S^1}(\beta)}$ are  fusion-preserving; this shows that $\varphi$ is fusion-preserving. As a connection on the trivial bundle, the 1-form $\tau_{S^1}(\beta)$ induces the trivial thin structure \cite[Proposition 3.2.3]{waldorf11}; so that $th(t_{\beta})$ is a thin bundle morphism. Hence, $\varphi$ is an isomorphism of thin fusion extensions.
\end{proof}

\begin{example}
We consider the transgression of the trivial gerbe $\mathcal{I}_{\omega}$ equipped with a multiplicative structure defined by a principal $\ueins$-bundle $P$ with connection over $G \times G$ (see Example \ref{ex:trivialgerbe}). The  trivialization $t_{\omega}: \trcon_{\mathcal{I}_{\omega}} \to \ueins \times LG$ induces an isomorphism between the transgressive thin fusion extension  $\trcon_{\mathcal{I}_{\omega}}$ and the thin fusion extension $\mathcal{L}_P$ of Examples \ref{ex:trivial} and \ref{ex:thinfusionextriv}.
In particular, $\mathcal{L}_P$ is transgressive.
\end{example}

\begin{example}
Let $G$ be a compact, simple, simply-connected Lie group, let $\gbas$ be the basic gerbe over $G$ (Example \ref{ex:basicgerbe}), and let $\wzwmodel$ be the universal central extension of $LG$ (Examples \ref{ex:wzwmodel} and \ref{ex:wzwmodelcon}). There is an isomorphism
\begin{equation*}
\varphi: \wzwmodel \to \trcon_{\gbas}
\end{equation*}
of central extensions that is fusion-preserving and thin, and so establishes an isomorphism between thin fusion extensions. In particular, the universal central extension $\mathcal{L}_G$ is transgressive.
The isomorphism $\varphi$ is defined by
\begin{equation*}
\varphi  (\phi,z)=\partial\mathcal{T} \cdot z \cdot \exp  2\pi\im\left ( -\int_{D^2} \omega \right )\text{.}
\end{equation*}
Here, $\mathcal{T}\maps  \phi^{*}\gbas \to \mathcal{I}_{\omega}$ is an arbitrarily chosen trivialization of $\phi^{*}\gbas$ over $D^2$ and $\partial \mathcal{T}$ denotes its restriction to the boundary; the latter is a trivialization of $\partial\phi^{*}\gbas$ over $S^1$, constituting an element in $\trcon_{\gbas}$ over the loop $\partial\phi$. It is straightforward to see that $\varphi$ is well-defined, $\ueins$-equivariant,  fusion-preserving, and a group homomorphism; see \cite[Section 4.3]{Waldorfa} for details.
In order to see that $\varphi$ is thin, we consider the thin connection  $\nu$ on $\wzwmodel$ that integrates the thin homotopy equivariant structure $d$  (see Example \ref{ex:wzwmodelcon}), and show the stronger statement that $\varphi$ is  connection-preserving.  Indeed, if $\gamma:[0,1] \to LG$ is a path, and $\phi_0,\phi_1:D^2 \to G$ are smooth maps with $\partial\phi_0 = \gamma(0)$ and $\partial\phi_1 = \gamma(1)$, then $pt^{\nu}_{\gamma}(\phi_0,1)=(\phi_1,\mathrm{e}^{2\pi \im S_{\mathrm{WZ}}(\Phi_{\gamma})})$. Let $\mathcal{T}_0,\mathcal{T}_1$ be trivializations of $\phi_0^{*}\gbas$ and $\phi_1^{*}\gbas$, respectively, and let $\tilde\nu$ denote the connection on $\trcon_{\gbas}$. Employing the definition of the transgression functor, see \cite[Section 4.3]{waldorf10}, the parallel transport in $\trcon_{\gbas}$ is $pt^{\tilde\nu}_{\gamma}(\partial \mathcal{T}_0)=\partial \mathcal{T}_1 \cdot \mathcal{A}_{\gbas}(h_{\gamma},\mathcal{T}_0,\mathcal{T}_1)$, where the latter term is the surface holonomy of $\gbas$ with the trivializations as boundary conditions. In the present case of a 2-connected Lie group, it can be computed via the 3-form $H$ and the two 2-forms $\omega_0,\omega_1$ of the trivializations $\mathcal{T}_0,\mathcal{T}_1$, namely as
\begin{equation*}
\mathcal{A}_{\gbas}(h_{\gamma},\mathcal{T}_0,\mathcal{T}_1)= \exp 2\pi \im \left (S_{\mathrm{WZ}}(\Phi_{\gamma}) + \int_{D^2}\omega_0 - \int_{D^2}\omega_1 \right)\text{,}
\end{equation*} 
see \cite[Proposition 3.1.4 (iii)]{waldorf4}. Now we compute
\begin{multline*}
pt_{\gamma}(\varphi(\phi_0,1)) =pt_{\gamma}(\partial \mathcal{T}_0)\cdot \exp  2\pi\im \left (- \int_{D^2} \omega_0 \right )= \partial \mathcal{T}_1 \cdot \exp 2\pi\im  \left (S_{\mathrm{WZ}}(\Phi_{\gamma})- \int_{D^2} \omega_1 \right )\\=\varphi(\phi_1,\mathrm{e}^{2\pi \im S_{\mathrm{WZ}}(\Phi_{\gamma})})=\varphi(pt_{\gamma}^{\nu}(\phi_0,1))\text{.}
\end{multline*}
This shows that $\varphi$ commutes with the parallel transport along arbitrary paths; hence, $\varphi$ is connection-preserving, in particular thin. 
\end{example}

\subsection{Regression and equivalence result}

By the main result of \cite{waldorf10}, the lifted transgression functor \erf{eq:trfull} is an equivalence of categories, and has  for fixed $x\in X$ a canonical inverse functor
\begin{equation}
\label{eq:regcon}
\uncon_x: \ufusbunconsf{LX} \to \hc 1 \ugrbcon X\text{,}
\end{equation}
called \emph{regression}. We need the following result about the regression of trivial bundles, which is explained in Appendix \ref{app:toptrivreg}. If $\varepsilon \in \Omega^1(LX)$ is a superficial connection on the trivial bundle over $LX$, and fusive with respect to the trivial fusion product, then every path splitting $\kappa \in \Omega^1(PX)$  defines a connection-preserving isomorphism $\mathcal{T}_{\kappa}: \uncon(\trivlin_{\varepsilon}) \to \mathcal{I}_{\rho_{\kappa}}$ between the regression of $\trivlin_{\varepsilon}$ and the trivial bundle gerbe $\mathcal{I}$ over $X$ equipped with a connection 2-form $\rho_{\kappa}\in \Omega^2(X)$ defined from $\kappa$.  

If $X$ is a group,  we always choose $x=1$ and omit the index.  
We use the functor $\uncon=\uncon_1$ in order to construct a regression functor
\begin{equation}
\label{eq:multreg}
\multuncon: \fusextcon{LG} \to \hc 1 \multgrbcon G
\end{equation}
defined on the category $\fusextcon {LG}$ introduced in the previous section.
Suppose  $\mathcal{L}$ is a central extension of $LG$ equipped with a multiplicative fusion product $\lambda$, a fusive superficial  connection,  and a multiplicative path splitting $\kappa$ of its error 1-form $\varepsilon$ (for the definition of $\multuncon$ we do not need that $\kappa$ is contractible). The group structure defines over $LG \times LG$ a connection-preserving, fusion-preserving bundle isomorphism
\begin{equation*}
\mu: \mathcal{L}_1 \otimes \mathcal{L}_2 \to \mathcal{L}_{12} \otimes \trivlin_\varepsilon\text{,}
\end{equation*}
and the associativity of the group structure implies a commutative diagram over $LG^3$ .

We let $\mathcal{G} := \uncon(\mathcal{L})$ be the regressed bundle gerbe over $G$ with connection. Over $G \times G$ we consider the connection-preserving isomorphism $\mathcal{M}$ defined as
\begin{equation}
\label{eq:multisocon}
\alxydim{@C=4em}{\mathcal{G}_1 \otimes \mathcal{G}_2 \ar[r]^-{\uncon(\mu)} & \mathcal{G}_{12} \otimes \uncon(\trivlin_{\varepsilon}) \ar[r]^-{\id \otimes \mathcal{T}_{\kappa}} & \mathcal{G}_{12} \otimes \mathcal{I}_{\rho_{\kappa}}\text{.}}
\end{equation}
The various pullbacks of $\mathcal{M}$ to $G \times G \times G$ constitute the outer arrows of the  following diagram in the category $\hc 1 \ugrbcon- (G \times G \times G)$:
\begin{equation*}
\footnotesize
\alxydim{@R=4em@C=2.8em}{\mathcal{G}_1 \otimes \mathcal{G}_2 \otimes \mathcal{G}_3 \ar[ddr]|{\id \otimes \uncon(\mu)_{2,3}} \ar[drr]|{\uncon(\mu)_{1,2} \otimes \id} \ar[rrrr]^-{\mathcal{M}_{1,2} \otimes \id} \ar[dddd]_{\id \otimes \mathcal{M}_{2,3}} &&&& \mathcal{G}_{12} \otimes \mathcal{G}_3 \otimes \mathcal{I}_{\rho_{\kappa,1,2}}\ar[ddl]|{\uncon(\mu)_{12,3} \otimes \id} \ar[dddd]^{\mathcal{M}_{12,3} \otimes \id} \\ &&\mathcal{G}_{12} \otimes \mathcal{G}_3 \otimes  \uncon(\trivlin_{\varepsilon_{1,2}}) \ar[d]_{\uncon(\mu)_{12,3 \otimes \id} } \ar[rd]^{\uncon(\mu)_{12,3} \otimes \mathcal{T}_{\kappa_{1,2}}} \ar[urr]|{\id \otimes \id \otimes \mathcal{T}_{\kappa_{1,2}}} && \\ &\hspace{-7em}\mathcal{G}_1 \otimes \mathcal{G}_{23} \otimes \uncon(\trivlin_{\varepsilon_{2,3}}) \ar[r]^-{\uncon(\mu)_{1,23} \otimes \id} \ar[dr]_{\uncon(\mu)_{1,23}\otimes \mathcal{T}_{\kappa_{2,3}}}\ar[ddl]|{\id \otimes \id \otimes \mathcal{T}_{\kappa_{2,3}}} & \mathcal{G}_{123} \otimes \uncon(\trivlin_{\varepsilon_{\Delta}}) \ar[r]_-{\id \otimes \mathcal{T}_{\kappa_{1,2}}} \ar[d]^{\id \otimes \mathcal{T}_{\kappa_{2,3}}}   &\mathcal{G}_{123}  \otimes  \uncon(\trivlin_{\varepsilon_{12,3}})\otimes \mathcal{I}_{\rho_{\kappa,1,2}}\hspace{-7em} \ar[ddr]|{\id \otimes \mathcal{T}_{\kappa_{12,3}} \otimes \id} &\\ &&\mathcal{G}_{123}  \otimes  \uncon(\trivlin_{\varepsilon_{1,23}})\otimes \mathcal{I}_{\rho_{\kappa,2,3}}\ar[drr]|{\id \otimes \mathcal{T}_{\kappa_{1,23}} \otimes \id} &&  \\ \mathcal{G}_1 \otimes \mathcal{G}_{23} \otimes \mathcal{I}_{\rho_{\kappa,2,3}} \ar[urr]|{\uncon(\mu)_{1,23} \otimes \id} \ar[rrrr]_-{\mathcal{M}_{1,23} \otimes \id} &&&& \mathcal{G}_{123} \otimes \mathcal{I}_{\rho_{\kappa,\Delta}}}
\end{equation*}
All triangular subdiagrams commute obviously. There remain two four-sided   subdiagrams whose commutativity is to check. The one that touches the upper left corner commutes due to the commutative diagram for $\mu$ over $LG^3$ and the fact that $\uncon$ is a functor. The one that touches the lower right corner commutes due to the multiplicativity of $\kappa$ and the additivity of the trivialization $\mathcal{T}_{\kappa}$ shown in Lemma \ref{lem:kappaadd}.

The commutativity of the diagram in $\hc 1 \ugrbcon- (G^3)$ implies the existence of a connection-preserving transformation $\alpha$ that fills the diagram  in $\ugrbcon {G^3}$. Thus, by Lemma \ref{lem:canalpha} there is a unique connection-preserving transformation $\alpha$ making $(\mathcal{G},\rho_{\kappa},\mathcal{M},\alpha)$ a multiplicative bundle gerbe with connection; this defines the functor $\multuncon$ on the level of objects.
On the level of morphisms, one similarly composes a commutative diagram from an isomorphism in $\fusextcon{LG}$,  and then Lemma \ref{lem:canbeta} implies that it yields a morphism between multiplicative bundle gerbes with connection.

\begin{comment}
Note that the definition of the functor $\multuncon$ does not use that $\kappa$ is contractible; this is needed for the next statement.
\end{comment}

\begin{theorem}
\label{th:equivalence}
For a connected Lie group $G$, the two functors $\multtrcon$ and $\multuncon$ form an equivalence of categories:
\begin{equation*}
\hc 1 \multgrbcon G \cong \fusextcon{LG}\text{.}
\end{equation*}
\end{theorem}

\begin{proof}
We consider first the composite $\multtrcon \circ \multuncon$. Let  $\mathcal{L}$ be a central extension of $LG$ equipped with a fusion product $\lambda$, a fusive superficial  connection,  and a multiplicative, contractible path splitting $\kappa$ of its error 1-form $\varepsilon$. We denote by $(\mathcal{G},\rho_{\kappa},\mathcal{M},\alpha)$ the regressed multiplicative bundle gerbe over $G$. Let $\varphi: \trcon\circ \uncon \to \id_{\ufusbunconsf{LG}}$ be the natural transformation that establishes one half  of the fact that $\trcon$ and $\uncon$ form an equivalence of categories. Thus, we have a connection-preserving, fusion-preserving isomorphism $\varphi_{\mathcal{L}}: \trcon_{\mathcal{G}} \to \mathcal{L}$ over $LG$. 
The diagram
\begin{equation*}
\alxydim{@C=3em}{\pr_1^{*}\trcon_{\mathcal{G}} \otimes \pr_2^{*}\trcon_{\mathcal{G}} \ar[rr]^{\trcon_{\mathcal{M}}} \ar[dd]_{\pr_1^{*}\varphi_{\mathcal{L}} \otimes \pr_2^{*}\varphi_{\mathcal{L}}} \ar[dr]_{\trcon_{\uncon(\mu)}} &  & m^{*}\trcon_{\mathcal{G}} \otimes \trcon_{\mathcal{I}_{\rho_{\kappa}}}  \ar[r]^{\id \otimes t_{\rho_{\kappa}}} & m^{*}\trcon_{\mathcal{G}} \otimes \trivlin_{\varepsilon}\ar[dd]^{m^{*}\varphi_{\mathcal{L}} \otimes \id} \\ & m^{*}\trcon_{\mathcal{G}} \otimes \trcon_{\uncon(\trivlin_{\varepsilon})} \ar[ur]_{\id \otimes \trcon_{\mathcal{T}_{\kappa}}}  \ar[d]^{m^{*}\varphi_{\mathcal{L}} \otimes \varphi_{\trivlin_{\varepsilon}}} \\ \pr_1^{*}\mathcal{L} \otimes \pr_2^{*}\mathcal{L} \ar[r]_{\mu} & m^{*}\mathcal{L} \otimes \trivlin_{\varepsilon} \ar@{=}[rr] && m^{*}\mathcal{L} \otimes \trivlin_{\varepsilon}}
\end{equation*}
of bundle isomorphisms over $LG \times LG$ is commutative: the  triangular diagram is the definition of the isomorphism $\mathcal{M}$, the left part commutes because $\varphi$ is natural, and the right part commutes due to Proposition \ref{prop:trivloop} (this uses that $\kappa$ is contractible). The bottom line is  the multiplication of $\mathcal{L}$, and the top line  is by definition the group structure on $\tr_{\mathcal{G}}$. Hence, $\varphi_{\mathcal{L}}$ is connection-preserving, fusion-preserving and a group homomorphism.

Now we look at $\multuncon \circ \multtrcon$. 
Let $(\mathcal{G},\rho,\mathcal{M},\alpha)$ be a multiplicative bundle gerbe with connection over $G$. Let $\mathcal{L}$ be its transgression, with fusion product $\lambda$, multiplication isomorphism $\mu$, error 1-form $\varepsilon = \tau_{S^1}(\rho)$ and path splitting $\kappa = \tau_{[0,1]}(\rho)|_{P(G \times G)}$. Let $\mathcal{A}: \uncon \circ \trcon \to \id_{\ugrbcon G}$ be the natural transformation that establishes the second half of the fact that $\trcon$ and $\uncon$ form an equivalence; thus, $\mathcal{A}_{\mathcal{G}}: \uncon(\mathcal{L}) \to \mathcal{G}$ is a connection-preserving isomorphism.  The diagram
\begin{equation*}
\alxydim{@C=1.5em}{\uncon(\mathcal{L})_1 \otimes \uncon(\mathcal{L})_2 \ar[rr]^{\uncon(\mu)} \ar[dd]_{\pr_1^{*}\mathcal{A}_{\mathcal{G}} \otimes \pr_2^{*}\mathcal{A}_{\mathcal{G}}} \ar[rd]^->>>>>{\uncon(\trcon_{\mathcal{M}})} & &  \uncon(\mathcal{L})_{12} \otimes \uncon(\trivlin_{\varepsilon}) \ar[rr]^-{\id \otimes \mathcal{T}_{\kappa}} && \uncon(\mathcal{L})_{12} \otimes \mathcal{I}_{\rho_\kappa} \ar[dd]^{m^{*}\mathcal{A}_{\mathcal{G}} \otimes \id}\\ & \hspace{-1em}\uncon(\mathcal{L})_{12} \otimes \uncon(\trcon_{\mathcal{I}_{\rho}})\hspace{-1em} \ar[d]^{m^{*}\mathcal{A}_{\mathcal{G}} \otimes \mathcal{A}_{\mathcal{I}_{\rho}}} \ar[ru]^<<<<<<<{\id \otimes \uncon(t_{\rho})}  \\ \mathcal{G}_1 \otimes \mathcal{G}_2 \ar[r] & \mathcal{G}_{12} \otimes \mathcal{I}_{\rho} \ar@{=}[rrr] &&& \mathcal{G}_{12} \otimes \mathcal{I}_{\rho}}
\end{equation*}
in the category $\hc 1 \ugrbcon- (G \times G)$ is commutative: the triangular diagram is the definition of the multiplication $\mu$, the left part is the naturality of $\mathcal{A}$, and the right part is Proposition \ref{prop:trivbase}. This means that there exists a connection-preserving transformation $\beta$ that fills the diagram. Now we are in the situation of Lemma \ref{lem:canbeta}, saying that $\beta$ satisfies the compatibility condition with respect to $\alpha$  and $\alpha'$, and it becomes an isomorphism between multiplicative bundle gerbes with connection over $G$.
\end{proof}

In the commutative diagram of Proposition \ref{prop:multtrans} the maps $\hc 0 \multtrcon$ and $\hc 0 th$ are surjective (by Theorem \ref{th:equivalence} and by definition, respectively), hence the map 
\begin{equation*}
\multtr: \hc 0 \multgrbadcon G \to \hc 0 \fusextth{LG}
\end{equation*}
is surjective.
Thus, we have the next part of  Theorem \ref{th:main}:

\begin{corollary}
\label{co:if}
If $G$ is connected, then every  thin fusion extension of $LG$ is transgressive.
\end{corollary}

In order to complete the proof of Theorem \ref{th:main}, we have to show that the map $\multtr$ is injective and so establishes a bijection.

\begin{proposition}
\label{prop:multun}
There exists a unique map
$\multun: \hc 0 \fusextth{LG} \to \hc 0 \multgrbadcon G$ such that the diagram
\begin{equation*}
\alxydim{@C=4em@R=\xyst}{\hc 0 \fusextcon {LG} \ar[d]_{\hc 0 th} \ar[r]^-{\hc 0 \multuncon} & \hc 0 \multgrbcon G \ar[d] \\ \hc 0 \fusextth{LG} \ar[r]_-{\multun} & \hc 0 \multgrbadcon G}
\end{equation*}
is commutative.
\end{proposition}

\begin{proof}
Uniqueness of the map follows since $\hc 0 th$ is surjective.  In order to define the map $\multun$, we infer that the regression functor $\uncon$ of \erf{eq:regcon} covers a functor $\un_x: \ufusbun {LX} \to \hc 1 \ugrb X$ on the level without connections \cite[Section 5.1]{waldorf10}. Using only the group structure and the multiplicativity  of the fusion product, the bundle gerbe $\un(\mathcal{L})$ can be equipped with a so-called strictly multiplicative structure, which in turns induces a multiplicative structure, see Sections 2 and 5 of \cite{Waldorf}. It remains to prove that the diagram is commutative. We assume that we have a fusive superficial connection $\nu$ on $\mathcal{L}$ together with a multiplicative path splitting $\kappa$.

By construction, $\multuncon(\mathcal{L})$ and $\multun(th(\mathcal{L}))$ have the same underlying bundle gerbe $\mathcal{G}=\un(\mathcal{L})$, and the same underlying isomorphism $\mathcal{M}=\un(\mu)$.
We show that the associators coincide; in order to do so, we prove that the isomorphism $\mathcal{M}$ and the associator $\alpha$ obtained from the strictly multiplicative structure are connection-preserving; thus, by Lemma \ref{lem:canalpha} $\alpha$ equals the associator of $\multuncon(\mathcal{L})$.

The isomorphism $\mathcal{M}$ is induced from the map $r: Y_{1,2} \to Y_{12}$ between the surjective submersions $Y_{1,2}:= P_1G \times P_1G$ of $\mathcal{G}_1\otimes \mathcal{G}_2$ and $Y_{12}:= G \times P_1G$ of $\mathcal{G}_{12}$, and a lift $R: P_{1,2} \to P_{12}$ of the map $r \times r$ to the total spaces of the principal $\ueins$-bundles of these gerbes. Explicitly, $r(\gamma_1,\gamma_2) := (\gamma_1(1),\gamma_1\gamma_2)$, the bundles are $P_{1,2}|_{(\gamma_1,\gamma_2),(\gamma_1',\gamma_2'))}= \mathcal{L}_{\gamma_1\lop\gamma_1'}\otimes \mathcal{L}_{\gamma_2 \lop \gamma_2'}$ and $P_{12}|_{(g,\gamma),(g',\gamma')}=\mathcal{L}_{\gamma\lop\gamma'}$, and $R$ is multiplication. 
An isomorphism induced from maps $(r,R)$ is connection-preserving with respect to the induced connections $\nu_{1,2}$ on $P_{1,2}$ and $\nu_{12}$ on $P_{12}$, if there exists a 1-form $\kappa \in \Omega^1(Y_{1,2})$ such  that  $R^{*}\nu_{12} + \pr_2^{*}\kappa - \pr_1^{*}\kappa = \nu_{1,2}$ over $Y_{1,2} \times_G Y_{1,2}$. 
\begin{comment}
In general, suppose $f: Y_1 \to Y_2$ and $F: P_1 \to P_2$ form a refinement between bundle gerbes. Suppose the bundle gerbes are equipped with connections composed of curvings $B_1$ and $B_2$ and of connections $\omega_1$ and $\omega_2$. If $B_1 = f^{*}B_2$ and $\omega_1= F^{*}\omega_2$ then we obviously obtain a connection-preserving isomorphism. Suppose instead that there exists a 1-form $\kappa\in\Omega^1(Y_1)$ such that $\omega_1 + \pr_2^{*}\kappa - \pr_1^{*}\kappa = F^{*}\omega_2$ and $f^{*}B_2 = B_1+\mathrm{d}\kappa$. Then we use $Q =\trivlin_{\kappa} \otimes (f \times \id)^{*}P_2 $. We have
\begin{equation*}
\alxydim{}{P_1|_{y_1,y_1'} \otimes Q_{y_1',y_2'} \ar[r]^-{F \otimes \id} &P_2|_{f(y_1),f(y_1')} \otimes \trivlin_{\kappa_{y_1}} \otimes \trivlin_{-\kappa_{y_1'}} \otimes \trivlin_{\kappa_{y_1'}} \otimes P_2|_{f(y_1'),y_2'} \ar[r] & \trivlin_{\kappa_{y_1}} \otimes Q_{y_1,y_2} \otimes P_2|_{y_2,y_2'}}
\end{equation*}  
so this is connection-preserving. We have
\begin{equation*}
B_2|_{y_2} - B_1|_{y_1} = \mathrm{d}\kappa_{y_1}+B_2|_{y_2} -B_2|_{f(y_1)}= \mathrm{d}\kappa_{y_1}+ \mathrm{curv}(P_2)_{f(y_1),y_2} = \mathrm{curv}(Q)_{y_1,y_2}\text{.}
\end{equation*}
If the condition for the curvings is not satisfied, one obtains a 2-form $\rho$ on the base space of the gerbes that compensates the error.
\end{comment}
This is exactly the property of the path splitting $\kappa$; hence $\mathcal{M}$ is connection-preserving. 
\begin{comment}
The associativity of the Lie groups $Y$ and $P$ imply  the strict commutativity of the refinements representing the four 1-isomorphisms in the diagram. In this case, the coherence of companions in double categories provides the required 2-isomorphism $\alpha$, and a general coherence result implies the  pentagon axiom. 
\end{comment}
The condition that $\kappa$ is multiplicative then implies that $\alpha$ is connection-preserving. 
\end{proof}

Now, we have two maps $\multtr$ and $\multun$ on the bottom of the commutative diagrams of Propositions \ref{prop:multtrans} and \ref{prop:multun}, covered along surjective maps by  maps $\hc 0 \multtrcon$  and $\hc 0 \multuncon$ that are inverses of each other according to Theorem \ref{th:equivalence}.
This suffices to show the last part of Theorem \ref{th:main}:

\begin{corollary}
\label{co:maintext}
The maps $\multtr$ and $\multun$ are inverses of each other, and establish a bijection
\begin{equation*}
\hc 0 \multgrbadcon G \cong \hc 0 \fusextth{LG}\text{.}
\end{equation*}
\end{corollary}

\setsecnumdepth{1}

\section{Segal-Witten reciprocity}

\label{sec:segalwitten}

Let $G$ be a Lie group and $\mathcal{L}$ be a central extension of $LG$. 
Let $\Sigma$ be a compact oriented surface with boundary components $b_1,...,b_k \subset \partial\Sigma$ parameterized by orientation-preserving diffeomorphisms  $f_i: S^1 \to b_i$. We have induced maps $r_i\maps C^{\infty}(\Sigma,G) \to LG$ defined by $r_i(\phi) := \phi \circ f_i$. We let  $\mathcal{L}_{\Sigma}$ denote the Baer sum of
the central extensions $r_i^{*}\mathcal{L}$ of $C^{\infty}(\Sigma,G)$,
\begin{equation*}
\mathcal{L}_{\Sigma} := r_1^{*}\mathcal{L} \otimes ... \otimes r_k^{*}\mathcal{L}\text{.}
\end{equation*} 
If $\Sigma$  is obtained from two surfaces $\Sigma_1$ and $\Sigma_2$ by gluing along some boundary components, and $\rho_i \maps C^{\infty}(\Sigma,G) \to C^{\infty}(\Sigma_i,G)$ are the restriction maps, then we have an isomorphism
\begin{equation*}
\rho : \mathcal{L}_{\Sigma} \to \rho_1^{*}\mathcal{L}_{\Sigma_1} \otimes \rho_2^{*}\mathcal{L}_{\Sigma_2}\text{.}
\end{equation*}

\begin{definition}
\label{def:smoothrp}
A central extension $\mathcal{L}$ of $LG$ has the \emph{smooth reciprocity property}, if  there exists a family $\{s_{\Sigma}\}$ of splittings $s_{\Sigma}$ of $\mathcal{L}_{\Sigma}$ for every compact oriented surface $\Sigma$, satisfying the gluing law
\begin{equation*}
\rho \circ s_{\Sigma} = \rho_1^{*}s_{\Sigma_1} \otimes \rho_2^{*}s_{\Sigma_2}\text{,}
\end{equation*}
whenever $\Sigma$ is obtained from two surfaces $\Sigma_1$ and $\Sigma_2$ by gluing along some boundary components. 
\end{definition}

By splitting we mean a smooth map $s_{\Sigma}: C^{\infty}(\Sigma,G) \to \mathcal{L}_{\Sigma}$ such that $p \circ s_{\Sigma} = \id_{C^{\infty}(\Sigma,G)}$, where $p$ is the projection $p:\mathcal{L}_{\Sigma} \to C^{\infty}(\Sigma,G)$.

\begin{remark}
\begin{enumerate}[(i)]

\item
Above definition  is derived from  a definition due to Brylinski and McLaughlin \cite{brylinski4} and a gluing law from \cite{Brylinski1996}; in these references the definition  is attributed to Segal \cite{segal1}.   

\item
There is a \emph{complex} version of the reciprocity property, where $G$ is a complex Lie group, $\Sigma$ is a Riemann surface, and $C^{\infty}(\Sigma,G)$ is replaced by $\mathrm{Hol}(\Sigma,G)$, the holomorphic maps from $\Sigma$ to $G$ \cite{Brylinski1996}. 

\item
The smooth reciprocity property is a property of the underlying principal $\ueins$-bundle of $\mathcal{L}$, i.e. the group structure is neglected. For the complex reciprocity property one additionally assumes that the sections $s_{\Sigma}$ are group homomorphisms.

\end{enumerate}
\end{remark}

\noindent
The Segal-Witten reciprocity law states that every transgressive central extension of the loop group of a complex Lie group has the complex reciprocity property.  
The following result is a weaker version  adapted to the \emph{smooth} reciprocity property.

\begin{theorem}
\label{th:recprop}
Every transgressive central extension of the loop group of any Lie group has the  smooth reciprocity property.
\end{theorem}

\begin{proof}
Let $(\mathcal{G},\rho,\mathcal{M},\alpha)$ be a multiplicative bundle gerbe with connection over $G$, and let $\mathcal{L}$ be the corresponding central extension. Suppose $\Sigma$ is a compact oriented surface, $\phi: \Sigma \to G$ is a smooth map, and $\mathcal{T}_i$ is a trivialization of $\phi^{*}\mathcal{G}|_{b_i}$ for every boundary component $b_i$. The \emph{surface holonomy} $\mathcal{A}_{\Sigma}(\phi,\mathcal{T}_1,...,\mathcal{T}_k)\in \ueins$ of $\mathcal{G}$ with boundary conditions $\mathcal{T}_1,...,\mathcal{T}_k$ is defined
in the following way. Choose a trivialization $\mathcal{S}:\phi^{*}\mathcal{G} \to \mathcal{I}_{\omega}$. For each boundary component, we have two trivializations that differ by a $\ueins$-bundle $T_i$ with connection over $b_i$, i.e. $\mathcal{S}|_{b_i} \cong \mathcal{T}_i \otimes T_i$ Then,
\begin{equation*}
\mathcal{A}_{\Sigma}(\phi,\mathcal{T}_1,...,\mathcal{T}_k) := \exp \left (  \int_{\Sigma}\omega \right ) \cdot \prod_{i=1}^{k} \mathrm{Hol}_{T_i}(b_i)^{-1}\text{.}
\end{equation*}
With boundary parameterizations $f_i:S^1 \to b_i$ we have $f_i^{*}\mathcal{T}_i \in (r_i^{*}\mathcal{L})_{\phi}$ and so $(f_1^{*}\mathcal{T}_1,...,f_k^{*}\mathcal{T}_k)\in \mathcal{L}_{\Sigma}|_{\phi}$. Now, a splitting $s_{\Sigma}:C^{\infty}(\Sigma,G) \to \mathcal{L}_{\Sigma}$ is defined by
\begin{equation*}
s_{\Sigma}(\phi) = (f_1^{*}\mathcal{T}_1,...,f_k^{*}\mathcal{T}_k)\cdot \mathcal{A}_{\Sigma}(\phi,\mathcal{T}_1,...,\mathcal{T}_k)\text{.}
\end{equation*}
This splitting satisfies the gluing property since surface holonomy has a more general gluing property, see 
\cite[Proposition 3.1]{carey2} and \cite[Lemma 3.3.3 (c)]{waldorf10}. \end{proof}

\begin{remark}
\begin{enumerate}[(i)]

\item
Theorem \ref{th:recprop} is proved in \cite[Theorem 6.2.1]{brylinski1} and \cite[Theorem 5.9]{brylinski4} for simply-connected Lie groups, and in the latter reference it is claimed that it generalizes to arbitrary Lie groups by a (left out) computation in simplicial Deligne cohomology.

\begin{comment}
\item 
The splitting $s_{\Sigma}$  of Theorem \ref{th:recprop} depends on the connection on the multiplicative bundle gerbe $\mathcal{G}$.
\end{comment}

\item
By Theorem \ref{th:main} the transgressive central extensions are precisely the  thin fusion extension. For a given  thin fusion extension $\mathcal{L}$ the splitting $s_{\Sigma}$ can be constructed directly from a choice of a fusive superficial connection $\nu$ that integrates the given thin structure. Indeed, the surface holonomy $\mathcal{A}_{\Sigma}(\mathcal{T}_1,...,\mathcal{T}_k)$ can be defined directly from $\nu$ and the  fusion product, see \cite[Section 5.3]{waldorf10}.

\end{enumerate}
\end{remark}

\noindent
It is an interesting question whether or not the sections $s_{\Sigma}$  can be chosen  multiplicative, i.e. to be group homomorphisms.
In \cite{brylinski4} it is claimed that this is possible for transgressive central extensions of loop groups of arbitrary Lie groups. In \cite{brylinski3} on pages 2 and 21 Brylinski withdraws that statement, and claims that only the sections for the \emph{complex} reciprocity property can be chosen multiplicative. Based on this claim,  it is proved in  \cite{Brylinski1996}  that the complex reciprocity property  characterizes transgressive central extensions of the loop group of  a connected  semisimple complex Lie group. 
\begin{comment}
In  \cite{Brylinski1996} Brylinski and McLaughlin claim  that for  connected complex semisimple Lie groups every central extension that has the complex reciprocity property is transgressive. In order to prove this, they construct a  gerbe $C$ over $G$ from a central extension. In the definition of a multiplicative structure on $C$ they use that the section $s_{\Sigma}$ is multiplicative. Namely, the claim that a certain lifting gerbe $C_{\Sigma}$ has a global object. That lifting gerbe describes the lifting of the structure group of a certain $C^{\infty}(\Sigma,G)$-bundle $Q_\Sigma$ from $C^{\infty}(\Sigma,G)$ to $\mathcal{L}_{\Sigma}$. But the mere section $s_{\Sigma}$ will not induce such a lifting - only a multiplicative section would.   
\end{comment}

In the present paper we make no claims about complex Lie groups.  In the following we only show via examples that there exist transgressive central extensions for which the sections $s_{\Sigma}$ cannot be chosen multiplicative (Example \ref{ex:nonmultsplitting}), and that there exist central extensions that have the smooth reciprocity property but are not transgressive (Example \ref{ex:nonchar}). For preparation, we need the following.

\begin{proposition}
\label{prop:nonmult}
The splitting $s_{\Sigma}$ constructed in Theorem \ref{th:recprop}  satisfies
\begin{equation*}
s_{\Sigma}(\phi_1\phi_2)= s_{\Sigma}(\phi_1)\cdot s_{\Sigma}(\phi_2) \cdot  \exp 2\pi \im \left (- \int_{\Sigma} (\phi_1,\phi_2)^{*}\rho  \right ) \text{,}
\end{equation*}
where $\rho\in \Omega^2(G \times G)$ is the 2-form of the multiplicative gerbe.  
\end{proposition}

\begin{proof}
Suppose $\phi_1,\phi_2:\Sigma \to G$. We may have chosen trivializations $\mathcal{S}_j: \phi_j^{*}\mathcal{G} \to \mathcal{I}_{\omega_j}$. Then, we get another trivialization $\mathcal{S}$ defined by
\begin{equation*}
\alxydim{@C=5em}{(\phi_1\phi_2)^{*}\mathcal{G} \ar[r]^-{(\phi_1,\phi_2)^{*}\mathcal{M}^{-1}} & \phi_1^{*}\mathcal{G} \otimes \phi_2^{*}\mathcal{G}  \otimes \mathcal{I}_{-(\phi_1,\phi_2)^{*}\rho} \ar[r]^-{\mathcal{S}_1 \otimes \mathcal{S}_2} & \mathcal{I}_{\omega_1+\omega_2-(\phi_1,\phi_2)^{*}\rho}\text{.}}
\end{equation*}
We further have trivializations $\mathcal{T}_{ij}: \phi_j^{*}\mathcal{G}|_{b_i} \to \mathcal{I}_0$ for $j=1,2$ and $i=1,...,k$ and the difference bundles $\mathcal{S}_j|_{b_i} \cong \mathcal{T}_{ij} \otimes T_{ij}$. We consider the trivialization $\mathcal{T}_i$ defined by
\begin{equation*}
\alxydim{@C=5em}{(\phi_1\phi_2)^{*}\mathcal{G}|_{b_i} \ar[r]^-{(\phi_1,\phi_2)^{*}\mathcal{M}^{-1}} & \phi_1^{*}\mathcal{G}|_{b_i} \otimes \phi_2^{*}\mathcal{G}|_{b_i}   \ar[r]^-{\mathcal{T}_{i1} \otimes \mathcal{T}_{i2}} & \mathcal{I}_{0}\text{,}}
\end{equation*}
and obtain
\begin{equation*}
\mathcal{S}|_{b_i} =(\mathcal{S}_1|_{b_i}\otimes \mathcal{S}_2|_{b_i}) \circ (\phi_1,\phi_2)^{*}\mathcal{M}^{-1}|_{b_i} \cong (\mathcal{T}_{i1} \otimes \mathcal{T}_{i2})\circ (\phi_1,\phi_2)^{*}\mathcal{M}^{-1}|_{b_i} \otimes T_{i1} \otimes T_{i2} = \mathcal{T}_i \otimes T_{i1} \otimes T_{i2}\text{.}
\end{equation*}
Thus,
\begin{multline}
\label{eq:shprod}
\mathcal{A}_{\Sigma}(\phi_1\phi_2,\mathcal{T}_1,...,\mathcal{T}_k) = \exp 2\pi \im \left (   \int_{\Sigma}\omega_1+\omega_2-(\phi_1,\phi_2)^{*}\rho \right ) \cdot \prod_{i=1}^{k} \mathrm{Hol}_{T_{i1} \otimes T_{i2}}(b_i)^{-1}
\\
= \mathcal{A}_{\Sigma}(\phi_1,\mathcal{T}_{11},...,\mathcal{T}_{k1}) \cdot \mathcal{A}_{\Sigma}(\phi_1,\mathcal{T}_{12},...,\mathcal{T}_{k2}) \cdot \exp 2\pi \im  \left ( \int_{\Sigma} -(\phi_1,\phi_2)^{*}\rho  \right )\text{.} 
\end{multline}
The product of $f_i^{*}\mathcal{T}_{i1}\in \mathcal{L}_{r_i(\phi_1)}$ and $f_i^{*}\mathcal{T}_{i2}\in \mathcal{L}_{r_i(\phi_2)}$ is $f_i^{*}\mathcal{T}_i$. The claim follows by computing $s_{\Sigma}(\phi_1\phi_2)$ using \erf{eq:shprod}. 
\begin{comment}
Indeed, we get
\begin{eqnarray*}
s_{\Sigma}(\phi_1\phi_2) 
&=& (f_1^{*}\mathcal{T}_1,...,f_k^{*}\mathcal{T}_k)\cdot \mathcal{A}_{\Sigma}(\phi_1\phi_2,\mathcal{T}_1,...,\mathcal{T}_k)
\\&=& (f_1^{*}\mathcal{T}_{11},...,f_k^{*}\mathcal{T}_{k1})\cdot (f_1^{*}\mathcal{T}_{12},...,f_k^{*}\mathcal{T}_{k2})\cdot \mathcal{A}_{\Sigma}(\phi_1,\mathcal{T}_{11},...,\mathcal{T}_{k1}) \\&&\qquad\cdot \mathcal{A}_{\Sigma}(\phi_1,\mathcal{T}_{12},...,\mathcal{T}_{k2}) \cdot \exp \left (  \int_{\Sigma} -(\phi_1,\phi_2)^{*}\rho  \right ) 
 \\&=& s_{\Sigma}(\phi_1)\cdot s_{\Sigma}(\phi_2) \cdot  \exp \left (  \int_{\Sigma} -(\phi_1,\phi_2)^{*}\rho  \right ) \text{.}
\end{eqnarray*}
\end{comment}
\end{proof}

The meaning of the calculation of Proposition \ref{prop:nonmult} is that
\begin{equation*}
\eta: C^{\infty}(\Sigma,G) \times C^{\infty}(\Sigma,G) \to \ueins : (\phi_1,\phi_2) \mapsto \exp 2\pi\im \left (  \int_{\Sigma} (\phi_1,\phi_2)^{*}\rho  \right )  
\end{equation*}
is a classifying 2-cocycle for the (topologically trivializable) central extension $\mathcal{L}_{\Sigma}$. Note that $\eta$  is a coboundary if and only if $\mathcal{L}_{\Sigma}$ has a multiplicative section. 
\begin{comment}
Let us assume $\eta = \Delta\varepsilon$. Since $\eta$ is normalized, $\varepsilon$ must be normalized, too:
$1=\eta(1,1)=\varepsilon(1)\varepsilon(1)\varepsilon(1)^{-1}=\varepsilon(1)\text{.}$
\end{comment}
Also note that if $G$ is abelian and $\eta$ is a coboundary, then $\eta$ is symmetric. 

\begin{example}
\label{ex:nonmultsplitting}
Consider $G=\ueins$ and the central extension $\mathcal{L}=\mathcal{L}_P=\mathcal{L}_{\Z}(-1,0,1)$, i.e. the transgression of the trivial gerbe with multiplicative structure given by the Poincaré bundle $P$ over $T=\ueins\times\ueins$, see Examples \ref{ex:trivial} and \ref{ex:notthinfusion}. By Theorem \ref{th:recprop}, $\mathcal{L}$  has the smooth reciprocity property.  Here, $\rho=\pr_1^{*}\theta\wedge \pr_2^{*}\theta\in\Omega^2(T)$.  It is easy to see that we have $s^{*} \rho = -\rho$, where $s: T \to T:(z,z')\mapsto (z',z)$. This shows that $\eta$ is \emph{skew}-symmetric. But $\eta$ can only be symmetric and skew-symmetric if $\eta =1$. In order to see that this is not the case, we note that 
\begin{equation*}
\int_{T} \rho =1\text{.}
\end{equation*}
\begin{comment}
Indeed, $\rho$ represents a non-trivial element in $\h^2_{dR}(S^1 \times S^1)$ whose pairing with the fundamental class of $S^1 \times S^1$ must be non-trivial. 
\end{comment}
Then, there must be an embedding $\phi: D^2 \to T$ of a disc such that $\int \phi^{*}\rho \notin \Z$. Defining $\phi_1,\phi_2$ by composing with the projections $\pr_1,\pr_2:T\to\ueins$,  we get  $\eta(\phi_1,\phi_2)\neq 0$. This shows that $\eta$ is not symmetric, so that $\mathcal{L}_{D^2}$ has no multiplicative section. Hence,  $\mathcal{L}$ is a transgressive central extension that has the smooth reciprocity property but does not admit multiplicative sections. \end{example}

\begin{example}
\label{ex:nonchar}
We let $G=\ueins$ and $\mathcal{L}_{\R}(\gamma)$ be the basic central extension of $L\ueins$ constructed in Example \ref{ex:notthinfusion} depending on $\gamma\in \R$. It is clear that $\mathcal{L}_{\R}(\gamma)$ has the smooth reciprocity property, because it is topologically trivial. On the other hand, we have seen in Example \ref{ex:notthinfusion} that it cannot be equipped with the structure of a thin fusion extension unless $\gamma\in \Z$. By Proposition \ref{prop:onlyif}, it is hence not transgressive. In other words, the smooth reciprocity property is not sufficient to characterize transgressive central extensions.
\end{example}

The conclusion of this section is that the  reciprocity property (in its original form or in the version of Definition \ref{def:smoothrp}) does not properly characterize transgressive central extensions of non-complex Lie groups. The theory of loop fusion and thin homotopy equivariance that we have developed in this article provides such characterization, valid for all connected Lie groups.

\setsecnumdepth{2}

\begin{appendix}

\setsecnumdepth{1}

\section{Regression of trivial fusion bundles}

\label{app:toptrivreg}

In this appendix we discuss the regression of trivial bundles with trivial fusion products (but non-trivial connections) over the loop space $LX$ of a connected smooth  manifold $X$. For this purpose, we restrict the  constructions of \cite[Sections 5 and 6]{waldorf10}  to that case; this has not yet been worked out explicitly.

For $x\in X$ we consider the diffeological space $P_xX$ of paths in $X$ starting at $x$ with sitting instants, equipped with the subduction  $\ev_1: P_xX \to X:\gamma \mapsto \gamma(1)$ (the diffeological analog of a surjective submersion). Two paths with the same end point compose to a loop via the smooth map $\lop_x: P_xX^{[2]} \to LX: (\gamma_1,\gamma_2) \mapsto \prev{\gamma_2} \pcomp \gamma_1$.

Suppose $\varepsilon \in \Omega^1(LX)$ is a superficial connection on the trivial bundle $\trivlin$ that is fusive with respect to  the trivial fusion product. The regression $\uncon_x(\trivlin_\varepsilon)$ is a bundle gerbe with connection over $X$, composed of the subduction $\ev_1:P_xX \to X$, the principal $S^1$-bundle with connection $\lop^{*}_x\trivlin_{\varepsilon}$, and the identity bundle gerbe product, which is connection-preserving because $\varepsilon$ is fusive. The difficult part is
to specify a curving: a 2-form $B_{\varepsilon} \in \Omega^2(P_xX)$ such that $\pr_2^{*}B_{\varepsilon} -  \pr_1^{*}B_{\varepsilon} = \mathrm{curv}(\lop_x^{*}\trivlin_{\varepsilon}) = \lop_x^{*}\mathrm{d}\varepsilon$.

Such a curving can be constructed because $\varepsilon$ is superficial, see \cite[Section 5.2]{waldorf10}. The construction uses a bijection between the 2-forms on a diffeological space $Y$ and certain smooth maps $\mathcal{B}Y \to \ueins$ on the space $\mathcal{B}Y$ of bigons in $Y$. A \emph{bigon} is a smooth fixed-ends homotopy $\Sigma$ between two paths in $Y$, and the correspondence between a smooth map $G: \mathcal{B}Y \to \ueins$ and a 2-form $B \in \Omega^2(Y)$ is established by the relation
\begin{equation}
\label{eq:bigonform}
G(\Sigma) = \exp 2\pi\im\left (- \int_{\Sigma} B \right )\text{.}
\end{equation}
 Suppose $\Sigma \in \mathcal{B}P_xX$ is a bigon between a path $\gamma_0 \in PP_xX$ and a path $\gamma_1\in PP_xX$. Thus, it is a smooth map $\Sigma:[0,1]^2 \to\ P_xX$ such that $\Sigma(0,t)=\gamma_0(t)$ and $\Sigma(1,t)=\gamma_1(t)$. 
\begin{figure}[h]
\begin{center}
\includegraphics{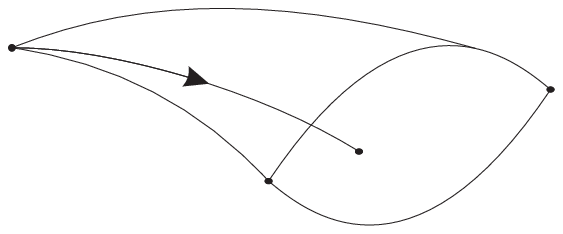}\setlength{\unitlength}{1pt}\begin{picture}(0,0)(188,700)\put(19.15625,745.30063){$x$}\put(142.27625,724.15786){$\Sigma$}\end{picture}
\qquad\quad
\includegraphics{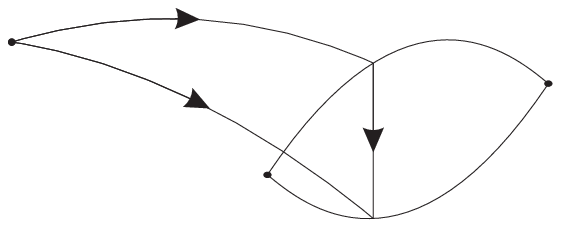}\setlength{\unitlength}{1pt}\begin{picture}(0,0)(187,596)\put(19.15625,640.71043){$x$}\put(73.51909,659.86526){$\Sigma^o(t)$}\put(60.96983,618.04205){$\Sigma^u(t)$}\put(139.24715,616.30951){$\Sigma^m(t)$}\end{picture}
\end{center}
\caption{The picture on the left shows a bigon $\Sigma$ in $\px Xx$: it can be regarded as a bigon in $X$ that has for each of its points  a chosen path connecting $x$ with that point. The picture on the right shows the three paths associated to a bigon $\Sigma$ and $t \in [0,1]$.}
\label{fig:bigon}
\end{figure}
For each $t\in[0,1]$ we extract a loop $\gamma_{\Sigma}(t) \in LX$ defined by
\begin{equation}
\label{def:gammaSigma}
\gamma_{\Sigma}(t) :=  (\Sigma^{m}(t) \pcomp \Sigma^{o}(t)) \lop  (\id \pcomp \Sigma^{u}(t))\text{,}
\end{equation}
where $\Sigma^{m}(t)$, $\Sigma^{o}(t)$, and $\Sigma^{u}(t)$ are the three paths depicted in Figure \ref{fig:bigon}.
Thus, $\gamma_{\Sigma}$ is a path in $LX$ that starts and ends at flat loops. 

Using a smoothing function $\phi:[0,1] \to [0,1]$, we define a parameterized version $\Sigma_{\sigma}$ for $\sigma\in[0,1]$ by $\Sigma_{\sigma}(s,t) := \Sigma(\phi(s)\sigma,t)$, i.e. $\Sigma_0$ is the identity bigon at the path $\gamma_0$, and $\Sigma_1=\Sigma$. Now we consider $h:[0,1]^2 \to LX: (\sigma,t)\mapsto \gamma_{\Sigma_\sigma}(t)$. By \cite[Lemma 5.2.1]{waldorf10} the curving $B_{\varepsilon}$ corresponds to the smooth map
\begin{equation*}
G_{\varepsilon}: \mathcal{B}P_xX \to \ueins : \Sigma \mapsto  \exp 2 \pi \im\left ( - \int_{[0,1]^2} h^{*}\mathrm{curv}(\trivlin_{\varepsilon}) \right )\text{.}
\end{equation*}
As $\mathrm{curv}(\trivlin_{\varepsilon})=\mathrm{d}\varepsilon$ we want to apply Stokes' theorem. Along the boundary of $[0,1]^2$, the map $h$ is as follows:
 $h(0,t) = \gamma_0(t)\lop\gamma_0(t)$,  $h(1,t)=\gamma_{\Sigma}(t)$, $h(\sigma,0)=\gamma_0(0) \lop \gamma_0(0)$, and $h(\sigma,1)=\gamma_0(1)\lop \gamma_0(1)$. As $\varepsilon$ is fusive, we have $\flat^{*}\varepsilon=0$, where $\flat:PX \to LX$ is the inclusion of flat loops, see Section \ref{sec:flatloops}. Thus, we get
\begin{equation*}
G_{\varepsilon}(\Sigma)=\exp 2 \pi \im \left( - \int_{0}^1 \gamma_{\Sigma}^{*}\varepsilon\right )\text{.}
\end{equation*}

The regressed bundle gerbe $\uncon_x(\trivlin_{\varepsilon})$ is not the trivial bundle gerbe (that one would have the identity subduction $\id_X$). It is trivializable, but not canonically trivializable. However, a trivialization $\mathcal{T}_{\kappa}$ can be obtained from a path splitting 
 $\kappa\in \Omega^1(PX)$  of $\varepsilon$, see Definition \ref{def:pathsplitting}. \begin{comment}
Since $\epsilon$ is fusive, it alsways admits a path splitting. In the following we will not distinguish between $\kappa$ and its restriction to $P_xX$.
\end{comment}
The trivialization $\mathcal{T}_{\kappa}$  is composed of the principal $\ueins$-bundle $\trivlin_{-\kappa}$ over $P_xX$ and of the  bundle isomorphism
\begin{equation*}
\id:\lop_x^{*}\trivlin_{\varepsilon} \otimes  \pr_2^{*}\trivlin_{-\kappa}  \to \pr_1^{*}\trivlin_{-\kappa}\text{,}
\end{equation*}
over $P_xX^{[2]}$, which is connection-preserving due to the defining property of a path splitting. There exists a unique 2-form $\rho_{\kappa}\in\Omega^2(X)$ such that $\mathcal{T}_{\kappa}: \uncon_x(\trivlin_{\varepsilon}) \to \mathcal{I}_{\rho_{\kappa}}$ is a connection-preserving isomorphism; this 2-form is characterized by the condition  $\ev_1^{*}\rho_{\kappa} = B_{\varepsilon}-\mathrm{d}\kappa$.

\begin{lemma}
Suppose $\rho\in\Omega^2(X)$ and $\varepsilon := \tau_{S^1}(\rho) \in \Omega^1(LX)$. Then, $\kappa := \tau_{[0,1]}(\rho)$ is a path splitting for $\varepsilon$, and $B_{\varepsilon} = \ev_1^{*}\rho+\mathrm{d}\kappa$. In particular, $\rho_{\kappa}=\rho$.
\end{lemma}

\begin{proof}
That $\kappa$ is a path splitting for $\varepsilon$ has been checked in Example \ref{ex:pathsplitting}. Let $\Sigma$ be a bigon in $P_xX$ between a  path $\gamma_0$ and a path $\gamma_1$.
We have
\begin{equation*}
G_{\varepsilon}(\Sigma)
=\exp 2\pi\im \left(- \int_{0}^1 \gamma_{\Sigma}^{*}\varepsilon\right )
= \exp 2\pi\im \left(  \int_{[0,1] \times S^1}  h_{\gamma_{\Sigma}}^{*}\rho\right )
\end{equation*}
with $h_{\gamma_{\Sigma}}: [0,1] \times S^1 \to X$ defined by $h_{\gamma_{\Sigma}}(t,z) := \gamma_{\Sigma}(t)(z)$. \begin{comment}
Indeeed, if $\varepsilon = \tau_{S^1}(\rho)$, and $\gamma \in PLX$, then with $\gamma^{\vee}: [0,1] \times S^1 \to X$ defined by $\gamma^{\vee}(t,z) := \gamma(t)(z)$ we have
\begin{equation*}
\int_{\gamma}\varepsilon =- \int_{\gamma^{\vee}} \rho\text{.}
\end{equation*}
Indeed, we have
\begin{multline*}
\int_{\gamma}\varepsilon=\int_0^1 \varepsilon_{\gamma(t)}(\partial_t\gamma(t)) \mathrm{d}t =\int_0^1 \int_{0}^1 \rho_{\gamma(t)(z)}(\partial_z\gamma(t)(z),\partial_t\gamma(t)(z)) \mathrm{d}z \mathrm{d}t\\=-\int_0^1 \int_{0}^1 \rho_{\gamma(t)(z)}(\partial_t\gamma(t)(z),\partial_z\gamma(t)(z))  \mathrm{d}z\mathrm{d}t=- \int_{\gamma^{\vee}}\rho\text{.}
\end{multline*}
\end{comment}
We obtain from the definition \erf{def:gammaSigma} of $\gamma_{\Sigma}$:
\begin{equation*}
h_{\gamma_{\Sigma}}(t,z) = \begin{cases}
 \gamma_0(t)(4z) & \text{ if }0 \leq z \leq \frac{1}{4} \\
\ev_1(\Sigma(4z-1)(t)) & \text{ if }\frac{1}{4} \leq z \leq \frac{1}{2} \\
\gamma_1(t)(3-4z) & \text{ if }\frac{1}{2} \leq z \leq \frac{3}{4} \\
x & \text{ if }\frac{3}{4} \leq z \leq 1 \\
\end{cases}
\end{equation*}
Splitting the domain of integration into those four parts and taking care with the involved orientations yields
\begin{eqnarray*}
\exp2\pi\im\left(  \int_{[0,1] \times S^1}  h_{\gamma_{\Sigma}}^{*}\rho\right )
= \exp 2\pi\im \left ( - \int_{\Sigma}\ev_1^{*}\rho + \int_{\gamma_0} \kappa- \int_{\gamma_1} \kappa \right )
\text{.}
\end{eqnarray*}
Finally,  Stokes' theorem gives
\begin{equation*}
\int_{\Sigma}\mathrm{d}\kappa = -\int_{\gamma_0}\kappa + \int_{\gamma_1}\kappa\text{.}
\end{equation*}
All together, we obtain
\begin{equation*}
G_{\varepsilon}(\Sigma)
=  \exp 2\pi \im \left (- \int_{\Sigma}(\ev_1^{*}\rho+\mathrm{d}\kappa) \right )\text{.}
\end{equation*}
Using \erf{eq:bigonform}, we get the claimed equality.
\end{proof}

As regression is a monoidal functor, we want to make sure that the trivialization $\mathcal{T}_{\kappa}$ is compatible with that monoidal structure.

\begin{lemma}
\label{lem:kappaadd}
Suppose $\varepsilon_1,\varepsilon_2$ are superficial connections on the trivial bundle $\trivlin$ over $LX$, and fusive with respect to the trivial fusion product. Suppose  $\kappa_1$ and $\kappa_2$ are path splittings for $\varepsilon_1$ and $\varepsilon_2$, respectively. Then, we have $\rho_{\kappa_1 + \kappa_2}=\rho_{\kappa_1} + \rho_{\kappa_2}$, and there exists a connection-preserving transformation
\begin{equation*}
\alxydim{@R=\xyst@C=5em}{\uncon_x(\trivlin_{\varepsilon_1+\varepsilon_2}) \ar[r]^-{\mathcal{T}_{\kappa_1+\kappa_2}} \ar[d]_{\cong}  &   \mathcal{I}_{\rho_{\kappa_1+\kappa_2}} \ar@{=>}[dl] \ar@{=}[d]  \\ \uncon_x(\trivlin_{\varepsilon_1}) \otimes \uncon_x(\trivlin_{\varepsilon_2}) \ar[r]_-{\mathcal{T}_{\kappa_1} \otimes \mathcal{T}_{\kappa_2}} & \mathcal{I}_{\rho_{\kappa_1}} \otimes \mathcal{I}_{\rho_{\kappa_2}}\text{.}}
\end{equation*} 
\end{lemma}

\begin{proof}
The isomorphism $\uncon_x(\trivlin_{\varepsilon_1+\varepsilon_2}) \cong \uncon_x(\trivlin_{\varepsilon_1}) \otimes \uncon_x(\trivlin_{\varepsilon_2})$ that implements that $\uncon_x$ is monoidal is induced from the connection-preserving, fusion-preserving isomorphism $\id\maps \trivlin_{\varepsilon_1+\varepsilon_2} \to \trivlin_{\varepsilon_1} \otimes \trivlin_{\varepsilon_2}$. In particular, we have $B_{\varepsilon_1+\varepsilon_2}=B_{\varepsilon_1}+B_{\varepsilon_2}$. We calculate
$\ev_1^{*}(\rho_{\kappa_1}+\rho_{\kappa_2}) = B_{\varepsilon_1}-\mathrm{d}\kappa_1+B_{\varepsilon_1}-\mathrm{d}\kappa_1 = B_{\varepsilon_1+\varepsilon_2} + \mathrm{d}(\kappa_1+\kappa_2)$; this shows that $\rho_{\kappa_1+\kappa_2} = \rho_{\kappa_1}+\rho_{\kappa_2}$. The announced connection-preserving transformation is now simply induced by the connection-preserving isomorphism $\id: \trivlin_{-(\kappa_1+\kappa_2)} \to \trivlin_{-\kappa_1} \otimes \trivlin_{-\kappa_2}$.   
\end{proof}

The next two propositions describe the relation between the trivialization $\mathcal{T}_{\kappa}$ of the regressed bundle  gerbe $\uncon_x(\trivlin_{\varepsilon})$, the canonical trivialization $t_{\rho}$ of $\trcon_{\mathcal{I}_{\rho}}$, and the two natural equivalences
\begin{equation*}
\mathcal{A}: \uncon_x \circ \trcon \to \id_{\hc 1 \ugrbcon X}
\quand
\varphi: \trcon \circ \uncon_x \to \id_{\ufusbunconsf X}
\end{equation*}
that establish that the functors $\uncon_x$ and $\trcon$ form an equivalence of categories  \cite[Theorem A]{waldorf10}.

\begin{proposition}
\label{prop:trivbase}
Suppose $\rho\in \Omega^2(X)$. Let $t_{\rho}: \trcon_{\mathcal{I}_{\rho}} \to \trivlin_{\varepsilon}$ be the canonical trivialization, with $\epsilon = \tau_{S^1}(\rho)\in\Omega^1(LX)$. Let $\mathcal{A}_{\mathcal{I}_{\rho}}: \uncon_x(\trcon_{\mathcal{I}_{\rho}}) \to \mathcal{I}_{\rho}$ be the component of the natural equivalence $\mathcal{A}$ at $\mathcal{I}_{\rho}$. Let $\kappa := \tau_{[0,1]}(\rho)$ be the canonical path splitting of $\varepsilon$ and let  $\mathcal{T}_{\kappa}: \uncon_x(\trivlin_{\varepsilon}) \to \mathcal{I}_{\rho}$ be the corresponding trivialization. Then, there exists a connection-preserving transformation
\begin{equation*}
\mathcal{A}_{\mathcal{I}_{\rho}}\cong\mathcal{T}_{\kappa} \circ \uncon_x(t_{\rho})  \text{.}
\end{equation*}
\end{proposition}

\begin{proof}
The isomorphism $\uncon_x(t_{\rho}): \uncon_x(\trcon_{\mathcal{I}_{\rho}}) \to \uncon_x(\trivlin_{\varepsilon})$ is induced from the bundle isomorphism $\lop^{*}t_{\rho}: \lop^{*}\trcon_{\mathcal{I}_{\rho}} \to \trivlin_{\lop^{*}\varepsilon}$ over $P_xX^{[2]}$.  The composition $\mathcal{T}_{\kappa} \circ \uncon_x(t_{\rho})$ is thus  given by the $S^1$-bundle $\trivlin_{-\kappa}$  and the isomorphism
\begin{equation}
\label{eq:isoalpha}
\alxydim{@C=4em}{\lop_x^{*}\trcon_{\mathcal{I}_{\rho}} \otimes \pr_2^{*}\trivlin_{-\kappa} \ar[r]^-{\lop_x^{*}t_{\rho} \otimes \id} & \lop_x^{*}\trivlin_{\varepsilon} \otimes \pr_2^{*}\trivlin_{-\kappa}\ar[r]^-{\id} & \pr_1^{*}\trivlin_{-\kappa}\text{.}}
\end{equation}

Next we describe the connection-preserving isomorphism $\mathcal{A}_{\mathcal{I}_{\rho}}$ following \cite[Section 6.1]{waldorf10}. It consists of an $S^1$-bundle $Q$ over $P_xX$ with connection, and of a connection-preserving bundle isomorphism $\alpha: \lop_x^{*}\trcon_{\mathcal{I}_{\rho}} \otimes \pr_2^{*}Q \to \pr_1^{*}Q$ over $P_xX^{[2]}$. 

The fibre of $Q$ over $\gamma\in P_xX$ consists of triples $(\mathcal{T},t_0,t)$, where $\mathcal{T}:\gamma^{*}\mathcal{I}_{\rho} \to \mathcal{I}_0$ is a trivialization (in turn consisting of an $S^1$-bundle $T$ with connection over $[0,1]$ and of a connection-preserving bundle isomorphism which here is necessarily the identity $\tau=\id_T$), $t_0 \in T_{0}$ and $t \in T_1$. Two triples $(\mathcal{T},t_0,t)$ and $(\mathcal{T}',t_0',t')$ are identified if there exists a connection-preserving transformation $\varphi:\mathcal{T} \Rightarrow \mathcal{T}'$ such that $\varphi(t_0)=t_0'$ and $\varphi(t)=t'$. The $S^1$-action on $S^1$ is $(\mathcal{T},t_0,t)\cdot z := (\mathcal{T},t_0,t\cdot z)$. In our situation, $Q$ has a canonical section  $s:P_xX  \to Q:(\gamma,z)\mapsto (\id_{\mathcal{I}_0},1,1)$, using that $\gamma^{*}\mathcal{I}_{\rho}=\mathcal{I}_{\gamma^{*}\rho}=\mathcal{I}_0$. 
The bundle isomorphism $\alpha$ is  over a point $(\gamma_1,\gamma_2)\in P_xX^{[2]}$ a map
\begin{equation*}
\alpha: \tr_{\mathcal{I}_{\rho}}|_{\gamma_1 \lop \gamma_2} \otimes Q_{\gamma_2}\to Q_{\gamma_1}\text{,}
\end{equation*}  
and it is characterized by $\alpha(t_{\rho}(\gamma_1\lop\gamma_2) \otimes s(\gamma_2))=s(\gamma_1)$. 
The connection on $Q$ is defined via its parallel transport. Using the section $s\maps P_xX \to Q$, it suffices to define a 1-form on $P_xX$, and we do this by defining  a smooth map $F:PP_xX \to S^1$. This map is given by
\begin{equation*}
F(\gamma) = \int_{h_\gamma}\rho\text{,}
\end{equation*}
where $h_\gamma:[0,1]^2 \to X$ is defined by $h_\gamma(s,t)=\gamma(s)(t)$. 
However, this map characterizes precisely the parallel transport of the 1-form $-\kappa = -\tau_{[0,1]}(\rho)\in \Omega^1(P_xX)$. 

Summarizing, $s$ defines a connection-preserving transformation between $\mathcal{A}_{\mathcal{I}_0}$ and the the trivialization consisting of the trivial bundle $\trivlin_{-\kappa}$ and of the isomorphism \erf{eq:isoalpha}. 
\end{proof}

\begin{proposition}
\label{prop:trivloop}
Suppose $\varepsilon \in \Omega^1(LX)$ is a superficial connection on the trivial $\ueins$-bundle over $LX$, and fusive with respect to the trivial fusion product. Let $\varphi_{\trivlin_{\varepsilon}}: \trcon_{\uncon_x(\trivlin_{\varepsilon})} \to \trivlin_{\varepsilon}$ be the component of the natural equivalence $\varphi$ at $\trivlin_{\varepsilon}$. Let $\kappa\in\Omega^1(PX)$ be a contractible path splitting for $\varepsilon$, and let $\mathcal{T}_{\kappa}: \uncon_x(\trivlin_{\varepsilon}) \to \mathcal{I}_{\rho_{\kappa}}$ be the corresponding trivialization. Let $t_{\rho_{\kappa}}: \trcon_{\mathcal{I}_{\rho_{\kappa}}} \to \trivlin_{\varepsilon_{\kappa}}$ be the canonical trivialization
 with $\varepsilon_{\kappa} = \tau_{S^1} (\rho_{\kappa})$. Then,
\begin{equation}
\label{eq:eqmorph}
\varphi_{\trivlin_{\varepsilon}} = t_{\rho_{\kappa}} \circ \trcon_{\mathcal{T}_{\kappa}}\text{;}
\end{equation}
in particular, $\varepsilon_{\kappa}=\varepsilon$. \end{proposition}

\begin{proof}
Note that \erf{eq:eqmorph} is an equality between two connection-preserving bundle isomorphisms going from  $\trcon_{\uncon_x(\trivlin_{\varepsilon})}$ to $\trivlin_{\varepsilon}$ and $\trivlin_{\varepsilon_{\kappa}}$, respectively. This implies $\varepsilon=\varepsilon_{\kappa}$. 
For $\beta\in LG$ a loop, $\beta^{*}\mathcal{T}_{\kappa}$ is a trivialization of $\beta^{*}\uncon_x(\trivlin_{\varepsilon})$, and thus an element of $\trcon_{\uncon_x(\trivlin_{\varepsilon})}$ over $\beta$. We have
\begin{equation*}
\trcon_{\mathcal{T}_{\kappa}}(\beta^{*}\mathcal{T}_{\kappa}) = \beta^{*}\mathcal{T}_{\kappa}  \circ  \beta^{*}\mathcal{T}_{\kappa}^{-1} = \id_{\mathcal{I}_{\beta^{*}\rho_{\kappa}}}\text{,}
\end{equation*}
considered as an element of $\trcon_{\mathcal{I}_{\rho_{\kappa}}}$. Under the canonical trivialization $t_{\rho_{\kappa}}$, this element is equal to $(\beta,1)\in LX\times  \ueins$. 

On the other hand, we compute the element $p := \varphi_{\trivlin_{\varepsilon}}(\beta^{*}\mathcal{T}_{\kappa}) \in \trivlin_{\varepsilon}$ following the definition of $\varphi$ given in \cite[Section 6.2]{waldorf10}. We have to consider the space $Z:=S^1 \lli{\beta}\times_{\ev_1} P_xX$ as the subduction of $\beta^{*} \uncon_x(\trivlin_{\varepsilon})$ and over $Z$ the bundle $\trivlin_{-\kappa}$, pulled back along the projection $Z \to P_xX$. Over $Z \times_{S^1} Z$ the trivialization $\beta^{*}\mathcal{T}_{\kappa}$ has the identity morphism
\begin{equation*}
\id: \lop_x^{*}\trivlin_{\varepsilon} \otimes \pr_2^{*}\trivlin_{-\kappa} \to \pr_1^{*}\trivlin_{-\kappa}\text{,}
\end{equation*}
also pulled back along $Z \times_{S^1} Z \to P_xX^{[2]}$. We represent the loop $\beta$  by a path $\gamma \in P_xX$ with $\gamma(1)=\beta(0)$  and paths $\gamma_k \in PX$ with $\gamma_k(0)=\beta(0)$ and $\gamma_k(1)=\beta(\frac{1}{2})$, related via a thin homotopy $h: (\gamma_1 \pcomp \gamma) \lop (\gamma_2 \pcomp \gamma) \to \beta$. In $Z$ we consider the retracting paths  $\alpha_i$ with $\alpha_i(0)=(0,\id \pcomp \gamma)$ and $\alpha_i(1)=(\frac{1}{2},\gamma_i \pcomp \gamma)$. Then, the prescription is
\begin{equation*}
p =(\beta, \exp 2\pi \im\left( - \int_{\alpha_2 \pcomp \prev{\alpha_1}} \pr^{*}\kappa \right ))=(\beta, \exp 2\pi\im \left(  \int_{\alpha_1} \pr^{*}\kappa- \int_{\alpha_2 } \pr^{*}\kappa \right ))\text{.}
\end{equation*}
Since the paths $\alpha_i$ are retractions and the path splitting $\kappa$ is contractible, both integrals vanish separately.
Thus, we have $p=(\beta,1)$; this yields the claimed equality. 
\end{proof}

\end{appendix}

\newpage

\tocsection{Table of terminology}

\label{sec:tableterm}

\newcommand{\notation}[3]{
  \noindent
  \begin{minipage}[b]{0.30\textwidth}\raggedright#1\end{minipage}
  \begin{minipage}[t]{0.6\textwidth}#2\vspace{0.3cm}\end{minipage}\hfill
  \begin{minipage}[t]{0.08\textwidth}\begin{flushright}#3\end{flushright}\end{minipage}}

\def\subitem{\quad --- }

\notation{Fusion product}{$\lambda: \mathcal{L}_{\gamma_1\lop\gamma_2} \otimes \mathcal{L}_{\gamma_2\lop\gamma_3} \to \mathcal{L}_{\gamma_1\lop\gamma_3}$}{Page \pageref{def:fusion}}
\notation{\subitem multiplicative}{$\lambda$ is a group homomorphism}{Page \pageref{def:fusion}}

\noindent
Thin homotopy equivariant structure

\notation{\strut}{$d_{\tau_0,\tau_1}: \mathcal{L}_{\tau_0} \to \mathcal{L}_{\tau_1}$ for thin homotopic loops $\tau_0,\tau_1$}{Page \pageref{def:thes}}
\notation{\subitem multiplicative}{$d$ is a group homomorphism}{Page \pageref{def:thes}}
\notation{\subitem compatible}{$d$ is connection-preserving}{Page \pageref{def:comp}}
\notation{\subitem symmetrizing}{rotation by an angle of $\pi$ switches factors under fusion}{Page \pageref{def:comp}}
\notation{\subitem fusive}{compatible and symmetrizing}{Page \pageref{def:comp}}

\notation{Bundle morphism}{$\varphi:\mathcal{L} \to \mathcal{L}'$}{}
\notation{\subitem fusion-preserving}{$\varphi(\lambda(p\otimes q))=\lambda'(\varphi(p) \otimes \varphi(q))$}{Page \pageref{par:fuspres}}
\notation{\subitem thin}{$\varphi(d_{\tau_0,\tau_1}(p))=d'_{\tau_0,\tau_1}(\varphi(p))$}{Page \pageref{par:thinbm}}

\notation{Connection}{\strut}{\strut}
\notation{\subitem thin}{induces a thin homotopy equivariant structure}{Page \pageref{def:thin}}
\notation{\subitem superficial}{thin and its holonomy is thin homotopy invariant}{Page \pageref{def:thin}}
\notation{\subitem compatible}{fusion product is connection-preserving}{Page \pageref{def:compatible}}
\notation{\subitem symmetrizing}{induced thin homotopy equivariant structure is symmetrizing}{Page \pageref{def:compatible}}
\notation{\subitem fusive}{compatible and symmetrizing}{Page \pageref{def:compatible}}

\noindent
Path splitting  of $\epsilon \in \Omega^k(LG \times LG )$

\notation{\strut}{$\kappa \in \Omega^k(PG \times PG)$ with $\varepsilon_{\gamma_1\lop \gamma_2,\gamma_1' \lop \gamma_2'} = \kappa_{\gamma_1,\gamma_1'} - \kappa_{\gamma_2,\gamma_2'}$}{Page \pageref{def:pathsplitting}}
\notation{\subitem multiplicative}{$\kappa_{\gamma_1\gamma_1',\gamma_2\gamma_2'} = \kappa_{\gamma_1,\gamma_2}+\kappa_{\gamma_1',\gamma_2'}$}{Page \pageref{par:pathsplittingmult}}
\notation{\subitem contractible}{$\int_{\phi_\gamma} \kappa = 0$, where $\phi_\gamma$ is the contraction of a path $\gamma$}{Page \pageref{par:pathsplittingcon}}

\notation{Thin structure}{thin homotopy equivariant structure induced by  a superficial  connection}{Page \pageref{def:thinstructure}}
\notation{\subitem fusive}{an inducing connection is fusive}{Page \pageref{def:thinstructure}}
\notation{\subitem multiplicative}{the error 1-form of an inducing connection admits a multiplicative and contractible path splitting }{}
\notation{\subitem fusive and multiplicative}{there is an inducing connection with both of above properties}{Page \pageref{def:mfthinstructure}}

\notation{Central extension of $LG$}{\strut}{\strut}
\notation{\subitem thin fusion}{equipped with a multiplicative fusion product and a fusive and multiplicative thin structure}{Page \pageref{def:thinfusionext}}

\newpage

\kobib{F:/uni/bibliothek/tex/}

\end{document}